\documentclass{amsart}

\usepackage{amsmath,amsthm,amstext,amssymb,bbm}
\usepackage{hyperref}
\usepackage{enumerate,enumitem}
\usepackage{dsfont}

\usepackage
 [a4paper, margin=0.9in]{geometry}

\theoremstyle{plain}
\newtheorem{theorem}{Theorem}[section]
\newtheorem{lemma}[theorem]{Lemma}

\newtheorem{corollary}[theorem]{Corollary}
\newtheorem{proposition}[theorem]{Proposition}

\newtheorem*{claim*}{Claim}
\newtheorem*{subclaim*}{Subclaim}
\theoremstyle{definition}
\newtheorem{definition}[theorem]{Definition}
\newtheorem{remark}[theorem]{Remark}

\newtheorem{question}[theorem]{Question}

\newcommand{\betrag}[1]{\vert{#1}\vert}

\newcommand{\dom}[1]{{{\rm{dom}}(#1)}}
\newcommand{\crit}[1]{{{\rm{crit}}\left({#1}\right)}}
\newcommand{\cof}[1]{{{\rm{cof}}(#1)}}
\newcommand{\otp}[1]{{{\rm{otp}}\left(#1\right)}}
\newcommand{\ran}[1]{{{\rm{ran}}(#1)}}

\newcommand{\length}[2]{{\rm{lh}_{#2}}({#1})}
\newcommand{\clo}[1]{{{\rm{Cl}}_{#1}}}
\newcommand{\suborder}[2]{{#1}{\downharpoonright} \hspace{1.0pt}{#2}}
\newcommand{\POT}[1]{{\mathcal{P}}({#1})}

\newcommand{\POTT}[2]{{\mathcal{P}^{{#2}}}({#1})}
\newcommand{\map}[3]{{#1}:{#2}\longrightarrow{#3}}
\newcommand{\Map}[5]{{#1}:{#2}\longrightarrow{#3};~{#4}\longmapsto{#5}}

\newcommand{\Set}[2]{\{{#1}~\vert~{#2}\}}
\newcommand{\seq}[2]{\langle{#1}~\vert~{#2}\rangle}

\newcommand{\goedel}[2]{{\prec}{#1},{#2}{\succ}}

\newcommand{\anf}[1]{{\text{``}\hspace{0.3ex}{#1}\hspace{0.1ex}\text{''}}}

\newcommand{\Poti}[2]{{\mathcal{P}}_{#2}(#1)}
\newcommand{\HH}[1]{{\rm{H}}(#1)}
\newcommand{\Ult}[2]{{\mathrm{Ult}}({#1},{#2})}

\newcommand{\Add}[2]{{\rm{Add}}({#1},{#2})}
\newcommand{\Col}[2]{{\rm{Col}}({#1},{#2})}

\newcommand{\ISP}{{\rm{ISP}}}

\newcommand{\SSP}{{\rm{SSP}}}
\newcommand{\PFA}{{\rm{PFA}}}
\newcommand{\id}{{\rm{id}}}
\newcommand{\Lim}{{\rm{Lim}}}

\newcommand{\Card}{{\rm{Card}}}

\newcommand{\LL}{{\rm{L}}}

\newcommand{\ZFC}{{\rm{ZFC}}}

\newcommand{\GCH}{{\rm{GCH}}}

\newcommand{\CCC}{{\mathbb{C}}}

\newcommand{\PPP}{{\mathbb{P}}}
\newcommand{\QQQ}{{\mathbb{Q}}}

\newcommand{\SSS}{{\mathbb{S}}}
\newcommand{\TTT}{{\mathbb{T}}}
\newcommand{\UUU}{{\mathbb{U}}}

\newcommand{\VV}{{\rm{V}}}

\newcommand{\calR}{\mathcal{R}}
\newcommand{\calS}{\mathcal{S}}
\newcommand{\calT}{\mathcal{T}}
\newcommand{\calU}{\mathcal{U}}

\title[Characterizing large cardinals]{Characterizing large cardinals through Neeman's pure side condition forcing}

\author{Peter Holy}
\address{Mathematisches Institut\\Rheinische Friedrich-Wilhelms-Universit\"at Bonn\\En\-de\-nicher Allee 60\\53115 Bonn\\Germany}
\email{pholy@math.uni-bonn.de}

\author{Philipp L\"ucke}
\address{Mathematisches Institut\\Rheinische Friedrich-Wilhelms-Universit\"at Bonn\\En\-de\-nicher Allee 60\\53115 Bonn\\Germany}
\email{pluecke@math.uni-bonn.de}

\author{Ana Njegomir}
\address{Mathematisches Institut\\Rheinische Friedrich-Wilhelms-Universit\"at Bonn\\En\-de\-nicher Allee 60\\53115 Bonn\\Germany}
\email{njegomir@math.uni-bonn.de}


\begin{document}

\begin{abstract}
We show that some of the most prominent large cardinal notions can be characterized through the validity of certain combinatorial principles at $\omega_2$ in forcing extensions by the \emph{pure side condition forcing} introduced by Neeman. 
The combinatorial properties that we make use of are natural principles, and in particular for inaccessible cardinals, these principles are equivalent to their corresponding large cardinal properties. 
Our characterizations make use of the concepts of \emph{internal large cardinals} introduced in this paper, and of the classical concept of \emph{generic elementary embeddings}. 
\end{abstract}

\maketitle



\section{Introduction}

The interplay between forcing and large cardinals is a major theme of contemporary set theory. 
A typical situation is that a large cardinal $\kappa$ is collapsed to become an accessible cardinal in a way that preserves certain combinatorial properties of $\kappa$.   
This approach establishes a deep connection between large cardinals and combinatorial principles for small cardinals. 
This connection is further emphasized  by various results from inner model theory showing that the validity of certain combinatorial principles causes small cardinals to be large in canonical inner models of set theory. 
In many important cases, these results are able to recover the type of large cardinal used to establish the consistency of the combinatorial principle in the first place, and therefore lead to equiconsistency results between large cardinal axioms and combinatorial principles for small cardinals. 
As an example, results of Solovay show that, if  $\theta$ is a Mahlo cardinal above an uncountable regular cardinal $\kappa$ and $G$ is a filter on the corresponding L\'{e}vy collapse $\Col{\kappa}{{<}\theta}$ that is generic over the ground model $\VV$, then \emph{Jensen's principle $\square_\kappa$} fails in $\VV[G]$. In the other direction, if $\kappa$ is an uncountable cardinal such that $\square_\kappa$ fails, then seminal results of Jensen in \cite{MR0309729} show that $\kappa^+$ is a Mahlo cardinal in G\"odel's constructible universe $\LL$.

 In this paper, we want to examine even stronger  connections between large cardinals and combinatorial principles established by forcing. 
That is, we want to study forcings that \emph{characterize} certain large cardinal properties through the validity of combinatorial principles in their generic extensions, in the sense that an uncountable regular cardinal possesses some large cardinal property if and only if a certain combinatorial principle holds for this cardinal in corresponding generic extensions. 
In addition, our characterizations will be \emph{strong}: They will make use of combinatorial principles that are equivalent to their corresponding large cardinal properties when conjuncted with inaccessibility.  
That is, one might be led to say that the principles we use to characterize large cardinals in the following are their combinatorial remainder after robbing them of their inaccessibility in a sufficiently nice way. 

It is easy to see that not all equiconsistency results necessarily lead to such characterizations. 
For example, in the case of the equiconsistency result for Mahlo cardinals described above, we can combine a  result of Todor\v cevi\'c showing that the \emph{Proper Forcing Axiom $\PFA$} implies that $\square_\kappa$ fails for all uncountable cardinals $\kappa$ (see \cite{MR763902}) with a result of Larson showing that $\PFA$ is preserved by ${<}\omega_2$-closed forcing (see \cite{MR1782117}) to see that, if $\PFA$ holds and $\kappa<\theta$ are regular cardinals greater than $\omega_1$, then $\square_\kappa$ fails in every $\Col{\kappa}{{<}\theta}$-generic extension.

The following definition makes the above-described approach more precise. We will be interested in the case when $\Phi(\theta)$ describes a large cardinal property of $\theta$, and when forcing with $\PPP(\theta)$ turns $\theta$ into a successor cardinal. We use $\Card$ to denote the class of all infinite cardinals.

\begin{definition}\label{definition:Characterization}
 Let $\vec{\PPP}=\seq{\PPP(\theta)}{\theta\in\Card}$ be a class-sequence of partial orders, and let $\Phi(v)$ and $\varphi(v)$ be parameter-free formulas in the language of set theory.  
 \begin{enumerate}[leftmargin=0.7cm]
  \item We say that \emph{$\vec{\PPP}$ characterizes $\Phi$ through $\varphi$} if $$\ZFC\vdash\forall\theta\in\Card ~ [\Phi(\theta) ~ \longleftrightarrow ~ \mathbbm{1}_{\PPP(\theta)}\Vdash\varphi(\check{\theta})].$$

 \item If $\vec{\PPP}$ characterizes $\Phi$ through $\varphi$, then we say that this characterization is \emph{strong} in case that $$\ZFC\vdash\textit{$\forall\theta$ inaccessible} ~ [\Phi(\theta) ~ \longleftrightarrow ~ \varphi(\theta)].$$
 \end{enumerate}
\end{definition}

The following observation shows that the most canonical choice of forcing to turn a large cardinal into a successor cardinal, the  \emph{L\'{e}vy collapse $\Col{\kappa}{{<}\theta}$}, is not suitable for characterizations of the above form. 

\begin{proposition}\label{proposition:ColDoesNotWork}
 Assume that the existence of an inaccessible cardinal is consistent with the axioms of $\ZFC$. If $n<\omega$, then no formula in the language of set theory characterizes the class of inaccessible cardinals through the sequence $\seq{\Col{\omega_n}{{<}\theta}}{\theta\in\Card}$. 
\end{proposition}

\begin{proof}
 Assume, towards a contradiction, that $\varphi(v)$ is a formula with this property, and that there is a model $\VV$ of $\ZFC$ that contains an inaccessible cardinal $\theta$. Let $G$ be $\Col{\omega_n}{{<}\theta}$-generic over $\VV$. Then, using that $\Col{\omega_n}{{<}\theta}\times\Col{\omega_n}{{<}\theta}$ and $\Col{\omega_n}{{<}\theta}$ are forcing equivalent, we may find $H_0,H_1\in\VV[G]$ with the property that $H_0\times H_1$ is $(\Col{\omega_n}{{<}\theta}\times\Col{\omega_n}{{<}\theta})$-generic over $\VV$ and $\VV[G]=\VV[H_0\times H_1]$. Since $\theta$ is inaccessible in $\VV$, our assumption implies that $\varphi(\theta)$ holds in $\VV[G]$ and, since the partial order $\Col{\omega_n}{{<}\theta}^\VV=\Col{\omega_n}{{<}\omega_{n+1}}^{\VV[H_0]}$ is weakly homogeneous in $\VV[H_0]$, we know that $\mathbbm{1}_{\Col{\omega_n}{{<}\theta}}\Vdash\varphi(\check{\theta})$ holds in $\VV[H_0]$. But, again by our assumption, this shows that $\theta=\omega_{n+1}^{\VV[H_0]}$ is inaccessible in $\VV[H_0]$, a contradiction. 
\end{proof}

In contrast, the results of this paper will show  that the \emph{pure side condition forcing} introduced by Neeman in \cite{MR3201836} allows a characterization of many important large cardinal notions through canonical combinatorial principles. 
Given an inaccessible cardinal $\theta$, this forcing notion\footnote{A formal definition of this partial order can be found in Section \ref{section:Neeman}.} uses finite sequences of elementary submodels of $\HH{\theta}$ to turn $\theta$ into $\omega_2$, while preserving many combinatorial properties of $\theta$, due to its strong structural properties. 
For example, Neeman's pure side condition forcing is strongly proper for a rich class of models, and its quotient forcings have the $\sigma$-approximation property. 
In \cite{MR3201836}, Neeman already used these properties to prove results that imply that an inaccessible cardinal is weakly compact if and only if the corresponding pure side condition forcing turns the cardinal into $\omega_2$ and the cardinal $\omega_2$ has the tree property in the corresponding generic extension (see {\cite[Section 5.1]{MR3201836}}). 
 %
In the first part of this paper, we will make use of the concept of \emph{internal large cardinals} that we will introduce in Section \ref{section:INternalAndMahlo}, in order to strongly characterize the following large cardinal properties with the help of Neeman's pure side condition forcing: 
 \begin{itemize}[leftmargin=0.7cm]
  \item Inaccessible and Mahlo cardinals.

  \item $\Pi^m_n$-indescribable cardinals for all $0<{m,n}<\omega$. 

  \item Subtle, ineffable and $\lambda$-ineffable cardinals for certain $\lambda$. 
  
  \item Supercompact cardinals. 
 \end{itemize}
In its second part, we will use the classical concept of \emph{generic elementary embeddings} to provide strong characterizations for the following large cardinal properties: 
 \begin{itemize}[leftmargin=0.7cm]
  \item Measurable cardinals.
  
  \item $\gamma$-supercompact cardinals for $\gamma\ge\kappa$.

  \item Supercompact cardinals. 

  \item Almost huge and super almost huge cardinals. 
 \end{itemize}
 In the final section of this paper, we will discuss possible variations of the results of this paper and state some open questions motivated by them. In particular, we will discuss the question whether similar characterizations are possible using other canonical sequences of collapse forcing notions.


\section{Neeman's pure side condition forcing}\label{section:Neeman}

Closely following \cite[Section 2]{MR3201836} for its most parts, this section is devoted to introducing (instances of) Neeman's model sequence poset. We will present a series of definitions and results from \cite{MR3201836}, which will be needed in the later sections of our paper, and we also make some additional observations that will be important for our later results.

\begin{definition}
 Given a transitive set $K$ and $\calS,\calT\subseteq K$, we let $\PPP_{K,\calS,\calT}$ denote the partial order defined by the following clauses: 
   \begin{enumerate}[leftmargin=0.7cm]
    \item A condition in $\PPP_{K,\calS,\calT}$ is a finite, $\in$-increasing sequence $\seq{M_i}{i<n}$ of elements of $\calS\cup\calT$ with the property that for all $i,j<n$, there is a $k<n$ with $M_k=M_i\cap M_j$. 
    
    \item Given conditions $p$ and $q$ in $\PPP_{K,\calS,\calT}$, we let $p\leq_{\PPP_{K,\calS,\calT}}q$ if and only if $\ran{p}\supseteq \ran{q}$. 
   \end{enumerate}
\end{definition}

In the setting of this definition, we will refer to elements of $\calS\cup\calT$ as \emph{nodes in $K$}.

\begin{remark}
Let $p=\seq{M_i}{i<n}$ be a condition in a partial order of the form $\PPP_{K,\calS,\calT}$. Then for all $j<k<n$, the node $M_j$ has a smaller rank than the node $M_{k}$. In particular, the ordering of $\ran{p}$ imposed by $p$ is uniquely determined by $\ran{p}$. Thus we can identify conditions in $\PPP_{K,\calS,\calT}$ with their range. We will use this identification tacitly throughout.
\end{remark}

In most parts of this paper, we will work with a special case of the above general definition, as provided by the following. We will however make use of the general definition in the second part of this paper, when we investigate (in $\VV$) forcing notions of the form $\PPP_{K,\calS,\calT}$ as defined in inner models. 

\begin{definition}
 Let $\theta$ be an infinite cardinal. 
 \begin{enumerate}[leftmargin=0.7cm]
  \item We let $\calS_\theta$ denote the set of all countable elementary submodels of $\HH{\theta}$ that are elements of $\HH{\theta}$. 
  
  \item We let $\calT_\theta$ denote the set of all transitive and countably closed elementary submodels of $\HH{\theta}$ that are elements of $\HH{\theta}$. 
  
  \item We set $\PPP_\theta=\PPP_{\HH{\theta},\hspace{0.9pt}{} \calS_\theta,\calT_\theta}$. 
 \end{enumerate} 
\end{definition}

The aim of this paper is to show that the sequence $\seq{\PPP_\theta}{\theta\in\Card}$ can be used to strongly characterize various classes of large cardinals in the sense of Definition \ref{definition:Characterization}. 
Towards this goal, we first review some of the strong structural properties of partial orders of the form $\PPP_\theta$ that were proven in \cite{MR3201836}. The assumptions on the sets $K$, $\calS$ and $\calT$ listed in the following definition are already sufficient in order to derive several such properties.

\begin{definition}\label{definition:SuitableAppropriate}
 Let $K$ be a transitive set and let $\calS,\calT\subseteq K$. 
 \begin{enumerate}[leftmargin=0.7cm]
  \item We say that $K$ is \emph{suitable} if $\omega_1\in K$ and the model $(K,\in)$ satisfies a sufficient fragment of $\ZFC$, in the sense that $K$ is closed under the operations of pairing, union, intersection, set difference, cartesian product, and transitive closure, closed under the range and restriction operations on functions, and such that for each $x\in K$, the closure of $x$ under intersections belongs to $K$, and there is a ordinal length sequence in $K$ consisting of the members of $x$ arranged in non-decreasing von Neumann rank. 
   
   \item If $K$ is suitable, then the pair $(\calS,\calT)$ is \emph{appropriate for $K$} if the following statements hold:
    \begin{enumerate}[leftmargin=0.7cm]
     \item Elements of $\calS$ are countable elementary submodels of $K$, and elements of $\calT$ are transitive and countably closed elementary submodels of $K$.  
     
     \item If $M\in\calS$ and $W\in M\cap\calT$, then $M\cap W \in W$ and $M\cap W\in\calS$. 
  \end{enumerate} 
   \end{enumerate}
\end{definition}

If a pair $(\calS,\calT)$ is appropriate for a suitable set $K$, then we will refer to elements of $\calS$ as \emph{small nodes in $K$}, and to elements of $\calT$ as \emph{transitive nodes in $K$}.

If an infinite cardinal $\theta$ satisfies certain cardinal arithmetic assumptions, then the corresponding partial order $\PPP_\theta$ fits into the above framework. Let us say that a cardinal $\theta$ is \emph{countably inaccessible} if and only if it is regular and $\delta^\omega <\theta$ holds for all $\delta < \theta$.

\begin{proposition}\label{proposition:appropriate}
 Let $\theta$ be an infinite cardinal. 
 \begin{enumerate}[leftmargin=0.7cm]
  \item If $\theta>\omega_1$, then $\HH{\theta}$ is suitable. 
  
  \item If $\theta$ is countably inaccessible, then $(\calS_\theta,\calT_\theta)$ is appropriate for $\HH{\theta}$.
 \end{enumerate}
\end{proposition}

\begin{proof}
 Only (ii) is non-trivial. Pick $M\in\calS_\theta$ and $W\in M\cap\calT_\theta$. Since $W$ is countably closed and $M$ is countable, we know that $W\cap M\in W$. Moreover, since there is a well-ordering of $W$ in $M$, we may define Skolem functions for $W$ in $M$. By this and the fact that $M$ is an elementary submodel of $\HH{\theta}$, we may conclude that  $M\cap W $ is an elementary submodel of both $W$ and $\HH{\theta}$. Since $M$ is countable, we can conclude that $M\cap W$ is an element of $\calS_\theta$.
\end{proof}

 The following result will frequently be used in our computations.

 \begin{lemma}[{\cite[Corollary 2.32]{MR3201836}}]\label{coolone}
   Let $K$ be a suitable set, let $(\calS,\calT)$ be appropriate for $K$, let $M\in\calS\cup\calT$ and let $p\in\PPP_{K,\calS,\calT}\cap M$. Then there is a condition $q$ below $p$ in $\PPP_{K,\calS,\calT}$ with $M\in q$. Moreover, $q$ can be taken as the closure of $p\cup\{M\}$ under intersections.
\end{lemma}

\begin{corollary}\label{corollary:aboutgeneric}
 If $\theta$ is a countably inaccessible cardinal  and $G$ is $\PPP_\theta$-generic over $\VV$, then $$\HH{\theta}^\VV ~ = ~ \bigcup\bigcup G ~ = ~ \bigcup(\calT_\theta^\VV\cap\bigcup G).$$  
\end{corollary}

\begin{proof}
  Since  our assumptions on $\theta$ imply that every subset of $\HH{\theta}^\VV$ of cardinality less than $\theta$ in $\VV$ is a subset of some $M\in\calT_\theta^\VV$, the above equalities follow directly from Lemma \ref{coolone}.  
\end{proof}

The first structural consequence of the assumptions listed in Definition \ref{definition:SuitableAppropriate}, which will be used throughout this paper is that of \emph{strong properness}, which was introduced by Mitchell in \cite{MR2162106}. Remember that, given a partial order $\PPP$ and a set $M$, a condition $p$ in $\PPP$ is a \emph{strong master condition for $M$} if $G\cap M$ is $(\PPP \cap M)$-generic over $\VV$, whenever $G$ is $\PPP$-generic over $\VV$ with $p\in G$. A partial order $\PPP$ is \emph{strongly proper} for some set $M$ if every condition in $\PPP\cap M$ can be extended to a strong master condition for $M$.  
Note that if $M$ is sufficiently elementary in some transitive set, then any strong master condition for $M$ is also a \emph{master condition for $M$}, i.e.\ it forces the generic object to intersect every dense set $D\in M$ of $\PPP$ inside $M$. The following result from \cite{MR3201836} shows that the above assumptions assure that partial orders of the form $\PPP_{K,\calS,\calT}$ are strongly proper for a large class of models.

\begin{lemma}[{\cite[Claim 4.1]{MR3201836}}]\label{strong properness}
 Let $K$ be a suitable set and let $(\calS,\calT)$ be appropriate for $K$. 
 \begin{enumerate}[leftmargin=0.7cm]
  \item If $p$ is a condition in $\PPP_{K,\calS,\calT}$ and $M$ is a node in $p$, then $p$ is a strong master condition for $M$. 
  
  \item The partial order $\PPP_{K,\calS,\calT}$ is strongly proper for every element of $\calS\cup\calT$. 
  
  \item If $W$ is a node in $K$, then the partial order $\PPP_{K,\calS,\calT}\cap W$ is strongly proper for every element of $(\calS\cup\calT)\cap W$. Moreover, if $p$ is a condition in $\PPP_{K,\calS,\calT}\cap W$ and $M$ is a node in $p$, then $p$ is a strong master condition for $M$ with respect to the partial order $\PPP_{K,\calS,\calT}\cap W$.
 \end{enumerate}
\end{lemma}

The following corollary follows directly from the above lemma and the observation that, given an uncountable regular cardinal $\theta$, a partial order $\PPP\subseteq\HH{\theta}$ that is strongly proper for every element of $\calS_\theta$ is also proper.

\begin{corollary}\label{corollary:Proper}
 If $\theta$ is a countably inaccessible cardinal, then the partial order $\PPP_\theta$ is proper 
 and therefore forcing with $\PPP_\theta$ preserves $\omega_1$. \qed
\end{corollary}

Closely following \cite[Sections 3 \& 4]{MR3201836}, we list some standard consequences of properness and strong properness.

\begin{lemma}[{\cite[Claim 3.3]{MR3201836}}]\label{name}
 Let $\vartheta$ be a sufficiently large regular cardinal, let  $M$ be an elementary submodel of  $\HH{\vartheta}$ and let $\PPP$ be a partial order in $M$. If $G$ is $\PPP$-generic over $\VV$ and $G$ contains a master condition for $M$, then the following statements hold true: 
 \begin{enumerate}[leftmargin=0.7cm]
   \item $M[G]$ is an elementary submodel of $\HH{\vartheta}[G]$ with $M[G]\cap V=M$. 
   
   \item If $\dot{f}$ is a $\PPP$-name in $M$ with the property that $\dot{f}^{G}$ is a function with ordinal domain and $\tau=\dot{f}\cap M$, then $\tau^{G}=\dot{f}^{G}\restriction M$.
\end{enumerate}
\end{lemma}

Remember that, given a non-empty set $X$, we say that $S\subseteq\POT{X}$ is a \emph{stationary subset of $\POT{X}$} if it meets every set of the form $C_f=\Set{x\subseteq X}{f[[x]^{{<}\omega}]\subseteq x}$, where $\map{f}{[X]^{{<}\omega}}{X}$ is a function sending finite subsets of $X$ to elements of $X$.

\begin{lemma}[{\cite[Claim 3.5]{MR3201836}}]\label{notcoll}
 Let $K$ be suitable, let $\PPP\subseteq K$ be a partial order and let $\calR$ be a set with the property that $\PPP$ is strongly proper for every element of $\calR$. If $\kappa$ is an uncountable cardinal with the property that for all $\alpha<\kappa$, the set $\Set{M\in \calR}{\alpha\subseteq M, ~ \betrag{M}<\kappa}$ is a stationary subset of $\POT{K}$, then forcing with $\PPP$ does not collapse $\kappa$. 
\end{lemma}

The next proposition (see {\cite[Section 5.1]{MR3201836}}) shows how the above lemma can be applied in our context. Since no proof of this statement was provided in \cite{MR3201836}, we will lay out the easy argument for the benefit of our readers.

\begin{proposition}\label{stat_trans_nodes}
 Let $\theta$ be a countably inaccessible cardinal. 
 \begin{enumerate}[leftmargin=0.7cm]
  \item If $\alpha<\theta$, then the set $\Set{M\in\calT_\theta}{\alpha\subseteq M}$ is a stationary subset of $\POT{\HH{\theta}}$. 
  
  \item Forcing with $\PPP_\theta$ preserves $\theta$. 
 \end{enumerate}
\end{proposition}

\begin{proof}
 Fix a function $\map{f}{[\HH{\theta}]^{{<}\omega}}{\HH{\theta}}$ and let $\vartheta>\theta $ be a sufficiently large regular cardinal. 
Using the countable inaccessibility of $\theta$, we construct a continuous chain $\seq{N_{\gamma}}{\gamma\leq\omega_1}$ of elementary submodels of $\HH{\vartheta}$ of cardinality less than $\theta$ such that $\alpha+1 \cup \{f,\theta\}\subseteq N_0$, $N_{\gamma}\cap \theta\in \theta$ and ${}^{\omega} N_{\gamma} \subseteq N_{\gamma+1}$ for all $\gamma<\omega_1$. Then $N_{\omega_1}\cap \HH{\theta} \in\calT_\theta$ and $f[[N_{\omega_1}\cap \HH{\theta}]^{{<}\omega}]\subseteq N_{\omega_1}\cap \HH{\theta}$. The second statement now follows from a combination of Proposition \ref{proposition:appropriate}, Lemma \ref{strong properness},(ii) and Lemma \ref{notcoll}. 
\end{proof}

We will make heavy use of the following property introduced by Hamkins (see \cite{MR2063629}).

\begin{definition}\label{definition:appcover}
  Given transitive classes $M\subseteq N$, the pair $(M,N)$ satisfies the \emph{$\sigma$-approxi\-ma\-tion property} if $A\in M$ whenever $A\in N$ is such that $A\subseteq B$ for some $B\in M$, and $A\cap x\in M$ for every $x\in\Poti{B}{\omega_1}^M$.\footnote{In case $M$ and $N$ have the same ordinals and satisfy enough set theory, this definition is equivalent to the more common definition of the $\sigma$-approximation property where rather than requiring $A\subseteq B$ for some $B\in M$, one only requires that $A\subseteq M$.} Moreover, we say that a partial order $\PPP$ has the $\sigma$-approximation property if the pair $(\VV,\VV[G])$ satisfies the $\sigma$-approximation property whenever $G$ is $\PPP$-generic over $\VV$. 
\end{definition}

The next lemma follows from a slight modification of the proof of {\cite[Lemma 3.6]{MR3201836}}.  For the sake of completeness, we present a proof of this statement. 

\begin{lemma}\label{lemma:SigmaApproximation}
 Let $\PPP$ be a partial order and let $\vartheta$ be a sufficiently large regular cardinal. If there is an unbounded subset $U$ of $\Poti{\HH{\vartheta}}{\omega_1}$ consisting of elementary submodels of $\HH{\vartheta}$ with the property that $\PPP$ is strongly proper for all elements of $U$, then the partial order $\PPP$ satisfies the $\sigma$-approximation property.  
\end{lemma}

\begin{proof} 
 Let $\dot{f}$ be a nice $\PPP$-name  for a function from an ordinal $\alpha$ to $2$ and let $p$ be a condition  in $\PPP$ with the property that $p$ forces all 
restrictions of $\dot{f}$ to countable subsets of $\alpha$ from the ground model to be elements of the ground model. By our assumption, there is a countable $M\in U$ with  $\PPP,\dot{f}, p\in M$, and there is a strong master condition $q$ for $M$ below $p$. Let $G$ be a $\PPP$-generic filter over $\VV$ with $q\in G$. If we set $\tau=\dot{f}\cap M$, then Lemma \ref{name},(ii) and our assumptions on $\dot{f}$ imply that $\tau^G=\dot{f}^G\restriction M\in\VV$. Furthermore, $G\cap M $ is $(\PPP\cap M)$-generic over $\VV$ and $\tau$ is a $(\PPP\cap M)$-name with $\tau^{G\cap M}=\tau^G\in\VV$. Thus, there exists a condition $s$ in $G\cap M$ such that $s\leq_\PPP p$ and $s\Vdash^\VV_{\PPP\cap M}\anf{\tau=\check{g}}$ for some function $\map{g}{M\cap\alpha}{2}$ in $\VV$. 

\begin{claim*}
 If $\beta\in M\cap\alpha$, then  $s\Vdash^M_\PPP\anf{\dot{f}(\check{\beta})=g(\beta)}$. 
\end{claim*}

\begin{proof}[Proof of the Claim]
 If not, then we may find $\beta\in M\cap\alpha$ and $t\in\PPP\cap M$ with the property that  $t\leq_\PPP s$ and $t\Vdash^M_\PPP\anf{\dot{f}(\check{\beta})\neq g(\beta)}$. Pick a strong master condition $u$ for $M$ with $u\leq_\PPP t$ and let $H$ be $\PPP$-generic over $\VV$ with $u\in H$. Then $H\cap M$ is $(\PPP\cap M)$-generic over $\VV$ and $s\in H$ implies $(\beta,g(i))\in\tau^{H\cap M}\subseteq\dot{f}^H$. By elementarity, we also know that $t\in H$ implies that $\dot{f}^H(\beta)\neq g(\beta)$, a contradiction. 
\end{proof}

 By elementarity, the above claim shows that for every $\beta<\alpha$, there is an $i<2$ with $s\Vdash^\VV_\PPP\anf{\dot{f}(\check{\beta})=i}$. Hence $s$ is a condition in $\PPP$ below $p$ that forces $\dot{f}$ to be an element of the ground model. It is now easy to see that these computations yield the statement of the lemma. 
\end{proof}

\begin{corollary}\label{corollary:SideConditionForcingApproximationProperty}
  Let $K$ be a suitable set and let $(\calS,\calT)$ be  appropriate for $K$. If $\calS$ is a stationary subset of $\POT{K}$, then the partial order $\PPP_{K,\calS,\calT}$ has the $\sigma$-approximation property. In particular, if $\theta$ is a countably inaccessible cardinal, then the partial order $\PPP_\theta$ has the $\sigma$-approximation property. 
\end{corollary}

\begin{proof}
 Let $\vartheta$ be a sufficiently large regular cardinal and let $U$ denote the collection of all countable elementary submodels $M$ of $\HH{\vartheta}$ with $M\cap K\in\calS$. Then $U$ is a stationary subset of $\POT{\HH{\vartheta}}$ and, since it consists of countable sets, it is unbounded in $\Poti{\HH{\vartheta}}{\omega_1}$. Pick $M\in U$. Then Lemma \ref{strong properness},(ii) implies that $\PPP_{K,\calS,\calT}$ is strongly proper for $M\cap K$. By our assumptions on $K$, the domain of $\PPP_{K,\calS,\calT}$ is a subset of $K$, and hence $\PPP_{K,\calS,\calT}\cap M=\PPP_{K,\calS,\calT}\cap M\cap K$. This shows that $\PPP_{K,\calS,\calT}$ is also strongly proper for $M$. In this situation, Lemma \ref{lemma:SigmaApproximation} directly implies  the first conclusion of the corollary. The second conclusion follows directly from Proposition \ref{proposition:appropriate} together with the fact that $\calS_\theta$ is a club subset of $\Poti{\HH{\theta}}{\omega_1}$, and that therefore it is a stationary subset of $\POT{\HH{\theta}}$.  
\end{proof}

In the following, we present results showing that for countably inaccessible cardinals, the factor forcings of $\PPP_\theta$ induced by transitive nodes in $\HH{\theta}$ also possess the $\sigma$-approximation property.

\begin{definition}
 Let $K$ be a transitive set, let $\calS,\calT\subseteq K$, and let $M\in\calS\cup\calT$. We let $\dot{\QQQ}^M_{K,\calS,\calT}$ denote the canonical $(\PPP_{K,\calS,\calT}\cap M)$-nice name for a suborder of $\PPP_{K,\calS,\calT}$ with the property that whenever $G$ is $(\PPP_{K,\calS,\calT}\cap M)$-generic over $\VV$, then $(\dot{\QQQ}^M_{K,\calS,\calT})^G$ consists of all conditions $p$ in $\PPP_{K,\calS,\calT}$ with $M\in p$ and $p\cap M\in G$.   
\end{definition}

Given a partial order $\PPP$ and a condition $p$, we let $\suborder{\PPP}{p}$ denote the suborder of $\PPP$ consisting of all conditions below $p$. The first part of the following lemma is a consequence of {\cite[Corollary 2.31]{MR3201836}}. The second part then follows from the first part together with Lemma \ref{strong properness},(i).

\begin{lemma}\label{lemma:forcing_equivalent}
 Let $K$ be a suitable set, let $(\calS,\calT)$ be  appropriate for $K$, and let $M\in\calS\cup\calT$. Then the map $$\Map{D^M_{K,\calS,\calT}}{\suborder{\PPP_{K,\calS,\calT}}{\langle M\rangle}}{(\PPP_{K,\calS,\calT}\cap M)*\dot{\QQQ}^M_{K,\calS,\calT}}{p}{(p\cap M,\check{p})}$$ is a dense embedding.  Moreover, if $G$ is $\PPP_{K,\calS,\calT}$-generic over $\VV$ with  $M\in\bigcup G$, then $\VV[G\cap M]$ is a $(\PPP_{K,\calS,\calT}\cap M)$-generic extension of $\VV$ and $\VV[G]$ is a $(\dot{\QQQ}^M_{K,\calS,\calT})^{G\cap M}$-generic extension of $\VV[G\cap M]$. 
\end{lemma}

\begin{lemma}[{\cite[Claim 4.3 \& 4.4]{MR3201836}}]\label{strproper}
  Let $K$ be a suitable set, let $(\calS,\calT)$ be  appropriate for $K$, let $W\in\calT$ and let $G$ be $(\PPP_{K,\calS,\calT}\cap W)$-generic over $\VV$. Define $$\hat{\calS} ~ = ~ \Set{M\in \calS} {W\in M, ~ M\cap W\in \bigcup G}.$$ If $\calS$ is a stationary subset of $\POT{K}$ in $\VV$, then $\hat{\calS}$ is a stationary subset of $\POT{K}$ in $\VV[G]$ and the partial order $(\dot{\QQQ}^W_{K,\calS,\calT})^G$ is strongly proper in $\VV[G]$ for every element of $\hat{\calS}$.   
\end{lemma}

Given an infinite cardinal $\theta$ and $M\in\calS_\theta\cup\calT_\theta$, we write $\dot{\QQQ}^M_\theta$ instead of $\dot{\QQQ}^M_{\HH{\theta},\calS_\theta,\calT_\theta}$ and $D^M_\theta$ instead of $D^M_{\HH{\theta},\calS_\theta,\calT_\theta}$. Using Lemma \ref{lemma:SigmaApproximation}, a small variation of the proof of Corollary \ref{corollary:SideConditionForcingApproximationProperty} yields the following result.

\begin{corollary}\label{corollary:sigma_approximation}
 Let $K$ be a suitable set, let the pair $(\calS,\calT)$ be  appropriate for $K$, let $W\in\calT$ and let $G$ be $(\PPP_{K,\calS,\calT}\cap W)$-generic over $\VV$. If $\calS$ is a stationary subset of $\POT{K}$ in $\VV$, then the partial order $(\dot{\QQQ}^W_{K,\calS,\calT})^G$ satisfies the $\sigma$-approximation property in $\VV[G]$. In particular, if $\theta$ is a countably inaccessible cardinal, $W\in\calT_\theta$ and $G$ is $(\PPP_\theta\cap W)$-generic over $\VV$, then the partial order $(\dot{\QQQ}^W_\theta)^G$ has the $\sigma$-approximation property in $\VV[G]$.  \qed
\end{corollary}

Our next goal is to show that for inaccessible cardinals $\theta$, partial orders of the form $\PPP_\theta$ satisfy the $\theta$-chain condition. This result is not mentioned in \cite{MR3201836}, and we will need it for the characterization of ineffable cardinals in Section \ref{section:ineffable}.  
The following result is due to Neeman (see {\cite[Claim 5.7] {MR3201836}}) in a slightly different setting, however with exactly the same proof also working in our setting.

\begin{lemma}\label{trans_in_generic}
 Let $\theta$ be an inaccessible cardinal 
  and let $\kappa<\theta$ be a cardinal with $\HH{\kappa}\in\calT_\theta$. Then the suborder $\suborder{\PPP_\theta}{\langle\HH{\kappa}\rangle}$ is dense in $\PPP_\theta$. 
\end{lemma}

\begin{proposition}\label{proposition:complete subforcing}
 Let $\theta$ be an inaccessible cardinal. If $\kappa<\theta$ is a cardinal with $\HH{\kappa}\in\calT_\theta$, then $\PPP_\theta\cap\HH{\kappa}$ is a complete subforcing of $\PPP_\theta$. 
\end{proposition}

\begin{proof}
Let $A$ be a maximal antichain of  $\PPP_\theta\cap\HH{\kappa}$. Let $G$ be a filter on $\PPP_\theta$ that is generic over $\VV$. Then Lemma \ref{trans_in_generic} implies that $\langle\HH{\kappa}\rangle\in G$. Since  Lemma \ref{strong properness}, (i) shows that $\langle\HH{\kappa}\rangle$ is a strong master condition for $\HH{\kappa}$, we know that $G\cap\HH{\kappa}$ is $(\PPP_\theta\cap\HH{\kappa})$-generic over $\VV$, and hence it intersects  $A$, showing that $A$ is a maximal antichain in $\PPP_\theta$, as desired. 
\end{proof}

\begin{lemma}\label{chain_condition}
 If $\theta$ is inaccessible, then $\PPP_\theta$ satisfies the $\theta$-chain condition. 
\end{lemma}

\begin{proof}
  Fix a maximal antichain $A$ in $\PPP_\theta$ and pick a sufficiently large regular cardinal $\vartheta>\theta$. Using the inaccessibility of $\theta$, we find an elementary submodel $M$ of $\HH{\vartheta}$ of cardinality less than $\theta$ with $A,\PPP_\theta\in M$ and  the property that $M\cap\HH{\theta}=\HH{\kappa}$ for some strong limit cardinal $\kappa$ of uncountable cofinality. Then, elementarity implies that $A\cap\HH{\kappa}$ is a maximal antichain in $\PPP_\theta\cap\HH{\kappa}$. Since the properties of $\kappa$ listed above ensure that $\HH{\kappa}\in\calT_\theta$, we can apply Proposition \ref{proposition:complete subforcing} to see that $A\cap\HH{\kappa}$ is a maximal antichain in $\PPP_\theta$ and hence $A=A\cap\HH{\kappa}\subseteq\HH{\kappa}$ has cardinality less than $\theta$.  
\end{proof}

We end this section by showing that the above results already yield a characterization of the class of all countably inaccessible  cardinals through Neeman's pure side condition forcing in the sense of Definition \ref{definition:Characterization}. The implication from (i) to (ii) in the proof of the following theorem is a direct adaptation of arguments contained in {\cite[Section 5.1]{MR3201836}}. 

\begin{theorem}\label{aleph2}
 The following are equivalent for every infinite cardinal $\theta$: 
 \begin{enumerate}[leftmargin=0.7cm]
  \item  $\theta$ is a countably inaccessible cardinal. 
  \item $\mathbbm{1}_{\PPP_\theta}\Vdash\anf{\check{\theta}=\omega_2}$. 

  \item $\mathbbm{1}_{\PPP_\theta}\Vdash\anf{\textit{$\check{\theta}$ is a regular cardinal greater than $\omega_1$}}$. 
 \end{enumerate}
\end{theorem}

\begin{proof}
 First, assume that (i) holds. Then Corollary \ref{corollary:Proper} and Proposition \ref{stat_trans_nodes}  imply that forcing with $\PPP_\theta$ preserves both $\omega_1$ and $\theta$. Therefore it suffices to prove the following claim. 

 \begin{claim*}
   Forcing with $\PPP_\theta$ collapses all cardinals between  $\omega_1$ and $\theta$.
 \end{claim*}

 \begin{proof}[Proof of the Claim]
  Let $G$ be $\PPP_\theta$-generic over $\VV$ and set $T=\calT_\theta^\VV\cap\bigcup G$. Then there is a canonical well-ordering of $T$ that is induced by both $\in$ and $\subsetneq$. 
Moreover, Corollary \ref{corollary:aboutgeneric} shows that $\bigcup T=\HH{\theta}$. 
 Let $W_0$ and $W_1$ be two successive nodes in the canonical well-ordering of $T$. Define $C$ to be the set of all $M\in\calS_\theta^\VV\cap\bigcup G$ with $W_0\in M\in W_1$. 
 Then there is also a canonical well-ordering of $C$ that is induced by both $\in$ and $\subsetneq$, and $C\subseteq\calS_\theta^\VV$ implies that this ordering has length at most $\omega_1$. Another application of Lemma \ref{coolone} together with the fact that conditions in $G$ are closed under intersections now shows that $\bigcup C=W_1$ and hence $W_1$ has cardinality at most $\omega_1$ in $\VV[G]$. 
 Since we already know that $\bigcup T=\HH{\theta}^\VV$,  this shows that every ordinal between $\omega_1$ and $\theta$  is collapsed to $\omega_1$ in $\VV[G]$. 
\end{proof}

 In the other direction, assume that there is an infinite cardinal $\theta$ with the property that (i) fails and (iii) holds. Then we know that $\theta$ is a regular cardinal greater than $\omega_1$ and there is a $\delta<\theta$ with $\delta^\omega\geq\theta$. Let $\delta_0$ be minimal with this property.
 
 \begin{claim*}
  $\calT_\theta=\emptyset$. 
 \end{claim*}
 
 \begin{proof}[Proof of the Claim]
  Assume, towards a contradiction, that there exists a $W$ in $\calT_\theta$. By elementarity, there is an ordinal $\gamma\in W$ with the property that for all $x\in W$, there is a function $\map{f}{\omega}{\gamma}$ in $W$ with $f\notin x$. Let $\gamma_0$ be minimal with this property. Then $\gamma_0\geq\delta_0$, because otherwise the minimality of $\delta_0$ would imply that ${}^\omega\gamma_0\in\HH{\theta}$ and elementarity would then allow us to show that ${}^\omega\gamma_0$ is contained in $W$, contradicting our assumptions on $\gamma_0$. But then $\delta_0\in W$ and the countable closure of $W$ implies that ${}^\omega\delta_0\subseteq W$. By our assumption, we can conclude that $\betrag{W}\geq\delta_0^\omega\geq\theta$ and hence $W\notin\HH{\theta}$, a contradiction. 
 \end{proof}

  \begin{claim*}
   Given $x\in\HH{\theta}$, the set $\Set{p\in\PPP_\theta}{x\in\bigcup p}$ is dense in $\PPP_\theta$. 
  \end{claim*}
  
  \begin{proof}[Proof of the Claim]
   Fix a condition $p$ in $\PPP_\theta$ and let $N$ denote the Skolem hull of $\{p,x\}$ in $\HH{\theta}$. Since $\theta$ is uncountable and regular, we have $N\in\calS_\theta$. Moreover, the above claim shows that $p\subseteq\calS_\theta$ and this implies that $M\subseteq N$ holds for all $M\in p$. In particular, the set $p\cup\{N\}$ is a condition in $\PPP_\theta$ below $p$. 
  \end{proof}

  Now, let $G$ be $\PPP_\theta$-generic over $\VV$. Then the above claims show that $\bigcup G\subseteq S_\theta^\VV$ and $\HH{\theta}^\VV=\bigcup\bigcup G$. The first statement directly implies that $\bigcup G$ is well-ordered by $\subsetneq$ in $\VV[G]$ and every proper initial segment of this well-order is a subset of an element of $\calS_\theta^\VV$. In combination with the second statement, this shows that $\theta$ is a union of $\omega_1$-many countable sets in $\VV[G]$, contradicting (iii). 
\end{proof}



\section{Small embeddings and inaccessible cardinals}\label{section:Inaccessibles}

In this section, we show that the sequence $\seq{\PPP_\theta}{\theta\in\Card}$ strongly characterizes the class of all inaccessible cardinals through the non-existence of certain trees in generic extensions. In the proof of this characterization, we will make use of \emph{small embedding characterizations} of large cardinals introduced by the authors in \cite{SECFLC}.\footnote{While it would not be necessary to make use of such characterizations in this very argument, and making use of them may even seem somewhat artificial at this point, this approach allows us to already introduce some of the techniques that will be necessary in many of our later arguments, when we consider stronger large cardinal notions.} These small embedding characterizations show that many important large cardinal properties are equivalent to the existence of certain elementary embeddings from a small transitive model into some $\HH{\vartheta}$ that send their critical point to the given cardinal. They are motivated by a classical result of Magidor that characterizes supercompact cardinals in this way (see {\cite[Theorem 1]{MR0295904}}). The following definition specifies the form of embeddings that we are interested in.

\begin{definition}\label{definition:SEC}
Let $\theta<\vartheta$ be cardinals. A \emph{small embedding for $\theta$} is a non-trivial elementary embedding $\map{j}{M}{\HH{\vartheta}}$ with $j(\crit{j})=\theta$ and $M\in \HH{\vartheta}$ transitive.
\end{definition}

Using this terminology, Magidor's result now says that a cardinal $\theta$ is supercompact if and only if for every $\vartheta>\theta$, there is a small embedding $\map{j}{M}{\HH{\vartheta}}$ with $M=\HH{\kappa}$ for some $\kappa<\theta$. The results of \cite{SECFLC} provide similar characterizations for many other large cardinal properties that all rely on certain correctness properties of the corresponding domain model $M$ of the embedding. For our characterization of inaccessible cardinals, we will use the following small embedding characterization of these cardinals, which is a minor variation of their small embedding characterization presented in \cite{SECFLC}. For the sake of completeness, we include the short proof of this statement.

\begin{lemma}\label{smallembcharinaccessible}
 The following statements are equivalent for every uncountable regular cardinal $\theta$: 
 \begin{enumerate}[leftmargin=0.7cm]
  \item $\theta$ is inaccessible. 

  \item For all sufficiently large cardinals $\vartheta$ and all $x\in\HH{\vartheta}$, there is a small embedding $\map{j}{M}{\HH{\vartheta}}$ for $\theta$ such that $x\in\ran{j}$ and $\crit{j}$ is a strong limit cardinal of uncountable cofinality. 
 \end{enumerate} 
\end{lemma}

\begin{proof} 
 First, note that an uncountable regular cardinal $\theta$ is inaccessible if and only if the strong limit cardinals of uncountable cofinality below $\theta$ form a stationary subset of $\theta$. Now, assume that $\theta$ is inaccessible, let $\vartheta>\theta$ be a cardinal, and pick $x\in\HH{\vartheta}$. Let $\seq{X_\alpha}{\alpha<\theta}$ be a continuous and increasing sequence of elementary substructures of $\HH{\vartheta}$ of cardinality less than $\theta$ with $x\in X_0$ and $\alpha\subseteq X_\alpha\cap\theta\in\theta$ for all $\alpha<\theta$. By the above remark, there is a strong limit cardinal $\alpha<\theta$ of uncountable cofinality such that $\alpha=X_\alpha\cap\theta$. Let $\map{\pi}{X_\alpha}{M}$ denote the corresponding transitive collapse. Then $\map{\pi^{{-}1}}{M}{\HH{\vartheta}}$ is a small embedding for $\theta$ with $x\in\ran{j}$, and  $\crit{j}=\alpha$ is a strong limit cardinal of uncountable cofinality. 

 In the other direction, assume that (ii) holds, let $C$ be a club in $\theta$ and pick a small embedding $\map{j}{M}{\HH{\vartheta}}$ for $\theta$ such that $C\in\ran{j}$ and $\crit{j}$ is a strong limit cardinal of uncountable cofinality. Then elementarity implies that $\crit{j}\in C$. This shows that the strong limit cardinals of uncountable cofinality are stationary in $\theta$, and hence that $\theta$ is inaccessible. 
\end{proof}

The following proposition establishes some connections between small embeddings and Neeman's pure side condition forcing.

\begin{proposition}\label{proposition:LiftingSmallEmbeddings}
 Let $\theta$ be an inaccessible cardinal and let $\map{j}{M}{\HH{\vartheta}}$ be a small embedding for $\theta$ with the property that $\vartheta$ is regular and $\crit{j}$ is a strong limit cardinal of uncountable cofinality. If $G$ is $\PPP_\theta$-generic over $\VV$, then the following statements hold: 
 \begin{enumerate}[leftmargin=0.7cm]
  \item $\HH{\crit{j}}\in M\cap\calT_\theta$ and $\PPP_{\crit{j}}=\PPP_\theta\cap\HH{\crit{j}}\in M$ is a complete suborder of $\PPP_\theta$ with $j(\PPP_{\crit{j}})=\PPP_\theta$. 

  \item $G_j:=G\cap\HH{\crit{j}}$ is $\PPP_{\crit{j}}$-generic over $\VV$ and $\VV[G]$ is a $(\dot{\QQQ}^{\HH{\crit{j}}}_\theta)^{G_j}$-generic extension of $\VV[G_j]$. 

  \item The pair $(\HH{\vartheta}^{\VV[G_j]},\HH{\vartheta}^{\VV[G]})$  has the $\sigma$-approximation property. 

  \item There is an elementary embedding $\map{j_G}{M[G_j]}{\HH{\vartheta}^{\VV[G]}}$  with the property that $j_G(\dot{x}^{G_j})=j(\dot{x})^G$ holds for every $\PPP_{\crit{j}}$-name $\dot{x}$ in $M$. 
 \end{enumerate}
\end{proposition}

\begin{proof}
 Using the inaccessibility of $\theta$ and elementarity, we can find a surjection $\map{e}{\crit{j}}{\HH{\crit{j}}^M}$ with the property that $j(e)[\kappa]=\HH{\kappa}$ holds for every strong limit cardinal $\kappa\le\theta$. But this implies that $$\HH{\crit{j}} ~ = ~ j(e)[\crit{j}] ~ = ~ e[\crit{j}] ~ = ~ \HH{\crit{j}}^M.$$ In addition, our assumptions directly imply that the set $\HH{\crit{j}}$ is a transitive, countably closed elementary substructure of $\HH{\theta}$, $\calS_{\crit{j}}=\calS_\theta\cap\HH{\crit{j}}$, $\calT_{\crit{j}}=\calT_\theta\cap\HH{\crit{j}}$,  $\PPP_{\crit{j}}=\PPP_\theta\cap\HH{\crit{j}}=\PPP_{\crit{j}}^M$, and Proposition \ref{proposition:complete subforcing} implies that $\PPP_{\crit{j}}$ is a complete suborder of $\PPP_\theta$. Since the partial order $\PPP_\theta$ is uniformly definable from the parameter $\theta$, we also obtain that $j(\PPP_{\crit{j}})=\PPP_\theta$. This proves (i). Moreover, by Lemma \ref{trans_in_generic}, this argument also shows that $\langle\HH{\crit{j}}\rangle\in G$, and  we can therefore apply Proposition \ref{proposition:appropriate} and Lemma \ref{lemma:forcing_equivalent} to show that (ii) holds. In addition, we can combine Corollary \ref{corollary:sigma_approximation} with the first two statements to derive (iii).  Finally,  since $j(\PPP_{\crit{j}})=\PPP_\theta$ and $j[G_j]=G_j\subseteq G$, a standard argument (see, for example, {\cite[Proposition 9.1]{MR2768691}}) shows that there is an embedding $\map{j_G}{M[G_j]}{\HH{\vartheta}[G]}$ with $j_G(\dot{x}^{G_j})=j(\dot{x})^G$ for every $\PPP_{\crit{j}}$-name $\dot{x}$ in $M$. But then $\HH{\vartheta}^{\VV[G]}=\HH{\vartheta}[G]$ shows that (iv) also holds. 
\end{proof}

The next definition introduces the combinatorial concept that relates to inaccessible cardinals via Neeman's pure side condition forcing.

\begin{definition}
 A tree of height $\omega_1$ and cardinality  $\aleph_1$ is a \emph{weak Kurepa tree} if it has at least $\aleph_2$-many cofinal branches. 
 \end{definition}

\begin{theorem}\label{theorem:CharInaccessible}
 The following statements are equivalent for every countably inaccessible cardinal $\theta$: 
 \begin{enumerate}[leftmargin=0.7cm]
   \item $\theta$ is an inaccessibe cardinal. 

   \item $\mathbbm{1}_{\PPP_\theta}\Vdash\anf{\textit{There are no weak Kurepa trees}}$.  
 \end{enumerate}
\end{theorem}

\begin{proof}
 First, assume that $\theta$ is inaccessible, let $\dot{T}$ be a $\PPP_\theta$-name for a tree of height $\omega_1$ and cardinality  $\aleph_1$ and let $\dot{x}$ be a nice $\PPP_\theta$-name for a subset of $\omega_1$ coding $\dot{T}$. Use Lemma \ref{smallembcharinaccessible} to find a small embedding $\map{j}{M}{\HH{\vartheta}}$ for $\theta$ such that  $\dot{x}\in\ran{j}$ and $\crit{j}$ is a strong limit cardinal of uncountable cofinality. Since Lemma \ref{chain_condition} implies that $\dot{x}\in\HH{\theta}$ and Proposition \ref{proposition:LiftingSmallEmbeddings} shows that $\PPP_{\crit{j}}=\PPP_\theta\cap\HH{\crit{j}}$, elementarity implies that $\dot{x}$ is actually a $\PPP_{\crit{j}}$-name contained in $M$. 
 Let $G$ be $\PPP_\theta$-generic over $\VV$ and let $\map{j_G}{M[G_j]}{\HH{\vartheta}^{\VV[G]}}$ be the embedding given by Proposition \ref{proposition:LiftingSmallEmbeddings}. Since $\dot{x}^G=\dot{x}^{G_j}\in M[G_j]\subseteq\HH{\vartheta}^{\VV[G_j]}$, we know that $\dot{T}^G$ is an element of $\HH{\vartheta}^{\VV[G_j]}$. Moreover, since the pair $(\HH{\vartheta}^{\VV[G_j]},\HH{\vartheta}^{\VV[G]})$  has the $\sigma$-approximation property, and Corollary \ref{corollary:Proper} implies  $\omega_1^\VV=\omega_1^{\VV[G_j]}=\omega_1^{\VV[G]}$, we know that every cofinal branch through $\dot{T}^G$ in $\VV[G]$ is an element of $\VV[G_j]$. Since $\PPP_{\crit{j}}$  has size less than $\theta$, we know that $\theta $ is still inaccessible in $\VV[G_j]$, and therefore that $\dot{T}^G$ has less than $\theta$-many cofinal branches in $\VV[G]$. But Theorem \ref{aleph2} shows that $\theta=\omega_2^{\VV[G]}$, which allows us to conclude that $\dot{T}^G$ is not a weak Kurepa tree in $\VV[G]$. 

For the reverse implication, assume that $\theta$ is countably inaccessible, however not inaccessible. Let $\kappa<\theta$ be minimal with $2^\kappa\geq\theta$. Since $\theta$ is countably inaccessible, we know that $\cof{\kappa}$ is uncountable. Let $G$ be $\PPP_\theta$-generic over $\VV$. Corollary \ref{corollary:Proper} implies that $\kappa$ has uncountable cofinality in $\VV[G]$. Since Theorem \ref{aleph2} shows that $\theta=\omega_2^{\VV[G]}$ , we therefore know that $\kappa$ has cofinality $\omega_1$ in $\VV[G]$.  In $\VV[G]$, pick a cofinal subset $U$ of $\kappa$ of order-type $\omega_1$, and define $\TTT$ to be the tree consisting of functions $\map{t}{\alpha}{2}$ in $\VV$ with $\alpha\in U$, ordered by inclusion. Then the minimality of $\lambda$ implies that the levels of $\TTT$ all have size at most $\aleph_1$ in $\VV[G]$. However, since each function from $\kappa$ to $2$ induces a unique cofinal branch through $\TTT$ and $2^\kappa\geq\theta=\omega_2^{\VV[G]}$, we know that $\TTT$ has at least $\aleph_2$-many branches in $\VV[G]$. This shows that $\TTT$ is a weak Kurepa tree in $\VV[G]$.
\end{proof}

A combination of the above result with Theorem \ref{aleph2} now shows that the sequence $\seq{\PPP_\theta}{\theta\in\Card}$ characterizes the class of  inaccessible cardinals through the statement 
 \begin{quote}
  \anf{\emph{$\theta$ is a regular cardinal greater than $\omega_1$ with the property that for every uncountable cardinal $\kappa<\theta$, every tree of cardinality and height $\kappa$ has less than $\theta$-many cofinal branches}}.
 \end{quote}

Since this statement is obviously a consequence of the inaccessibility of $\theta$, the provided characterization is strong.


\section{Internal large cardinals and Mahlo cardinals}\label{section:INternalAndMahlo}

In the case of small embedding characterizations of large cardinal properties that imply the Mahloness of the given cardinal, the  combinatorics obtained by lifting the witnessing embeddings to a suitable collapse extension can be phrased as meaningful combinatorial principles that we call \emph{internal large cardinals}. These principles describe strong fragments of large cardinal properties that were characterized by small embeddings, which however can also hold at small cardinals $\theta$. They postulate the existence of small embeddings $\map{j}{M}{\HH{\vartheta}}$ for $\theta$ together with the existence of transitive models $N$ of $\ZFC^{-}$ with the property that $M\in N\subseteq\HH{\vartheta}$, for which (an appropriate variant of) the correctness property that held between $M$ and $\VV$ in the original small embedding characterization now holds between $M$ and $N$, and some correctness property induced by the properties of the tails of the collapse forcing used holds between $N$ and $\HH{\vartheta}$. The guiding idea of this setup is that it should resemble the situation after lifting a given small embedding, with the inner model $N$ resembling the collapse extension of the original universe in which only the critical point of the small embedding, rather than the actual large cardinal, has been collapsed.
For many important consistency proofs, the principles defined in this way turn out to  capture the crucial combinatorial properties of small cardinals established in these arguments.  
Moreover, in several cases, these principles turn out to be reformulations of existing combinatorial properties. 
For example, if we combine Magidor's small embedding characterization of supercompactness given by the results of \cite{MR0295904} with collapse forcings whose tails possess the $\sigma$-approximation property, then we end up with the following internal version of supercompactness, that turns out to be equivalent to a generalized tree property introduced by  Wei{\ss} in \cite{weissthesis} and \cite{MR2959668} (see Proposition \ref{proposition:InternalSupercompactISP}):

\begin{definition}\label{definition:InteralAPSupercompact}
 A cardinal $\theta$ is \emph{internally AP supercompact} if for all sufficiently large regular cardinals $\vartheta$ and all $x\in\HH{\vartheta}$, there is a small embedding $\map{j}{M}{\HH{\vartheta}}$ for $\theta$, and a transitive model $N$ of $\ZFC^-$ such that $x\in\ran{j}$, and such that the following statements hold: 
 \begin{enumerate}[leftmargin=0.7cm]
  \item $N\subseteq\HH{\vartheta}$ and the pair $(N,\HH{\vartheta})$ satisfies the $\sigma$-approximation property. 
 
  \item $M=\HH{\kappa}^N$ for some $N$-cardinal $\kappa<\theta$.
 \end{enumerate}
\end{definition}

We will later use this principle to characterize supercompactness through Neeman's pure side condition forcing (see Corollary \ref{corollary:CharSupercompactISP}).  
Moreover, all internal large cardinal principles studied in this paper will turn out to be consequences of the above principle and, in combination with results of Viale and Wei{\ss} from \cite{MR2838054}, this fact implies that $\PFA$ causes $\omega_2$ to possess all internal large cardinal properties discussed in this paper.

We will make use of the concept of internal large cardinals in many places throughout this paper. While the general setup will be postponed to the forthcoming \cite{internal}, we will only introduce and make use of internal large cardinals with respect to the $\sigma$-approximation property in this paper (\anf{\emph{internally AP}} large cardinals). In the following, we present the definition of internal Mahloness and show how this concept can be used to characterize Mahlo cardinals through Neeman's pure side condition forcing. This is motivated by the following small embedding characterization of Mahlo cardinals from \cite{SECFLC} (see {\cite[Corollary 2.2 \& Lemma 3.4]{SECFLC}}). Its proof is a small modification of the proof of Lemma \ref{smallembcharinaccessible}.

\begin{lemma} \label{lemma:SECMahlo}
 The following statements are equivalent for every uncountable regular cardinal $\theta$: 
 \begin{enumerate}[leftmargin=0.7cm]
  \item $\theta$ is a Mahlo cardinal. 

  \item For every sufficiently large cardinal $\vartheta$ and all $x\in\HH{\vartheta}$, there is a small embedding $\map{j}{ M}{\HH{\vartheta}}$ for $\theta$ such that $x\in\ran{j}$ and $\crit{j}$ is an inaccessible cardinal.
 \end{enumerate} 
\end{lemma}

This leads us to the following:

\begin{definition}\label{internallymahlo}
A cardinal $\theta$ is \emph{internally AP Mahlo} if for all sufficiently large regular cardinals $\vartheta$ and all $x\in\HH{\vartheta}$, there is a small embedding $\map{j}{M}{\HH{\vartheta}}$ for $\theta$, and a transitive model $N$ of $\ZFC^-$ such that $x\in\ran j$, and the following statements hold:  
 \begin{enumerate}[leftmargin=0.7cm]
   \item $N\subseteq\HH{\vartheta}$, and the pair $(N,\HH{\vartheta})$ satisfies the $\sigma$-approximation property. 

   \item $M\in N$, and $\Poti{\crit{j}}{\omega_1}^N\subseteq M$.\footnote{In the situation of Lemma \ref{lemma:SECMahlo}, (ii), we also obtain that $\HH{\crit j}\subseteq M$ by \cite[Lemma 3.1]{SECFLC}, and it would thus also be reasonable to require that $\HH{\crit j}^N\subseteq M$ here. However, in the light of generalized such assumptions for example in Definition \ref{definition:internallyineffable}, it seems more suitable to use our present assumption. In any case, we will not make use of this assumption in the present section (but we will make use of our generalized assumptions in our later sections).}

  \item $\crit j$ is a regular cardinal in $N$.\footnote{Since $N$ is supposed to resemble an intermediate collapse forcing extension of some original universe of set theory, in which $\crit j$ has been collapsed, we only ask for it to be regular, rather than inaccessible, in $N$.} 
  \end{enumerate} 
\end{definition}

For a tree $\TTT$ of height $\theta$ and $S\subseteq\theta$, let $\TTT\restriction S$ denote the tree consisting of all nodes of $\TTT$ on levels in $S$, with the ordering inherited from $\TTT$. The following notion introduced by Todor{\v{c}}evi{\'c} will allow us to study important consequences of the above definition.

\begin{definition}\label{definition:SpecialTree}
 Let $\theta$ be an uncountable regular cardinal, let $S$ be a subset of $\theta$ and let $\TTT$ be a tree of height $\theta$. 
 \begin{enumerate}[leftmargin=0.7cm]
   \item A map $\map{r}{\TTT\restriction S}{\TTT}$ is \emph{regressive} if $r(t)<_\TTT t$ holds for every $t\in\TTT\restriction S$ that is not minimal in $\TTT$. 
  \item The set \emph{$S$ is non-stationary with respect to $\TTT$} if there is a regressive map $\map{r}{\TTT\restriction S}{\TTT}$ with the property that for every $t\in\TTT$ there is a function $\map{c_t}{r^{{-}1}\{t\}}{\theta_t}$ such that $\theta_t$ is a cardinal smaller than $\theta$ and $c_t$ is injective on $\leq_\TTT$-chains. 

   \item The tree $\TTT$ is \emph{special} if the set $\theta$ is non-stationary with respect to the tree $\TTT$. 
 \end{enumerate}
\end{definition}

It is easy to see that $\ZFC^-$ proves special trees to not have cofinal branches. Moreover, note that the statement \anf{\emph{$\TTT$ is special}} is upwards-absolute between transitive models of $\ZFC^-$ in which the height of $\TTT$ remains regular. A result of Todor{\v{c}}evi{\'c} (see {\cite[Theorem 14]{MR793235}}) shows that, given an infinite cardinal $\kappa$, this definition generalizes the classical notion of a  \emph{special $\kappa^+$-tree}, i.e.\ a tree of height $\kappa^+$ that is a union of $\kappa$-many antichains. In addition, Todor{\v{c}}evi{\'c} showed that an inaccessible cardinal $\theta$ is Mahlo if and only if there are no special $\theta$-Aronszajn trees (see {\cite[Theorem 6.1.4]{MR2355670}}). 

\begin{proposition}\label{proposition:InternalAPMahloSTP}
 Let $\theta$ be an internally AP Mahlo cardinal. Then $\theta$ is uncountable and regular, and there are no special $\theta$-Aronszajn trees. 
\end{proposition}

\begin{proof}
 First, if $\map{j}{M}{\HH{\vartheta}}$ is any small embedding 
 for $\theta$, then elementarity implies that $\crit{j}$ is uncountable and regular in $M$, and hence $\theta$ has the same properties in $\HH{\vartheta}$. Next, assume that there is a special $\theta$-Aronszajn tree $\TTT$ with domain $\theta$ and let $\map{j}{M}{\HH{\vartheta}}$ be a small embedding for $\theta$ such that the properties listed in Definition \ref{internallymahlo} hold, and such that there is a tree $\SSS$ of height $\crit{j}$ with domain $\crit{j}$ in $M$ with $j(\SSS)=\TTT$. Then elementarity implies that $\SSS$ is special in $M$ and, by the above remarks, it is also special in $N$. Since $j\restriction\SSS=\id_\SSS$ and $\TTT$ is a $\theta$-Aronszajn tree, we know that $\SSS$ is a proper initial segment of $\TTT$ and hence $\HH{\vartheta}$ contains a cofinal branch $b$ through $\SSS$. But then, together with the above remarks, the regularity of $\crit{j}$ in $N$ implies that $b$ is not an element of $N$. By our assumptions, this implies that there is $x\in\Poti{\crit{j}}{\omega_1}^N$ with $b\cap x\notin N$. However, since $\crit{j}$ is regular and uncountable in $N$, there is some node $s\in b\cap\SSS$ such that $b\cap x$ is contained in the set of all predecessors of $s$ in $\TTT$ and hence $b\cap x$ is an element of $N$, a contradiction.  
\end{proof}

\begin{corollary}\label{corollary:InternalMahloAtInaccessibles}
 The following statements are equivalent for every inaccessible cardinal $\theta$:  
 \begin{enumerate}[leftmargin=0.7cm]
  \item $\theta$ is a Mahlo cardinal. 

  \item $\theta$ is internally AP Mahlo. 
 \end{enumerate}
\end{corollary}

\begin{proof}
 Assume that (i) holds and $\vartheta$ is a sufficiently large regular cardinal in the sense of Lemma \ref{lemma:SECMahlo}. Given $x\in\HH{\vartheta}$, let $\map{j}{M}{\HH{\vartheta}}$ be a small embedding for $\vartheta$ with $x\in\ran{j}$ and with the properties listed in Lemma \ref{lemma:SECMahlo}, and set $N=\HH{\vartheta}$. Then $j$ and $N$ witness that (ii) holds. In the other direction, if (ii) holds, then Proposition \ref{proposition:InternalAPMahloSTP} implies that $\theta$ is an inaccessible cardinal with the property that there are no special $\theta$-Aronszajn trees, and {\cite[Theorem 6.1.4]{MR2355670}} then shows that $\theta$ is a Mahlo cardinal. 
\end{proof}

Using the above, we are now ready to prove the following result that extends Proposition \ref{proposition:InternalAPMahloSTP} and directly yields the desired characterization of Mahloness.

\begin{theorem}\label{MahloChar}
  The following statements are equivalent for every inaccessible cardinal $\theta$:  
 \begin{enumerate}[leftmargin=0.7cm]
  \item $\theta$ is a Mahlo cardinal. 

  \item $\mathbbm{1}_{\PPP_\theta}\Vdash\anf{\textit{$\omega_2$ is internally AP Mahlo}}$. 

\item $\mathbbm{1}_{\PPP_\theta}\Vdash\anf{\textit{There are no special $\omega_2$-Aronszajn trees}}$.
\end{enumerate}
\end{theorem}

\begin{proof}
 First, assume that (i) holds, let $\vartheta$ be a sufficiently large regular cardinal in the sense of Lemma \ref{lemma:SECMahlo}, let $G$ be $\PPP_\theta$-generic over $\VV$ and pick $x\in\HH{\vartheta}^{\VV[G]}$. By Lemma \ref{chain_condition}, we can find a $\PPP_\theta$-name $\dot{x}$ in $\HH{\vartheta}^\VV$ with $x=\dot{x}^G$. 
Pick a small embedding $\map{j}{M}{\HH{\vartheta}}$ with $\dot{x}\in\ran{j}$, and with the properties listed in Lemma \ref{lemma:SECMahlo}. Next, let $\map{j_G}{M[G_j]}{\HH{\vartheta}^{\VV[G]}}$ be the embedding given by Proposition \ref{proposition:LiftingSmallEmbeddings}, and set $N=\HH{\vartheta}^{\VV[G_j]}$. Then $M[G_j]\in N\subseteq\HH{\vartheta}^{\VV[G]}$, and the pair $(N,\HH{\vartheta}^{\VV[G]})$ satisfies the $\sigma$-approximation property. Moreover, since $\crit{j}$ is inaccessible in $\VV$ and $\PPP_{\crit{j}}=\PPP_\theta\cap\HH{\crit{j}}$, Theorem \ref{aleph2} implies that $\crit{j}$ is regular in $N$. Finally, a combination of Lemma \ref{chain_condition} with Proposition \ref{proposition:LiftingSmallEmbeddings} shows that $\Poti{\crit{j}}{\omega_1}^N\subseteq\HH{\crit{j}}^N\subseteq M[G_j]$. In summary, this shows that $j_G$ and $N$ witness that $\theta$ is internally AP Mahlo with respect to $x$ in $\VV[G]$. In particular, we can conlude that (ii) holds in this case.

 Now, assume that (i) fails. By {\cite[Theorem 6.1.4]{MR2355670}}, this implies that there exists a special $\theta$-Aronszajn tree $\TTT$. Let $G$ be $\PPP_\theta$-generic over $\VV$. Then, Theorem \ref{aleph2}  shows that $\theta=\omega_2^{\VV[G]}$ and thus, by the above remarks, $\TTT$ is a special $\omega_2$-Aronszajn tree in $\VV[G]$. This shows that (iii) fails. 
\end{proof}

By combining the results of Section \ref{section:Inaccessibles} with Corollary \ref{corollary:InternalMahloAtInaccessibles} and the above theorem, we now directly conclude that  the sequence $\seq{\PPP_\theta}{\theta\in\Card}$ provides a strong characterization of the class of all Mahlo cardinals.


\section{Indescribable cardinals}\label{section:Indescribable}

In this section, we present strong characterizations of indescribable cardinals through Neeman's pure side condition forcing. A combination of results from \cite{MR3201836} with Theorem \ref{theorem:CharInaccessible} already provides such a characterization for $\Pi^1_1$-indescribable cardinals (i.e.\ weakly compact cardinals) with the help of the \emph{tree property}. In the case where either $m$ or $n$ is greater than $1$, $\Pi^m_n$-indescribable cardinals seem to be lacking such a canonical combinatorial essence. This motivates viewing the internal large cardinal principles used in the following characterizations as properties that capture the combinatorial essence of the higher degrees of indescribability.

As we will have to work a lot with higher order objects in this section, let us indicate the order of free variables by a superscript attached to them, letting $v^0$ denote a standard first order free variable, letting $v^1$ denote a free variable that is to be interpreted by an element of the powerset of the domain, and so on. In the same way, we will also label higher order quantifiers. Remember that, given $0<{m,n}<\omega$, an uncountable cardinal $\kappa$ is $\Pi^m_n$-indescribable if for every $\Pi^m_n$-formula $\Phi(v^1)$ and every $A\subseteq\VV_\kappa$ such that $\VV_\kappa\models\Phi(A)$, there is a $\delta<\kappa$ with $\VV_\delta\models\Phi(A\cap\VV_\delta)$.

In \cite{SECFLC}, we used results of Hauser from \cite{MR1133077} to obtain the following small embedding characterization for indescribable cardinals (see {\cite[Lemma 3.4 \& 4.2]{SECFLC}}), which will be the basis of an internal concept of indescribability in the following.

\begin{lemma}\label{indescribablesmallembeddingchar}
 Given $0<{m,n}<\omega$, the following statements are equivalent for every cardinal $\theta$: 
 \begin{enumerate}[leftmargin=0.7cm]
  \item $\theta$ is a $\Pi^m_n$-indescribable cardinal. 

  \item For every sufficiently large cardinal $\vartheta$ and all $x\in\HH{\vartheta}$, there is a small embedding $\map{j}{M}{\HH{\vartheta}}$ for $\theta$ such that $x\in\ran{j}$,  ${}^{{<}\crit{j}}M\subseteq M$ and $$(\VV_{\crit j}\models\varphi(A))^M ~ \Longrightarrow ~  \VV_{\crit j}\models\varphi(A)$$ for every $\Pi^m_n$-formula $\varphi(v^1)$ with parameter $A\in M\cap\VV_{\crit j+1}$.
  \end{enumerate}
\end{lemma}

Note that, although  the statement of {\cite[Lemma 4.2]{SECFLC}} does not mention the above closure assumption  ${}^{{<}\crit{j}}M\subseteq M$, the proof presented there yields domain models $M$ with this property. Moreover, note that $\crit{j}$ is inaccessible in $\VV$ whenever $j$ is an embedding witnessing (ii).

\begin{definition}\label{definition:InternApprIndescrCard}
 Given $0<{m,n}<\omega$, we say that a cardinal $\theta$ is \emph{internally AP $\Pi^m_n$-indescribable} 
if for all sufficiently large regular cardinals $\vartheta$ and all $x\in\HH{\vartheta}$, 
 there is a small embedding $\map{j}{M}{\HH{\vartheta}}$ for $\theta$, and a transitive model $N$ of $\ZFC^-$ such that $x\in\ran j$, and the following statements hold: 
\begin{enumerate}[leftmargin=0.7cm]
   \item $N\subseteq\HH{\vartheta}$, and the pair $(N,\HH{\vartheta})$ satisfies the $\sigma$-approximation property.  

 \item $M\in N$, and $\Poti{\crit{j}}{\omega_1}^N\subseteq M$. 

  \item $\crit j$ is regular in $N$ and $$(\HH{\crit{j}}\models\Phi(A))^M ~ \Longrightarrow ~ (\HH{\crit{j}}\models\Phi(A))^N$$ for every $\Pi^m_n$-formula $\Phi(v^1)$ with parameter $A\in\POT{\HH{\crit{j}}}^M$.   
 \end{enumerate}
\end{definition}

Note that the above definition directly implies that internally AP indescribable cardinals are internally AP Mahlo. As mentioned earlier, this principle may be viewed as a strong substitute for the tree property with respect to higher levels of indescribability. For the basic case of $\Pi^1_1$-indescribability, we easily obtain the following result.

\begin{proposition}\label{proposition:IntAPwcTP}
 Let $\theta$ be an internally AP $\Pi^1_1$-indescribable cardinal. Then $\theta$ is an uncountable regular cardinal with the tree property.  
\end{proposition}

\begin{proof}
 By Proposition \ref{proposition:InternalAPMahloSTP}, we know that $\theta$ is uncountable and regular. Assume for a contradiction that there exists a $\theta$-Aronszajn tree $\TTT$ with domain $\theta$ and pick a small embedding $\map{j}{M}{\HH{\vartheta}}$ for $\theta$ such that  the properties listed in Definition \ref{definition:InternApprIndescrCard} hold for $m=n=1$ and $j(\SSS)=\TTT$ holds for some tree $\SSS\in M$ of height $\crit{j}$. Then $\SSS$ is a $\crit{j}$-Aronszajn tree in $M$ and, since this statement can be formulated over $\HH{\crit{j}}^M$ by a $\Pi^1_1$-formula with parameter $\SSS\in\POT{\HH{\crit{j}}}^M$, we can conclude that $\SSS$ is a $\crit{j}$-Aronszajn tree in $N$. As in the proof of  Proposition \ref{proposition:InternalAPMahloSTP}, elementarity implies that $\SSS$ is an initial segment of $\TTT$, and the $\sigma$-approximation property implies that $N$ contains a cofinal branch through $\SSS$, a contradiction. 
\end{proof}

Before we may commence with the main results of this section, we need to make a few technical observations. Namely, we will want to identify countable subsets of $\VV_{\kappa+k}$ with certain elements of $\VV_{\kappa+k}$, and also view forcing statements about $\Pi^m_n$-formulas themselves as $\Pi^m_n$-formulas. The basic problem about this is that the forming of (standard) ordered pairs is rank-increasing. 
For example, names for elements of $\VV_{\kappa+k}$ are usually not elements of $\VV_{\kappa+k}$ when $k>0$, even if $\kappa$ is regular and the forcing satisfies the $\kappa$-chain condition. 
However, there are well-known alternative definitions of ordered pairs (see, for example, \cite{MR0269496}) that possess all the nice properties of the usual ordered pairs that we will need, and which are, in addition, not rank-increasing. 
While it would be tedious to do so, it is completely straightforward to verify that one can base all set theory (like the definition of finite tuples) and forcing theory (starting with the definition of forcing names) on these modified ordered pairs, and preserve all of their standard properties, while additionally obtaining our desired properties. We will assume that we work with the modified ordered pairs for the remainder of this section. The following lemma shows how this approach allows us to formulate $\Pi^m_n$-statements in the forcing language in a $\Pi^m_n$-way.

\begin{lemma}\label{opconsequences}
  Work in $\ZFC^-$. Let $m\in\omega$, and let $\kappa$ be a cardinal such that $\POTT{\HH{\kappa}}{m}$ exists. Assume that $\PPP\subseteq\HH{\kappa}$ is a partial order such that forcing with $\PPP$ preserves $\kappa$ and such that for every $\PPP$-name $\tau$ for an element of $\HH{\kappa}$, there is a $\PPP$-name $\sigma$ in $\HH{\kappa}$ with $\mathbbm{1}_{\PPP}\Vdash\anf{\sigma=\tau}$. 
 \begin{enumerate}[leftmargin=0.7cm]
  \item If $\tau$ is a $\PPP$-name for an element of $\POTT{\HH{\kappa}}{m+1}$, then there is a $\PPP$-name $\sigma$ in $\POTT{\HH{\kappa}}{m+1}$ with the property that $\mathbbm{1}_{\PPP}\Vdash\anf{\sigma=\tau}$. 

  \item If $\sigma_0,\ldots,\sigma_k$ is a finite sequence of $\PPP$-names in $\POTT{\HH{\kappa}}{m+1}$, and $\Phi$ is a $\Pi^{m+1}_n$-formula for some $n\in\omega$, then the statement $p\Vdash\Phi(\sigma_0,\ldots,\sigma_k)$ is  equivalent to a $\Pi^{m+1}_n$-formula. 
 \end{enumerate}
\end{lemma}

We are now ready to prove the main results of this section. 
The following result will show that our characterization of indescribability through internal AP indescribability presented below is strong.

\begin{lemma}\label{lemma:InaccInternalApproxIndescr}
 Given $0<{m,n}<\omega$, the following statements are equivalent for every inaccessible cardinal $\theta$: 
 \begin{enumerate}[leftmargin=0.7cm]
  \item $\theta$ is a $\Pi^m_n$-indescribable cardinal.  
  
  \item $\theta$ is an internally AP $\Pi^m_n$-indescribable cardinal. 
 \end{enumerate} 
\end{lemma} 

 \begin{proof}  
  First, assume that $(i)$ holds, let $\vartheta$ be sufficiently large in the sense of Lemma \ref{indescribablesmallembeddingchar},(ii), fix $x\in\HH{\vartheta}$ and pick a small embedding $\map{j}{M}{\HH{\vartheta}}$ satisfying $x\in\ran{j}$, and satisfying the properties listed in Lemma \ref{indescribablesmallembeddingchar}. By the above remarks, $\crit{j}$ is an inaccessible cardinal and hence $\HH{\crit{j}}=\VV_\crit{j}$. Therefore, if we set $N=\HH{\vartheta}$, then $j$ and $N$ witness that $\theta$ is internally AP $\Pi^m_n$-indescribable with respect to $x$.

 In the other direction, assume that the inaccessible cardinal $\theta$ is internally AP $\Pi^m_n$-indescribable. Fix $A\subseteq\VV_\theta=\HH{\theta}$ and a $\Pi^m_n$-formula $\Phi(v^1)$ with $\VV_\theta\models\Phi(A)$. Let $\map{j}{M}{\HH{\vartheta}}$ and $N$ witness that $\theta$ is internally AP $\Pi^m_n$-indescribable with respect to $\{A,\VV_{\theta+\omega}\}$. Since $\theta$ is inaccessible, elementarity implies that $\crit{j}$ is a strong limit cardinal in $\VV$.

\begin{claim*}
 $\VV_{\crit j}\subseteq M$ and $\VV_{\crit{j}+\omega}\subseteq N$.
\end{claim*}

\begin{proof}[Proof of the Claim]
 The proof of Proposition \ref{proposition:LiftingSmallEmbeddings} contains an argument that proves the first statement. Next, assume that the second statement fails. Then there is some $k<\omega$ and $x\subseteq\VV_{\crit{j}+k}$ with $\VV_{\crit{j}+k}\subseteq N$ and $x\notin N$. By the above remarks, we can identify countable subsets of $\VV_{\crit{j}+k}$ with elements of $\VV_{\crit{j}+k}$ in a canonical way. This shows that $\Poti{x}{\omega_1}\subseteq N$, and hence the $\sigma$-approximation property implies that $x\in N$, a contradiction. 
\end{proof}

Since our assumptions imply that $\crit{j}$ is regular in $N$, the above claim directly shows that $\crit{j}$ is an inaccessible cardinal in $\VV$. Moreover, by the above claim and our assumptions, we have  $A\cap\VV_{\crit{j}}\in M$ and $j(A\cap\VV_{\crit{j}})=A$.
In addition, the above choices ensure that   $(\VV_\kappa\models\Phi(A))^{\HH{\vartheta}}$ holds and hence elementarity implies that $(\VV_{\crit{j}}\models\Phi(A\cap\VV_{\crit{j}}))^M$.  By Clause (iii) of Definition \ref{definition:InternApprIndescrCard}, we therefore know that  $(\VV_{\crit{j}}\models\Phi(A\cap\VV_{\crit{j}}))^N$ holds, and, since the above claim shows that $\Pi^m_n$-formulas over $\VV_{\crit{j}}$ are absolute between $N$ and $\VV$, this allows us to conclude that $\VV_{\crit{j}}\models\Phi(A\cap\VV_{\crit{j}})$.  
\end{proof}

We may now show that indescribability can be characterized by internal AP indescribability via Neeman's pure side condition forcing. 
Note that the equivalence of (i) and (iii) in the case $m=n=1$ already follows from the results of {\cite[Section 5.1]{MR3201836}}. 

\begin{theorem}
 Given $0<{m,n}<\omega$, the following statements are equivalent for every inaccessible cardinal $\theta$: 
 \begin{enumerate}[leftmargin=0.7cm]
  \item $\theta$ is a $\Pi^m_n$-indescribable cardinal. 

  \item $\mathbbm{1}_{\PPP_\theta}\Vdash$\anf{$\omega_2$ is internally AP $\Pi^m_n$-indescribable}.
 \end{enumerate} 
 Moreover, if $m=n=1$, then the above statements are also equivalent to the following statement: 
  \begin{enumerate}[leftmargin=0.7cm]
  \setcounter{enumi}{2}
  \item $\mathbbm{1}_{\PPP_\theta}\Vdash$\anf{$\omega_2$ has the tree property}. 
 \end{enumerate}
\end{theorem}

\begin{proof}
   Assume first that (i) holds. Pick a regular cardinal $\vartheta>\theta$ that is sufficiently large in the sense of Lemma \ref{indescribablesmallembeddingchar}, let $G$ be $\PPP_\theta$-generic over $\VV$, and pick $x\in\HH{\vartheta}^{\VV[G]}$. By Lemma \ref{chain_condition}, there is a $\PPP_\theta$-name $\dot{x}\in\HH{\vartheta}^\VV$ with $x=\dot{x}^G$. Pick a small embedding $\map{j}{M}{\HH{\vartheta}}$ witnessing that $\theta$ is $\Pi^m_n$-indescribable with respect to $\dot{x}$ as in  Lemma \ref{indescribablesmallembeddingchar}, and let $\map{j_G}{M[G_j]}{\HH{\vartheta}^{\VV[G]}}$ be the embedding given by Proposition \ref{proposition:LiftingSmallEmbeddings}. Define  $N=\HH{\vartheta}^{\VV[G_j]}$  and pick a $\Pi^m_n$-formula $\Phi(v^1)$ and $A\in\POT{\HH{\crit j}}^{M[G_j]}$ with the property  that $(\HH{\crit j}\models\Phi(A))^{M[G_j]}$.  Our assumption $({}^{{<}\crit{j}}M)^\VV\subseteq M$ implies that $\crit{j}$ is an inaccessible cardinal in $\VV$, Theorem \ref{aleph2} shows that $\crit{j}=\omega_2^{\VV[G_j]}$, and therefore Lemma \ref{chain_condition} shows that $$\Poti{\crit{j}}{\omega_1}^N ~ \subseteq ~ \HH{\crit{j}}^{\VV[G_j]} ~ \subseteq ~ M[G_j].$$ 
 Using Lemma \ref{opconsequences}, there is a $\PPP_{\crit{j}}$-name $\tau\in\POT{\HH{\crit j}}^M$ for an element of $\POT{\HH{\crit j}}$, and a condition $r\in G_j$ such that $A=\tau^{G_j}$, and such that $$r\Vdash_{\PPP_{\crit{j}}}^M\anf{\HH{\crit j}\models\Phi(\tau)}.$$ 
Again, by Lemma \ref{opconsequences}, the above forcing statement is equivalent to a $\Pi^m_n$-statement over $\HH{\crit j}$ and this statement holds true in $M$. By our assumptions, this statement holds in $\VV$, and therefore  $$r\Vdash^\VV_{\PPP_{\crit{j}}}\anf{\HH{\crit j}\models\Phi(\tau)}.$$ This allows us to conclude that  $(\HH{\crit j}\models\Phi(A))^N$ holds. Hence $j_G$ and $N$ witness that $\omega_2$ is internally AP $\Pi^m_n$-indescribable with respect to $x$ in $\VV[G]$.

   Now, assume that (ii) holds. Fix a subset $A$ of $\HH{\theta}$, and assume that $\Phi(v^1)$ is a $\Pi^m_n$-formula with $\VV_\theta\models\Phi(A)$. Let $C$ denote the club of strong limit cardinals below $\theta$ and fix a bijection $\map{b}{\theta}{\HH{\theta}}$ with $b[\kappa]=\HH{\kappa}$ for all $\kappa\in C$. 

Let $G$ be $\PPP_\theta$-generic over $\VV$, and work in $\VV[G]$. By our assumption, we can find a small embedding $\map{j}{M}{\HH{\vartheta}}$ and a transitive $\ZFC^-$-model $N$ witnessing the internal AP $\Pi^m_n$-in\-de\-scri\-ba\-bi\-li\-ty of $\theta$ with respect to $\{A,b,C,\VV_{\kappa+\omega}\}$. By elementarity, we have $\crit{j}\in C$, $\HH{\crit j}^{\VV}\in M\subseteq N$, $j(\HH{\crit{j}}^\VV)=\HH{\theta}^\VV$, $\bar A=A\cap\HH{\crit j}^\VV\in M$ and $A=j(\bar A)$. Moreover, since $\crit{j}$ is regular in $N$, the $\sigma$-approximation property between $N$ and $\HH{\vartheta}$ implies that $\cof{\crit{j}}>\omega$. Another application of the $\sigma$-approximation property then yields $\POT{\crit{j}}^\VV\subseteq N$, and therefore $\crit j$ is a regular cardinal in $\VV$.

\begin{claim*}\label{tripleformula}
  Given $k<\omega$, we have $\POTT{\HH{\crit j}}{k}^\VV\subseteq N$, and there is a $\Pi^k_0$-formula $\Phi_k(v^1,w^k)$ satisfying $$\POTT{\HH{\crit j}}{k}^\VV =  \Set{B\in\POTT{\HH{\crit j}}{k}^N}{(\HH{\crit j}\models\Phi_k(\HH{\crit j}^\VV,B))^N}$$ and $$\POTT{\HH{\theta}}{k}^\VV ~ = ~ \Set{B\in\POTT{\HH{\theta}}{k}}{\HH{\theta}\models\Phi_k(\HH{\theta}^\VV,B)}.$$ 
\end{claim*}

\begin{proof}[Proof of the Claim]
  Using induction, we will simultaneously define the formulas $\Phi_k(v^1,w^k)$, show that they satisfy the above statements, and also verify that $\POTT{\HH{\crit j}}{k}^\VV\subseteq N$. In order to start, set $\Phi_0(v^1,w^0)\equiv w^0\in v^1$. Then $\Phi_0$ is clearly as desired, and we already argued above that $\HH{\crit j}^\VV\subseteq N$.
  
  Now, assume that we arrived at stage $k+1$ of our induction. Assume, for a contradiction, that there is a subset  $B$ of $\POTT{\HH{\crit{j}}}{k+1}^\VV$ with $B\notin N$. Since the pair $(N,\VV[G])$ satisfies the $\sigma$-approximation property, we can find $b\in\Poti{\POTT{\HH{\crit{j}}}{k}^\VV}{\omega_1}^N$ with $B\cap b\notin N$. Then Corollary \ref{corollary:Proper} shows that the partial order $\PPP_\theta$ is proper in $\VV$, and we therefore find $c\in\Poti{\POTT{\HH{\crit{j}}}{k}}{\omega_1}^\VV$ with $b\subseteq c$. By identifying elements of $\Poti{\POTT{\HH{\crit{j}}}{k}}{\omega_1}^\VV$ with elements of $\POTT{\HH{\crit{j}}}{k}$, we can conclude that $B\cap c\in N$ and hence $B\cap b\in N$, a contradiction. This shows that $\POTT{\HH{\crit j}}{k+1}^\VV\subseteq N$. Moreover, since Corollary \ref{corollary:SideConditionForcingApproximationProperty} implies  that the pair $(\VV,\VV[G])$ also satisfies the $\sigma$-approximation property, it follows that $\POTT{\HH{\crit j}}{k+1}^\VV$ exactly consists of all $B\in\POTT{\HH{\crit j}}{k+1}^N$ with the property that for all $D\in\POTT{\HH{\crit j}}{k}^\VV$ that code a countable subset $d$ of $\POTT{\HH{\crit j}}{k}^\VV$, there is an element of $\POTT{\HH{\crit j}}{k}^\VV$  coding the  subset $B\cap d$. Furthermore, it also follows from the $\sigma$-approximation property for the pair $(\VV,\VV[G])$ that $\POTT{\HH{\theta}}{k+1}^\VV$ exactly consists of all $B\in\POTT{\HH{\theta}}{k+1}$ with the property that for all $D\in\POTT{\HH{\theta}}{k}$ that code a countable subset $d$ of $\POTT{\HH{\theta}}{k}^\VV$, there is an element of $\POTT{\HH{\theta}}{k}^\VV$ coding $B\cap d$.   
  Now, let $\Phi_{k+1}(v^1,w^{k+1})$ denote the canonical $\Sigma^{k+1}_0$-formula stating that for every $D\in\POTT{v^1}{k}$ such that $\Phi_k(v^1,D)$ holds and $D$ codes a countable subset $d$ of $\POTT{v^1}{k}$, there is $E\in\POTT{v^1}{k}$ such that $\Phi_k(v^1,E)$ holds and $E$ codes $d\cap w^{k+1}$. 
Then, the above remarks show that the two equalities stated in the above claim also hold at stage $k+1$. 
\end{proof}

Let $\Phi_*(u^1,v^1)$ denote the relativisation of $\Phi(v^1)$ using the formulas $\Phi_k(v^1,w^k)$, i.e.\ we obtain $\Phi_*$ from $\Phi$ by replacing each subformula of the form $\exists^k x ~ \psi$ by $\exists^k x ~ \left[\psi ~ \wedge ~ \Phi_k(u^1,x)\right]$. Then $\Phi_*$ is again a $\Pi^m_n$-formula and, by the above claim and our assumptions, we know that $\HH{\theta}\models\Phi_*(\HH{\theta}^\VV,A)$. Therefore  we can use elementarity to conclude that $$(\HH{\crit{j}}\models\Phi_*(\HH{\crit{j}}^\VV,\bar A))^M$$ and, since $j$ and $N$ witness the internal AP $\Pi^m_n$-indescribability of $\theta$, we know that $$(\HH{\crit j}\models\Phi_*(\HH{\crit j}^\VV,\bar A))^N.$$ But then the above claim shows that $(\HH{\crit j}\models\Phi(\bar A))^\VV$. These computations show that $\theta$ is $\Pi^m_n$-indescribable in $\VV$.

 Finally, assume that $m=n=1$. Then the above computations and  Proposition \ref{proposition:IntAPwcTP} directly show that (i) implies (iii). In the other direction, if (i) fails, then there is a $\theta$-Aronszajn tree $T$ and a combination of Corollary \ref{corollary:SideConditionForcingApproximationProperty} with Theorem \ref{aleph2} shows that $\mathbbm{1}_{\PPP_\theta}\Vdash\anf{\textit{$\check{T}$ is an $\omega_2$-Aronszajn tree}}$. 
\end{proof}


\section[Ineffable cardinals]{Subtle, ineffable and $\lambda$-ineffable cardinals}\label{section:ineffable}

In this section, we will present characterizations of subtle, of ineffable, and of $\lambda$-ineffable cardinals for a proper class of cardinals $\lambda$. Since the latter form a hierarchy that leads up to supercompactness (see \cite{MR0327518}), these results will also yield a characterization of supercompactness. 
The forward directions of the two main theorems of this section (Theorem \ref{subtletheorem} and Theorem \ref{ineffabletheorem}) are based on the proofs of {\cite[Theorem 7.4 \& 7.5]{SECFLC}}. 
These results are closely connected to work of Wei\ss\ from his \cite{weissthesis} and \cite{MR2959668} (see \cite[Section 7]{SECFLC}).  
Most definitions and results in this section  deal with the following concept. 

\begin{definition}
 Given a set $A$, a sequence $\seq{d_{a}}{a\in A}$ is an \emph{$A$-list} if $d_a \subseteq a$ for all $a\in A$.
\end{definition}

Following \cite{JensenKunen1969:Ineffable}, we can now define an uncountable regular cardinal $\theta$ to be \emph{subtle} if for every  $\theta$-list  $\seq{d_\alpha}{\alpha<\theta}$ and  every club $C$ in $\theta$, there are $\alpha,\beta\in C$ with $\alpha<\beta$ and $d_{\alpha}=d_{\beta}\cap \alpha$. The results of \cite{SECFLC} yield the following small embedding characterization of subtlety (see {\cite[Lemma 5.2, 5.4 \& 3.4]{SECFLC}}).

\begin{lemma}\label{lemma:SubtleSmallChar}
 The following statements are equivalent for every cardinal $\theta$: 
 \begin{enumerate}[leftmargin=0.7cm]
  \item $\theta$ is subtle.

  \item For all sufficiently large cardinals $\vartheta$, all $x\in\HH{\vartheta}$, every $\theta$-list $\vec{d}=\seq{d_\alpha}{\alpha<\theta}$, and every club $C$ in $\theta$, there is a small embedding $\map{j}{M}{\HH{\vartheta}}$ for $\theta$, such that $\crit j$ is inaccessible, $C,\vec{d},x\in \ran{j}$ and $d_\alpha=d_{\crit{j}}\cap\alpha$ for some $\alpha\in C\cap\crit{j}$.
 \end{enumerate} 
\end{lemma}

The above characterization motivates the following principle of internal subtlety:

\begin{definition}
A cardinal $\theta$ is \emph{internally AP subtle} if for all sufficiently large regular cardinals $\vartheta$, all $x\in \HH{\vartheta}$, every club $C$ in $\theta$, and every $\theta$-list $\vec{d}=\seq{d_{\alpha}}{\alpha<\theta}$, there is a small embedding $\map{j}{M}{\HH{\vartheta}}$ for $\theta$ and a transitive model $N$ of ZFC$^-$ such that $\vec{d},x,C\in \ran{j}$ and the following statements hold:
 \begin{enumerate}[leftmargin=0.7cm]
  \item$N \subseteq \HH{\vartheta}$ and the pair $(N,\HH{\vartheta})$ satisfies the $\sigma$-approximation property. 

  \item $M\in N$ and $\Poti{\crit{j}}{\omega_1}^N\subseteq M$. 

  \item If $d_{\crit{j}}\in N$, then there is  $\alpha\in C\cap\crit{j}$ with $d_{\alpha}=d_{\crit{j}}\cap \alpha$.
\end{enumerate}
\end{definition}

The consequences of internal subtlety can now be analyzed with the help of the following principle introduced by Wei\ss\ in \cite{weissthesis}.

\begin{definition}
Let $\theta$ be an uncountable regular cardinal.
 \begin{enumerate}[leftmargin=0.7cm]
  \item A $\theta$-list $\seq{d_\alpha}{\alpha<\theta}$ is \emph{slender} if there is a club $C$ in $\theta$ such that for every $\gamma\in C$ and for every $\alpha<\gamma$, there is a $\beta<\gamma$ with $d_{\gamma}\cap \alpha=d_{\beta}\cap\alpha$. 

  \item $\SSP(\theta)$ is the statement that for every slender $\theta$-list $\seq{d_{\alpha}}{\alpha<\theta}$ and every club $C$ in $\theta$, there are $\alpha,\beta\in C$ with $\alpha<\beta$ and $d_{\alpha}=d_{\beta}\cap\alpha$.
\end{enumerate}
\end{definition}

\begin{lemma}\label{lemma:InternalSubtleSSP}
 If $\theta$ is an internally AP subtle cardinal, then $\SSP(\theta)$ holds. 
\end{lemma}

 \begin{proof}
  Fix a slender $\theta$-list $\vec{d}=\seq{d_\alpha}{\alpha<\theta}$, a club $C_0$ in $\theta$ and a club $C\subseteq C_0$  witnessing the slenderness of $\vec{d}$. 
 Let $\vartheta$ be a sufficiently large regular cardinal such that there is a small embedding $\map{j}{M}{\HH{\vartheta}}$ and a transitive $\ZFC^-$-model $N$ witnessing the internal AP subtlety of $\theta$ with respect to $\vec{d}$ and $C$. Then elementarity implies that $\crit{j}\in C\subseteq C_0$.  
  Assume for a contradiction that $d_\crit{j}\notin N$. Then the $\sigma$-approximation property yields an $x\in\Poti{\crit{j}}{\omega_1}^N$ with $d_\crit{j}\cap x\notin N$. 
 Then $x\in M$ and, since $\crit{j}$ is a regular cardinal in $M$, there is an $\alpha<\crit{j}$ with $x\subseteq \alpha$. In this situation, the slenderness of $\vec{d}$ yields a $\beta<\crit{j}$ with $d_{\crit{j}}\cap\alpha=d_\beta\cap\alpha$. But then we have $$d_\crit{j} \cap x ~ =  ~ d_\crit{j} \cap x\cap\alpha ~ = ~ d_\beta\cap x\cap\alpha.$$ Since $\vec{d}\in\ran{j}$, we have $d_\beta\in M\subseteq N$ and hence $d_\crit{j} \cap x\in N$, a contradiction. 
  These computations show that $d_\crit{j}\in N$ and therefore our assumptions yield an $\alpha<\crit{j}$ with the property that $\alpha\in C\subseteq C_0$ and $d_\alpha=d_{\crit{j}}\cap\alpha$. 
 \end{proof}

\begin{corollary}
 The following statements are equivalent for every inaccessible cardinal $\theta$: 
 \begin{enumerate}[leftmargin=0.7cm]
  \item $\theta$ is a subtle cardinal. 

  \item $\theta$ is an internally AP subtle cardinal. 
 \end{enumerate}
\end{corollary}

\begin{proof}
 The forward direction is a direct consequence of Lemma \ref{lemma:SubtleSmallChar}. In the other direction, the results of {\cite[Section 1.2]{weissthesis}} show that the inaccessibility of $\theta$ implies that every $\theta$-list is slender and we can apply Lemma \ref{lemma:InternalSubtleSSP} to conclude that (i) is a consequence of (ii). 
\end{proof}

The following result provides a characterization of the class of all  subtle cardinals using Neeman's pure side condition forcing. The above corollary already shows that this characterization of subtlety is strong.

\begin{theorem}\label{subtletheorem}
  The following statements are equivalent for every inaccessible cardinal $\theta$:
  \begin{enumerate}[leftmargin=0.7cm]
    \item $\theta$ is a subtle cardinal.
    \item $\mathbbm{1}_{\PPP_\theta}\Vdash\anf{\textit{$\omega_2$ is internally AP subtle}}$.

   \item $\mathbbm{1}_{\PPP_\theta}\Vdash\SSP(\omega_2)$. 
  \end{enumerate}
\end{theorem}

\begin{proof}
First, assume that (iii) holds. Let $\vec d=\seq{d_\alpha}{\alpha<\theta}$ be a $\theta$-list, and let $C$ be a club subset of $\theta$. Since $\theta$ is inaccessible, we know that $\vec d$ is slender. Let $G$ be $\PPP_\theta$-generic over $\VV$. Since slenderness of $\theta$-lists is clearly upwards absolute to models that preserve the regularity of $\theta$, our assumption  implies that there are $\alpha,\beta\in C$ with $\alpha<\beta$ and $d_\alpha=d_\beta\cap\alpha$. These computations show that (i) holds.

Now, assume that (i) holds. Let $\dot{d}$ be a $\PPP_\theta$-name for a $\theta$-list, let $\dot{C}$ be a $\PPP_\theta$-name for a club in $\theta$, and let $\dot{x}$ be any $\PPP_\theta$-name. By Lemma \ref{chain_condition}, we can find a club $C$ in $\theta$ consisting only of limit ordinals that are closed under the G\"odel pairing function $\goedel{\cdot}{\cdot}$, with  $\mathbbm{1}_{\PPP_\theta}\Vdash$\anf{$\check{C}\subseteq\dot{C}$}. For every $\alpha<\theta$, let $\dot{d}_{\alpha}$ be a nice $\PPP_\theta$-name for the $\alpha$-th element of $\dot{d}$.  Let  $\vartheta>2^{\theta}$ be a regular cardinal that is sufficiently large with respect to Lemma \ref{lemma:SubtleSmallChar}, and which satisfies $\dot{d},\dot{x}\in \HH{\vartheta}$. Define $A$ to be the set of all inaccessible cardinals $\kappa$ less than $\theta$ for which there exists a small embedding $\map{j}{M}{\HH{\vartheta}}$ for $\theta$ with  critical point $\kappa$ and $\dot{d},\dot{x},C,\dot{C}\in \ran{j}$. Finally, let $G$ be $\PPP_\theta$-generic  over $\VV$.

Assume first that there is $\kappa<\theta$ and a small embedding $\map{j}{M}{\HH{\vartheta}}$ for $\theta$ in $\VV$ so that $j$ witnesses that $\kappa$ is an element of $A$, and $\dot{d}^{G}_\kappa\notin\VV[G_j]$, with $G_j=G \cap \HH{\kappa}$ defined as in Proposition \ref{proposition:LiftingSmallEmbeddings}.  Let $\map{j_G}{M[G_j]}{\HH{\vartheta}^{\VV[G]}}$ be the lifting of $j$ provided by Proposition \ref{proposition:LiftingSmallEmbeddings}, and set $N=\HH{\vartheta}^{\VV[G_j]}$. Then $\dot{d}^G,\dot{x}^G, \dot{C}^G\in\ran{j_G}$, the pair $(N,\HH{\vartheta}^{\VV[G]})$ satisfies the $\sigma$-approximation property, Theorem \ref{aleph2} implies that $\kappa=\omega_2^N$, and another application of Lemma \ref{chain_condition} yields $\Poti{\kappa}{\omega_1}^N\subseteq\HH{\kappa}^N\subseteq M[G_j]$. Since $\dot{d}_\kappa^G\notin N$, we can conclude that $j_G$ and $N$ witness that $\theta$ is internally AP subtle in $\VV[G]$ with respect to $\dot{d}^G, \dot{x}^G$ and $\dot{C}^G$.  

Otherwise, assume that whenever $\map{j}{M}{\HH{\vartheta}}$ is a small embedding for $\theta$ in $\VV$ that witnesses that some $\kappa<\theta$ is an element of $A$, then $\dot{d}^G_{\kappa}\in \VV[G_j]$.  Fix a condition $p\in G$ that forces this statement,  pick some $q\leq_{\PPP_\theta}p$, and work in $\VV$. Let $B$ denote the set of all $\kappa \in A$ such that $q$ is a condition in $\PPP_\kappa$.  Since  $\PPP_\theta$ satisfies the $\theta$-chain condition, we can find a function  $\map{g}{B}{\theta}$  and sequences $\seq{q_\kappa}{\kappa\in B}, \seq{\dot{r}_\kappa}{\kappa \in B}$ and $\vec{e}=\seq{\dot{e}_\kappa}{\kappa\in B}$ such that the following statements hold for all $\kappa\in B$:
 \begin{enumerate}[leftmargin=0.7cm]
   \item[(a)] $g(\kappa)>\kappa$ is inaccessible and $\dot{d}_\kappa$ is a $\PPP_{g(\kappa)}$-name. 
   
   \item[(b)] $q_\kappa$ is a condition in $\PPP_\kappa$ below $q$. 

 \item[(c)] $\dot{r}_{\kappa}$ is a $\PPP_\kappa$-name for a condition in $\dot{\QQQ}^{\HH{\kappa}}_\theta$ that is an element of $\HH{g(\kappa)}$. 
 
  \item[(d)] $\dot{e}_\kappa$ is a $\PPP_\kappa$-name for a subset of $\kappa$ with $\langle q_{\kappa},\dot{r}_\kappa\rangle\Vdash_{\PPP_\kappa*\dot{\QQQ}_\theta^{\HH{\kappa}}}\anf{\dot{d}_\kappa=\dot{e}_{\kappa}}$.
 \end{enumerate}
Given $\kappa\in B$, let $E_\kappa$ denote the set of all triples $\langle s,\beta, i\rangle\in \PPP_\kappa\times\kappa\times2\subseteq\HH{\kappa}$ with $$s\Vdash_{\PPP_\kappa}{\anf{\check{\beta}\in\dot{e}_{\kappa} ~ \longleftrightarrow ~ i=1}}.$$ Fix a bijection $\map{b}{\theta}{\HH{\theta}}$ with  $b[\kappa]=\HH{\kappa}$ for every inaccessible cardinal $\kappa\le\theta$, and define $\vec{d}=\seq{d_{\alpha}}{\alpha<\theta}$ to be the unique $\theta$-list with $$d_\alpha ~ = ~ \{\goedel{0}{0}\} ~ \cup ~ \{\goedel{b^{{-}1}(q_\alpha)}{1}\} ~ \cup ~ \Set{\goedel{b^{{-}1}(e)}{2}}{e\in E_\alpha}~ \subseteq ~ \alpha$$ for all $\alpha\in B$, and with $d_{\alpha}=\emptyset$ for all $\alpha\in\theta\setminus B$.  Next, let $\map{j}{M}{\HH{\vartheta}}$ be a small embedding for $\theta$ which witnesses the subtlety of $\theta$ with respect to $C$, $\vec{d}$ and $\{b,\dot{d},g,q,\dot{C}\}$, and let $\kappa$ denote the critical point of $j$. Then there is an $\alpha\in C\cap\kappa$ with $d_{\alpha}=d_{\kappa}\cap \alpha$. In this situation, the embedding $j$ witnesses that $\kappa$ is an element of $A$ and, by elementarity, $q\in\ran{j}$ implies that $q\in\HH{\kappa}$ and $\kappa\in B$. But then, $\goedel{0}{0}\in d_\kappa\cap\alpha$, and therefore $\alpha\in B$.  By Proposition \ref{proposition:LiftingSmallEmbeddings}, this shows that $\PPP_\kappa$ is a complete suborder of $\PPP_\theta$, $\PPP_\alpha$ is a complete suborder of $\PPP_\kappa$ and $\HH{\alpha}\in\calT_\kappa\subseteq\calT_\theta$. Moreover,  the above coherence implies that $q_{\alpha}=q_{\kappa}\in \PPP_\alpha$ and $E_{\alpha}\subseteq E_{\kappa}$. By elementarity, we have $g(\alpha)<\kappa$, and therefore the above remarks show that $\dot{r}_\alpha$ is also a $\PPP_\alpha$-name for a condition in $\dot{\QQQ}_\kappa^{\HH{\alpha}}$. Hence, there is a condition $u$ in $\PPP_\kappa$ satisfying  $$D^{\HH{\alpha}}_\kappa(u) ~ \leq_{\PPP_\alpha*\dot{\QQQ}_\kappa^{\HH{\alpha}}} ~ (q_\alpha,\dot{r}_\alpha).$$ Then $u\leq_{\PPP_\kappa}u\cap\HH{\alpha}\leq_{\PPP_\kappa}q_\alpha=q_\kappa$,  and we may find a condition $v$ in $\PPP_\theta$ with $$D^{\HH{\kappa}}_\theta(v) ~ \leq_{\PPP_\kappa*\dot{\QQQ}_\theta^{\HH{\kappa}}} ~ (u,\dot{r}_\kappa).$$  Let $H$ be $\PPP_\theta$-generic over $\VV$ with $v\in H$. Then $q\in H$, $u\in H_j=H\cap\HH{\kappa}$ and the above choices ensure that $\dot{d}_\kappa^H=\dot{e}_\kappa^{H_j}$ and $\dot{d}_\alpha^H=\dot{e}_\alpha^{H\cap\HH{\alpha}}$. 

 \begin{claim*}
  $\dot{d}_\alpha^{H}=\dot{d}^H_{\kappa}\cap \alpha$. 
 \end{claim*}

 \begin{proof}[Proof of the Claim]
  Pick $\beta\in\dot{d}_\alpha^{H} $. By the above computations, we have $\beta\in\dot{e}_\alpha^{H\cap\HH{\alpha}}$, and hence there is an $s\in H\cap\HH{\alpha}\subseteq H_j$ with $\langle s,\beta, 1\rangle \in E_\alpha\subseteq E_\kappa$. But then, $\beta\in \dot{e}_\kappa^{H_j}=\dot{d}^H_{\kappa}$. In the other direction, pick $\beta \in \alpha \setminus \dot{d}_\alpha^{H}$. Then $\beta \in \alpha \setminus \dot{e}_\alpha^{H\cap\HH{\alpha}}$, and there is an $s\in H\cap\HH{\alpha}$ with $\langle s,\beta,0\rangle \in E_{\alpha}\subseteq E_\kappa$.  Thus $\beta\notin\dot{e}^{H_j}_\kappa=\dot{e}_\kappa^H$.
 \end{proof}

Let $\map{j_H}{M[H_j]}{\HH{\vartheta}^{\VV[H]}}$ be the small embedding for $\theta$ provided by an application of Proposition \ref{proposition:LiftingSmallEmbeddings}, and set $N=\HH{\vartheta}^{\VV[H_j]}$. As in the first case, we know that $\dot{d}^H,\dot{x}^H, \dot{C}^H\in\ran{j_H}$, the pair $(N,\HH{\vartheta}^{\VV[H]})$ satisfies the $\sigma$-approximation property, $\kappa=\omega_2^N$ and  $\Poti{\kappa}{\omega_1}^N\subseteq\HH{\kappa}^N\subseteq M[H_j]$. By the above claim, this shows that $j_H$ and $N$ witness that $\theta$ is internally AP subtle in $\VV[H]$ with respect to $\dot{d}^H, \dot{x}^H$ and $\dot{C}^H$.  In particular, there is a condition in $H$ below $q$ that forces this statement.

This density argument shows that, in $\VV[G]$, we can find a small embedding $\map{j}{M}{\HH{\vartheta}}$ for $\theta$ that  witnesses the internal AP subtlety of $\theta$ with respect to $\dot{d}^G,\dot{x}^G,\dot{C}^G$. In particular, these computations show that (ii) holds. 
\end{proof}

Next, we turn our attention towards the hierarchy of ineffable cardinals. Remember that, given a regular uncountable cardinal $\theta$ and a cardinal $\lambda\geq\theta$, the cardinal $\theta$ is \emph{$\lambda$-ineffable} if for every $\Poti{\lambda}{\theta}$-list $\vec{d}=\seq{d_a}{a\in\Poti{\lambda}{\theta}}$, there exists a subset $D$ of $\lambda$ such that the set $\Set{a\in\Poti{\lambda}{\theta}}{d_a=D\cap a}$ is stationary in $\Poti{\lambda}{\theta}$. This large cardinal property has the following small embedding characterization (see {\cite[Lemma 5.5, 5.9 \& 3.4]{SECFLC}}).

\begin{lemma}\label{lemma:LambdaIneffableSmallChar}
The following statements are equivalent for all cardinals $\theta\leq\lambda$ satisfying $\lambda=\lambda^{{<}\theta}$: 
  \begin{enumerate}[leftmargin=0.7cm]
    \item $\theta$ is $\lambda$-ineffable.
   
    \item For all sufficiently large cardinals $\vartheta$, every $\Poti{\lambda}{\theta}$-list $\vec{d}=\seq{d_a}{a\in\Poti{\lambda}{\theta}}$ and all $x\in\HH{\vartheta}$, there is a small embedding $\map{j}{M}{\HH{\vartheta}}$ for $\theta$ and $\delta\in M\cap\theta$ such that 
$\crit j$ is inaccessible, $\Poti{\delta}{\crit j}\subseteq M$, $\vec{d},x\in\ran{j}$, $j(\delta)=\lambda$ and $j^{{-}1}[d_{j[\delta]}]\in M$.
\end{enumerate}
\end{lemma}

Note that the assumption $\lambda=\lambda^{{<}\theta}$ is only needed to ensure that $\Poti{\delta}{\crit{j}}$ is a subset of the domain of the given small embedding. Other degrees of ineffability can also be characterized by small embeddings by removing this assumption from the above equivalence.

The above characterization again gives rise to an internal large cardinal principle:

\begin{definition}\label{definition:internallyineffable}
 Given cardinals $\theta\leq\lambda$, the cardinal $\theta$ is \emph{internally AP $\lambda$-ineffable} if for all sufficiently large regular cardinals $\vartheta$, all $x\in\HH{\vartheta}$, and every $\Poti{\lambda}{\theta}$-list $\vec{d}=\seq{d_a}{a\in\Poti{\lambda}{\theta}}$, there is a small embedding $\map{j}{M}{\HH{\vartheta}}$ for $\theta$, an ordinal $\delta\in M\cap\theta$ and a transitive model $N$ of $\ZFC^-$ such that $j(\delta)=\lambda$, $\vec{d},x\in\ran j$, and the following statements hold: 
 \begin{enumerate}[leftmargin=0.7cm]
   \item $N\subseteq\HH{\vartheta}$ and the pair $(N,\HH{\vartheta})$ satisfies the $\sigma$-approximation property. 

   \item $M\in N$ and $\Poti{\delta}{\omega_1}^N\subseteq M$. 

  \item If $j^{{-}1}[d_{j[\delta]}]\in N$, then $j^{{-}1}[d_{j[\delta]}]\in M$. 
  \end{enumerate} 
\end{definition}

Analogous to the above study of internal subtlety, the consequences of this principle can be studied through combinatorial concepts introduced by Wei{\ss} in \cite{weissthesis} and \cite{MR2959668}.

\begin{definition}\label{definition:ISP}
 Let $\theta$ be an uncountable regular cardinal and let $\lambda\geq\theta$ be a cardinal. 
 \begin{enumerate}[leftmargin=0.7cm]
  \item A $\Poti{\lambda}{\theta}$-list $\seq{d_a}{a\in\Poti{\lambda}{\theta}}$ is \emph{slender} if for every sufficiently large cardinal $\vartheta$, there is a club $C$ in $\Poti{\HH{\vartheta}}{\theta}$ with $b\cap d_{X\cap\lambda}\in X$ for all $X\in C$ and all $b\in X\cap\Poti{\lambda}{\omega_1}$. 

  \item $\ISP(\theta,\lambda)$ is the statement that for every slender $\Poti{\lambda}{\theta}$-list $\seq{d_a}{a\in\Poti{\lambda}{\theta}}$, there exists $D\subseteq\lambda$ such that the set $\Set{a\in\Poti{\lambda}{\theta}}{d_a=D\cap a}$ is stationary in $\Poti{\lambda}{\theta}$. 
 \end{enumerate}
\end{definition}

\begin{lemma}\label{lemma:InternalIneffableISP}
  If $\theta$ is an internally AP $\lambda$-ineffable cardinal, then $\ISP(\theta,\lambda)$ holds. 
\end{lemma}

\begin{proof}
   Fix a slender $\Poti{\lambda}{\theta}$-list $\vec{d}=\seq{d_a}{a\in\Poti{\lambda}{\theta}}$. 
 Then we can find a sufficiently large cardinal $\vartheta$ such that there is a function $\map{f}{\Poti{\HH{\vartheta}}{\omega}}{\Poti{\HH{\vartheta}}{\theta}}$ with the property that $\clo{f}$ is a club in $\Poti{\HH{\vartheta}}{\theta}$ witnessing the slenderness of $\vec{d}$. 
Let $\vartheta^\prime$ be a sufficiently large regular cardinal such that there is a small embedding $\map{j}{M}{\HH{\vartheta^\prime}}$ with $f\in\ran{j}$, $\delta\in M$ and a transitive $\ZFC^-$-model $N$ witnessing the internal AP $\lambda$-ineffability of $\theta$ with respect to $\vec{d}$. 
Pick $\varepsilon\in M$ with $j(\varepsilon)=\vartheta$. We then have $X=j[\HH{\varepsilon}^M]\in\clo{f}$.

 Assume for a contradiction that $j^{{-}1}[d_{j[\delta]}]\notin N$. 
 Then the $\omega_1$-approximation property yields an element $x$ of $\Poti{\delta}{\omega_1}^N$  with $x\cap j^{{-}1}[d_{j[\delta]}]\notin N$. 
 Then our assumptions imply that $x$ is an element of $\Poti{\delta}{\omega_1}^M$. But then $j(x)\in X\cap\Poti{\lambda}{\omega_1}$ and the slenderness of $\vec{d}$ implies that $j(x)\cap d_{j[\delta]}\in X$. Then we can conclude that $$x\cap j^{{-}1}[d_{j[\delta]}] ~ = ~ j^{{-1}}[j(x)\cap d_{j[\delta]}] ~ = ~ j^{{-}1}(j(x)\cap d_{j[\delta]}) ~ \in ~  M ~ \subseteq ~ N,$$ a contradiction. 
 These computations show that $j^{{-}1}[d_{j[\delta]}]\in N$ and our assumptions imply that this set is also an element of $M$. 
 Define $D=j(j^{{-}1}[d_{j[\delta]}])$ and $S=\Set{a\in\Poti{\lambda}{\theta}}{d_a=D\cap a}$. 
 Assume, towards a contradiction, that $S$ is not stationary in $\Poti{\lambda}{\theta}$. By elementarity, there is a function $\map{f_0}{\Poti{\delta}{\omega}}{\Poti{\delta}{\crit{j}}}$ in $M$ such that  $\clo{j(f_0)}\cap S=\emptyset$. But then $j[\delta]\in\clo{j(f_0)}\cap S$, a contradiction.  
\end{proof}

\begin{corollary}\label{corollary:InaccIneffableInternalIneffable}
 The following statements are equivalent for every inaccessible cardinal $\theta$ and every cardinal $\lambda\geq\theta$ satisfying $\lambda=\lambda^{{<}\theta}$: 
 \begin{enumerate}[leftmargin=0.7cm]
  \item $\theta$ is a $\lambda$-ineffable cardinal. 

  \item $\theta$ is an internally AP $\lambda$-ineffable cardinal. 
 \end{enumerate}
\end{corollary}

\begin{proof}
 Lemma \ref{lemma:LambdaIneffableSmallChar} directly shows that (i) implies (ii). In the other direction, the results of \cite{MR0327518} and {\cite[Section 2]{MR2959668}} show that an inaccessible cardinal $\theta$ is $\lambda$-ineffable if and only if $\ISP(\theta,\lambda)$ holds. Therefore, Lemma \ref{lemma:InternalIneffableISP} shows that (ii) implies (i). 
\end{proof}

The next theorem yields a characterization of the set of all $\lambda$-ineffable cardinals $\theta$ with  $\lambda=\lambda^{{<}\theta}$. In particular, it shows that Neeman's pure side condition forcing can be used to  characterize  the class of all ineffable cardinals. 
The above corollary already shows that these characterizations are strong.

\begin{theorem}\label{ineffabletheorem}
  The following statements are equivalent for every inaccessible cardinal $\theta$ and every cardinal $\lambda$ with $\lambda=\lambda^{{<}\theta}$:
\begin{enumerate}[leftmargin=0.7cm]
  \item $\theta $ is a $\lambda$-ineffable cardinal. 

  \item $\mathbbm{1}_{\PPP_\theta}\Vdash\anf{\textit{$\omega_2$ is interally AP $\check{\lambda}$-ineffable}}$.  

 \item $\mathbbm{1}_{\PPP_\theta}\Vdash\ISP(\omega_2,\check{\lambda})$.  
\end{enumerate}
\end{theorem}

\begin{proof}
  First, assume that (iii) holds. Since Corollary \ref{corollary:SideConditionForcingApproximationProperty} shows that $\PPP_\theta$ satisfies the $\sigma$-approximation property, and Lemma \ref{chain_condition} implies that $\PPP_\theta$ satisfies the $\theta$-chain condition, we can combine our assumptions with {\cite[Theorem 6.3]{MR2838054}} and {\cite[Proposition 2.2]{MR2959668}} to conclude that $\theta$ is $\lambda$-ineffable. 

 Now, assume that (i) holds. Let $\dot{d}$ be a $\PPP_\theta$-name for a $\Poti{\lambda}{\theta}$-list, and let $\dot{x}$ be an arbitrary $\PPP_\theta$-name. For every $a\in\Poti{\lambda}{\theta}$, let $\dot{d}_a$ be a nice $\PPP_\theta$-name for the component of $\dot{d}$ that is indexed by $\check a$. Fix a bijection  $\map{b}{\theta}{\HH{\theta}}$ such that $b[\kappa]=\HH{\kappa}$ holds for every inaccessible cardinal $\kappa \le \theta$. Pick a regular cardinal $\vartheta>2^\lambda$ that is sufficiently large with respect to Lemma \ref{lemma:LambdaIneffableSmallChar}, and which satisfies $\dot{d},\dot{x}\in \HH{\vartheta}$. Define $A$ to be the set of all  $a\in \Poti{\lambda}{\theta}$ for which there exists a small embedding $\map{j}{M}{\HH{\vartheta}}$ for $\theta$ and $\delta\in M\cap \theta$ with $j(\delta)=\lambda$, $a=j[\delta]$, $\kappa_a =\crit{j}=a\cap \theta$ inaccessible, $\Poti{\delta}{\kappa_a}\subseteq M$ and $b,\dot{d},\dot{x}\in \ran{j}$. Let $G$ be $\PPP_\theta$-generic over $\VV$.

Assume first that there exists $a\in\Poti{\lambda}{\theta}^\VV$, a small embedding $\map{j}{M}{\HH{\vartheta}^\VV}$ for $\theta$ in $\VV$, and  $\delta\in M\cap \theta$ such that $j$ and $\delta$ witness that $a$ is an element of $A$, and such that $\dot{d}^G_{a}\notin \VV[G_j]$. Let $\map{j_G}{M[G_j]}{\HH{\vartheta}^{\VV[G]}}$ be the lifting of $j$ provided by Proposition \ref{proposition:LiftingSmallEmbeddings}, and set $N=\HH{\vartheta}^{\VV[G_j]}$. Then $\dot{d}^G, \dot{x}^G\in \ran{j_G}$, Corollary \ref{corollary:sigma_approximation} shows that the pair $(N,\HH{\vartheta}^{\VV[G]})$ satisfies the $\sigma$-approximation property and, by Lemma \ref{chain_condition}, $\Poti{\delta}{\kappa_a}^{\VV} \subseteq M$ implies that $\Poti{\delta}{\kappa_a}^N \subseteq M[G_j]$. Since $\dot{d}^G_{a}\notin N$, we can conclude that $j_G$, $\delta$ and $N$ witness the internal AP $\lambda$-ineffability of $\theta$ with respect to $\dot{d}^G$ and $\dot{x}^G$ in $\VV[G]$.

Otherwise, assume that whenever $\map{j}{M}{\HH{\vartheta}}^\VV$ is a small embedding for $\theta$ in $\VV$ that witnesses that some $a\in\Poti{\lambda}{\theta}^\VV$ is an element of $A$, then $\dot{d}^G_a\in \VV[{G}_j]$. Pick a condition $p$ in $G$ which forces this statement. Work in $\VV$, fix a condition  $q$ below $p$ in $\PPP_\theta$, and define $B$ to be the set of all $a\in A$ such that $q\in \PPP_{\kappa_a}$. By our assumption and by Lemma \ref{chain_condition}, we can find sequences $\seq{q_a}{a\in B}$, $\seq{\dot{r}_a}{a\in B}$ and $\seq{\dot{e}_a}{a\in B}$, such that the following statements hold for all $a\in B$: 
 \begin{enumerate}[leftmargin=0.7cm]
  \item[(a)] $q_a$ is a condition in $\PPP_{\kappa_a}$	below $q$. 

  \item[(b)] $\dot{r}_a$ is a $\PPP_{\kappa_a}$-name for a condition in $\dot{\QQQ}_\theta^{\HH{\kappa_a}}$. 

  \item[(c)] $\dot{e}_a$ is a $\PPP_{\kappa_a}$-nice name for a subset of $a$ with $\langle q_a,\dot{r}_a\rangle\Vdash_{\PPP_{\kappa_a}*\dot{\QQQ}_\theta^{\HH{\kappa_a}}} \anf{\dot{d}_a=\dot{e}_a}$. 
 \end{enumerate}
Given $a\in B$, we have $b^{{-1}}[\PPP_{\kappa_a}]\subseteq b^{{-}1}[\HH{\kappa_a}]=\kappa_a\subseteq a$, and elementarity implies that the set $a$ is closed under $\goedel{\cdot}{\cdot}$. This shows that there is a unique $\Poti{\lambda}{\theta}$-list $\vec{d}=\seq{d_a}{a\in \Poti{\lambda}{\theta}}$ with $$d_a ~ = ~ \Set{\goedel{b^{{-}1}(s)} {\beta} }{\langle\check{\beta}, s \rangle \in \dot{e}_a} ~ \subseteq ~ a$$ for all  $a\in B$, and with $d_a=\emptyset$  for all $a\in \Poti{\lambda}{\theta}\setminus B$.   Fix a small embedding $\map {j}{M}{\HH{\vartheta}}$ for $\theta$ and $\delta \in M \cap \theta$ that witness the $\lambda$-ineffability of $\theta$ with respect to $\vec{d}$ and $\{b,\dot{d},q,\dot{x}\}$, as in Lemma \ref{lemma:LambdaIneffableSmallChar}. Then $j$ and $\delta$ witness that $j[\delta]\in B$. Pick $u$ in $\PPP_\theta$ with 
$$D_\theta^{\HH{\kappa_{j[\delta]}}}(u) ~ \leq_{\PPP_{\kappa_{j[\delta]}}*\dot{\QQQ}_\theta^{\HH{\kappa_{j[\delta]}}}} ~ \langle q_{j[\delta]},\dot{r}_{j[\delta]}\rangle,$$ and let $H$ be  $\PPP_\theta$-generic over $\VV$ with $u\in H$. Since $q_{j[\delta]}\in H_j$, we have  $\dot{d}^H_{j[\delta]}=\dot{e}^{H_j}_{j[\delta]}$.  Note that this implies that for all $\gamma<\delta$, we have  $\gamma\in j^{{-}1}[\dot{d}^H_{j[\delta]}]$ if and only if there is an $s\in H_j$ with $\goedel{b^{{-}1}(s)}{j(\gamma)}\in d_{j[\delta]}$.  Observe that $b\restriction\kappa_{j[\delta]}\in M$,   $j(b\restriction\kappa_{j[\delta]})=b$ and $j\restriction H_j=\id_{H_j}$. Hence, $j^{{-}1}[\dot{d}_{j[\delta]}^H]$ is equal to the set of all $\gamma<\delta$ with the property that there is an $s\in H_j$ with
 $\goedel{(b\restriction\kappa_{j[\delta]})^{{-}1}(s)}{\gamma}\in j^{{-}1}[d_{j[\delta]}]$.  Since the above choices ensure that $j^{{-}1}[d_{j[\delta]}]\in M$, we can conclude that $j^{{-}1}[\dot{d}_{j[\delta]}^H]$ is an element of $M[H_j]$. Let $\map{j_H}{M[H_j]}{\HH{\vartheta}^{\VV[H]}}$ denote the lifting of $j$ provided by Proposition \ref{proposition:LiftingSmallEmbeddings}, and set $N=\HH{\vartheta}^{\VV[H_j]}$. As above, we have $\dot{d}^H,\dot{x}^H\in\ran{j_H}$, the pair $(N,\HH{\vartheta}^{\VV[H]})$ satisfies the $\sigma$-approximation property, and $\Poti{\delta}{\kappa_{j[\delta]}}^N\subseteq M[H_j]$. Since $j^{{-}1}[\dot{d}_{j[\delta]}^H]\in M[H_j]$, we can conclude that $j_H$ and $\delta$ witness that $\theta$ is internally AP $\lambda$-ineffable with respect to $\dot{d}^H$ and $\dot{x}^H$ in $\VV[H]$.

Using a standard density argument, the above computations allow us to conclude that $\theta$ is internally AP $\lambda$-ineffable in $\VV[G]$. In particular, these arguments show that (i) implies (ii).  
\end{proof}

In the remainder of this section, we will use the above results to strongly characterize supercompactness through the validity of internal supercompactness in generic extension. 
This characterization is based on the following equivalence provided by the results of \cite{MR2838054}.

\begin{proposition}\label{proposition:InternalSupercompactISP}
 Given an uncountable regular cardinal $\theta$, the cardinal $\theta$ is   
 internally AP supercompact if and only if $\ISP(\theta,\lambda)$ holds for all cardinals $\lambda\geq\theta$. 
\end{proposition}

\begin{proof}
 First, assume that  $\ISP(\theta,\lambda)$ holds for all cardinals $\lambda\geq\theta$. 
Fix some regular cardinal $\vartheta>\theta$, $x\in\HH{\vartheta}$ and a cardinal $\vartheta^\prime>\betrag{\HH{\vartheta}}$. 
 %
By {\cite[Proposition 3.2]{MR2838054}}, the fact that $\ISP(\theta,\betrag{\HH{\vartheta^\prime}})$ holds implies that there is a stationary subset of $\Poti{\HH{\vartheta^\prime}}{\theta}$ consisting of  elementary submodels $X$ of $\HH{\vartheta^\prime}$ of cardinality less than $\theta$ with $X\cap\theta\in\theta$ and the property that, if $N$ denotes the transitive collapse of $X$, then the pair $(N,\HH{\vartheta^\prime})$ satisfies the $\sigma$-approximation property. 
Pick such a submodel $X$ with $\theta,\vartheta,x\in X$. Let $\map{\pi}{X}{N}$ denote the corresponding collapsing map and set $\kappa=\pi(\vartheta)<\theta$.  
Define  $M=\HH{\kappa}^N$ and $\map{j=\pi^{{-}1}\restriction M}{M}{\HH{\vartheta}}$.
Then $j$ is a small embedding for $\theta$ with $x\in\ran{j}$ and  $N$ is a transitive model of $\ZFC^-$ with $N\subseteq\HH{\vartheta}$, $\kappa<\theta$ is a cardinal in $N$ and $M=\HH{\kappa}^N$. 
   Since the above constructions ensure that the pair $(N,\HH{\vartheta})$ has the $\sigma$-approximation property, we  can conclude that $\theta$ is internally AP supercompact with respect to $x$.

 In the other direction, assume that $\theta$ is internally AP supercompact. Fix a set $x$, a cardinal $\lambda\geq\theta$ and a $\Poti{\lambda}{\theta}$-list $\vec{d}=\seq{d_a}{a\in\Poti{\lambda}{\theta}}$. Pick a small embedding $\map{j}{M}{\HH{\vartheta}}$ for $\theta$ and a model $N$ that witnesses that $\theta$ is internally AP supercompact with respect to $\langle\vec{d},x\rangle$. Then there are $N$-cardinals $\delta<\kappa$ with $M=\HH{\kappa}^N$ and $j(\delta)=\lambda$.  
 If $j^{{-}1}[d_{j[\delta]}]\in N$, then $j^{{-}1}[d_{j[\delta]}]\in\POT{\delta}^N\subseteq\HH{\kappa}^N=M$. This shows that $j$, $\delta$ and $N$ witness that $\theta$ is internally AP $\lambda$-ineffable with respect to $\vec{d}$ and $x$. 
 In this situation, Lemma \ref{lemma:InternalIneffableISP} allows us to conclude that $\ISP(\theta,\lambda)$ holds for all cardinals $\lambda\geq\theta$. 
\end{proof}

In combination with {\cite[Theorem 4.8]{MR2838054}}, the above proposition shows that $\PFA$ implies that $\omega_2$ is internally AP supercompact and hence possesses all internal large cardinal properties discussed in this paper.

The proof of the following corollary combines the above observations to provide a characterization of supercompactness that is based on its internal version.

\begin{corollary}\label{corollary:CharSupercompactISP}
 The following statements are equivalent for every inaccessible cardinal $\theta$: 
 \begin{enumerate}[leftmargin=0.7cm]
  \item $\theta$ is a supercompact cardinal. 

  \item $\mathbbm{1}_{\PPP_\theta}\Vdash\anf{\textit{$\omega_2$ is internally AP supercompact}}$. 

  \item $\mathbbm{1}_{\PPP_\theta}\Vdash\anf{\textit{$\ISP(\omega_2,\lambda)$ holds for all cardinals $\lambda\geq\omega_2$}}$. 
 \end{enumerate}
\end{corollary}

\begin{proof}
 Since the results of \cite{MR0327518} show that a cardinal $\theta$ is supercompact if and only if it is $\lambda$-ineffable for all $\lambda\geq\theta$, the equivalence between (i) and (iii) follows directly from Theorem \ref{ineffabletheorem}, the fact that there is a proper class of cardinals $\lambda$ satisfying $\lambda=\lambda^{{<}\theta}$ and the fact that $\ISP(\theta,\lambda_1)$ implies $\ISP(\theta,\lambda_0)$ for all cardinals $\theta\leq\lambda_0\leq\lambda_1$. 
 The equivalence between (ii)  and (iii) is given by Proposition \ref{proposition:InternalSupercompactISP}. 
\end{proof}

 As mentioned in the proof of Corollary \ref{corollary:InaccIneffableInternalIneffable}, an inaccessible cardinal $\theta$ is supercompact if and only if $\ISP(\theta,\lambda)$ holds for all cardinals $\lambda\geq\theta$. 
In combination with Proposition \ref{proposition:InternalSupercompactISP}, this directly shows that the characterization of supercompactness provided by the above corollary is strong.


\section[Supercompact cardinals]{$\gamma$-supercompact cardinals}\label{section:supercompact}

In this section and the next, we show that it is possible to use Neeman's pure side condition forcing to characterize levels of supercompactness, almost huge, and super almost huge cardinals. Since either no small embedding characterizations for these properties are known, or the existing small embedding characterizations are not suitable for our purposes, our characterizations instead make use of the classical concept of \emph{generic elementary embeddings}. 
In the following, we will provide an alternative characterization of supercompactness. Afterwards, we will use ideas from these proofs to characterize even stronger large cardinal properties. 
The following lemma lies at the heart of these results. Its proof heavily relies on the concepts and results presented in  {\cite[Section 6]{MR2838054}}. Remember that, given transitive classes $M\subseteq N$ and $\theta\in M$ a cardinal in $N$, the pair $(M,N)$ satisfies the \emph{$\theta$-cover property} if for every $A\in N$ with $A\subseteq M$ and $\betrag{A}^N<\theta$, there exists $B\in M$ with $A\subseteq B$ and $\betrag{B}^M<\theta$.

\begin{lemma}\label{lemma:FilterInGroundModel}
 Let $\VV[G]$ be a generic extension of the ground model $\VV$, let $\VV[G,H]$ be a generic extension of $\VV[G]$ and let $\map{j}{\VV[G]}{M}$ be an elementary embedding definable in $\VV[G,H]$ with critical point $\theta$. Assume that the following statements hold: 
 \begin{enumerate}[leftmargin=0.7cm]
  \item $\theta$ is an inaccessible cardinal in $\VV$. 

  \item The pair $(\VV,\VV[G])$ satisfies the $\sigma$-approximation and the $\theta$-cover property. 

  \item The pair $(\VV[G],\VV[G,H])$ satisfies the $\sigma$-approximation property. 
 \end{enumerate}
In this situation, if $\theta\leq\gamma<j(\theta)$ is an ordinal with $j[\gamma]\in M$, then  $j[\gamma]\in j(\Poti{\gamma}{\theta}^\VV)$, and the set $$\calU ~ = ~ \Set{A\in\POT{\Poti{\gamma}{\theta}}^\VV}{j[\gamma]\in j(A)}$$ is an element of $\VV$. 
\end{lemma}

\begin{proof}
 The above assumptions imply that $\omega_1^\VV=\omega_1^{\VV[G]}=\omega_1^{\VV[G,H]}$, and hence $\theta$ is an uncountable regular cardinal greater than $\omega_1$ in $\VV[G]$.

 \begin{claim*}
  $\calU\in\VV[G]$. 
 \end{claim*}
 
 \begin{proof}[Proof of the Claim]
  Assume, towards a contradiction, that the set $\calU$ is not an element of $\VV[G]$. Then there is $u\in\Poti{\POT{\Poti{\gamma}{\theta}}^\VV}{\omega_1}^{\VV[G]}$ with $\calU\cap u\notin\VV[G]$. Define $$\Map{d}{\Poti{\gamma}{\theta}^{\VV[G]}}{\POT{u}^{\VV[G]}}{x}{\Set{A\in u}{x\in A}}.$$ By our assumptions, there is $c\in\Poti{\POT{\Poti{\gamma}{\theta}}}{\theta}^\VV$ with $u\subseteq c$. In the following, let $\map{a}{\Poti{\gamma}{\theta}^{\VV[G]}}{\POT{c}^\VV}$ denote the unique function with $a(x)=\Set{A\in c}{x\in A}$ for all $x\in\dom{a}\cap\VV$ and $a(x)=\emptyset$ for all $x\in\dom{a}\setminus\VV$. Since $d(x)=\emptyset$ for all $x\in\Poti{\gamma}{\theta}^{\VV[G]}\setminus\VV$, we then have $d(x)=a(x)\cap u$ for all $x\in \Poti{\gamma}{\theta}^{\VV[G]}$ and $$\ran{d} ~ = ~ \Set{a(x)\cap u}{x\in \Poti{\gamma}{\theta}^{\VV[G]}} ~ \subseteq ~ \Set{u\cap y}{y\in\POT{c}^\VV}.$$ Since $\theta$ is inaccessible in $\VV$, this implies that $\ran{d}$ has cardinality less than $\theta$ in $\VV[G]$ and there 
is a bijection $\map{b}{\mu}{\ran{d}}$ in $\VV[G]$ for some $\mu<\theta$. In this situation, we have $j[\gamma]\in j(\Poti{\gamma}{\theta}^{\VV[G]})$, and elementarity yields an $\alpha<\mu$ with $j(b)(\alpha)=j(d)(j[\gamma])$. But then $$j[(b(\alpha))] ~ = ~ j(b)(\alpha) ~ = ~ j(d)(j[\gamma]) ~ = ~ \Set{j(A)}{A\in u, ~ j[\gamma]\in j(A)} ~ = ~ j[(\calU\cap u)],$$ and this implies that $\calU\cap u= b(\alpha)\in\VV[G]$, a contradiction. 
 \end{proof}

 \begin{claim*}
  $j[\gamma]\in j(\Poti{\gamma}{\theta}^\VV)$. 
 \end{claim*}
 
 \begin{proof}[Proof of the Claim]
  Assume, towards a contradiction, that the set $j[\gamma]$ is not an element of $j(\Poti{\gamma}{\theta}^\VV)$.   By our assumptions on $\VV$ and $\VV[G]$, there is a function $\map{a}{\Poti{\gamma}{\theta}^{\VV[G]}}{\Poti{\gamma}{\omega_1}^\VV}$ in $\VV[G]$ with $a(x)=\emptyset$ for all $x\in\dom{a}\cap\VV$ and $a(x)\cap x\notin\VV$ for all $x\in\dom{a}\setminus\VV$. Define $$\Map{d}{\Poti{\gamma}{\theta}^{\VV[G]}}{\Poti{\gamma}{\omega_1}^{\VV[G]}}{x}{a(x)\cap x},$$ and set $D=\Set{\alpha<\gamma}{j(\alpha)\in j(d)(j[\gamma])}$. Then our assumption and elementarity imply that $D\neq\emptyset$.

  \begin{subclaim*}
   $D\in\VV[G]$. 
  \end{subclaim*}
  
  \begin{proof}[Proof of the Subclaim]
   Assume, towards a contradiction, that $D$ is not an element of $\VV[G]$. Then there is $u\in\Poti{\gamma}{\omega_1}^{\VV[G]}$ with $D\cap u\notin\VV[G]$. Define $$R ~ = ~ \Set{d(x)\cap u}{u\subseteq x\in\Poti{\gamma}{\theta}^{\VV[G]}},$$ and fix $c\in\Poti{\gamma}{\theta}^\VV$ with $u\subseteq c$. Then $$R ~ = ~ \Set{a(x)\cap c\cap u}{u\subseteq x\in\Poti{\gamma}{\theta}^{\VV[G]}} ~ \subseteq ~ \Set{u\cap y}{y\in\Poti{c}{\omega_1}^\VV},$$ and, since $\theta$ is inaccessible in $\VV$, there is a bijection $\map{b}{\mu}{R}$ in $\VV[G]$ with $\mu<\theta$.

 We now have $j(d)(j[\gamma])\cap j(u)\in j(R)$, because $j(u)=j[u]\subseteq j[\gamma]\in j(\Poti{\gamma}{\theta}^{\VV[G]})$. Hence there is an $\alpha<\mu$ with $$j[(b(\alpha))] ~ = ~ j(b)(\alpha) ~ = ~ j(d)(j[\gamma])\cap j(u) ~ = ~ j[(D\cap u)],$$ and this implies that $D\cap u=b(\alpha)\in\VV[G]$, a contradiction.  
  \end{proof}

  Define $U=\Set{x\in\Poti{\gamma}{\theta}^{\VV[G]}}{d(x)=D\cap x}\in\VV[G]$.

  \begin{subclaim*}
   In $\VV[G]$, the set $U$ is unbounded in $\Poti{\gamma}{\theta}$. 
  \end{subclaim*}
  
   \begin{proof}[Proof of the Subclaim]
    We have $j(d)(j[\gamma])=j[D]=j(D)\cap j[\gamma]$, and this shows that $j[\gamma]\in j(U)$. Now, if $x\in\Poti{\gamma}{\theta}^{\VV[G]}$, then $j(x)=j[x]\subseteq j[\gamma]\in j(U)$, and hence elementarity yields a $y\in U$ with $x\subseteq y$. 
  \end{proof}

  Now, work in $\VV[G]$ and use our assumptions together with the last claim to construct a sequence $\seq{x_\alpha}{\alpha\le\omega_1}$ of elements of $U$ and a sequence $\seq{y_\alpha}{\alpha\le\omega_1}$ of elements of $\Poti{\gamma}{\theta}^\VV$, such that $d(x_0)\neq\emptyset$, and such that $$\bigcup\Set{y_{\bar{\alpha}}}{\bar{\alpha}<\alpha} ~ \subseteq ~ x_\alpha ~ \subseteq ~  y_\alpha$$ for all $\alpha\le\omega_1$.  Then we have $$d(x_{\bar{\alpha}}) ~ = ~ D\cap x_{\bar{\alpha}} ~ \subseteq ~ D\cap x_\alpha ~ = ~ d(x_\alpha)$$ for all $\bar{\alpha}\leq\alpha\le\omega_1$. Since $d(x_{\omega_1})$ is a countable set, this implies that there is an $\alpha_*<\omega_1$ with $d(x_{\alpha_*})=d(x_\alpha)$ for all $\alpha_*\leq\alpha\le\omega_1$. Then 
  \begin{equation*}
   \begin{split}
     d(x_{\alpha_*}) ~  & = ~ d(x_{\alpha_*+1})\cap x_{\alpha_*} ~ \subseteq  ~ a(x_{\alpha_*+1})\cap y_{\alpha_*}   \\
 & \subseteq ~ a(x_{\alpha_*+1})\cap x_{\alpha_*+1} ~ = ~ d(x_{\alpha_*+1}) ~ = ~ d(x_{\alpha_*})
   \end{split}
  \end{equation*}
  and therefore $\emptyset\neq d(x_{\alpha_*})=a(x_{\alpha_*+1}) \cap y_{\alpha_*}\in\VV$, a contradiction. 
 \end{proof}

 Assume, towards a contradiction, that $\calU$ is not an element of $\VV$. Since $\calU\in\VV[G]$, this implies that there is a $u\in\Poti{\POT{\Poti{\gamma}{\theta}}}{\omega_1}^\VV$ with $\calU\cap u\notin\VV$. Define $$\Map{d}{\Poti{\gamma}{\theta}^\VV}{\POT{u}^\VV}{x}{\Set{A\in u}{x\in A}}.$$ Since $\theta$ is inaccessible in $\VV$, we can find a bijection $\map{b}{\mu}{\ran{d}}$ in $\VV$ with $\mu<\theta$. By the above claim, we have $j[\gamma]\in j(\Poti{\gamma}{\theta}^\VV)$ and hence there is an $\alpha<\mu$ with $j(d)(j[\gamma])=j(b)(\alpha)$. But then $$j[(b(\alpha))] ~ = ~ j(b)(\alpha) ~ = ~ j(d)(j[\gamma]) ~ = ~ \Set{j(A)}{A\in u, ~ j[\gamma]\in j(A)} ~ = ~ j[(\calU\cap u)],$$ and this implies that $\calU\cap u=b(\alpha)\in\VV$, a contradiction.  
\end{proof}

We now study typical situations in which the assumptions of Lemma \ref{lemma:FilterInGroundModel} are satisfied.

\begin{definition}\label{definition:GenGammaSupercompact}
  Given an uncountable regular cardinal $\theta$ and an ordinal $\gamma\geq\theta$, we say that a partial order \emph{$\PPP$ witnesses that $\theta$ is generically $\gamma$-supercompact} if there is a $\PPP$-name $\dot{\calU}$ such that $\dot{\calU}^G$ is a fine, $\VV$-normal, $\VV$-${<}\theta$-complete ultrafilter on $\POT{\Poti{\gamma}{\theta}}^\VV$  in $\VV[G]$, with the property that the corresponding ultrapower $\Ult{\VV}{\dot{\calU}^G}$ is well-founded whenever $G$ is $\PPP$-generic over $\VV$. 
\end{definition}

\begin{proposition}\label{proposition:GenericElementaryEmbeddingFromGenericSupercompact}
  Let $\PPP$ be a partial order witnessing that an uncountable regular cardinal $\theta$ is generically $\gamma$-supercompact, and let $\dot{\calU}$ be the corresponding $\PPP$-name. If $G$ is $\PPP$-generic over $\VV$, and $\map{j}{\VV}{\Ult{\VV}{\dot{\calU}^G}}$ is the corresponding ultrapower embedding defined in $\VV[G]$, then $j$ has critical point $\theta$, $j(\theta)>\gamma$, and $j[\gamma]\in\Ult{\VV}{\dot{\calU}^G}$.   
\end{proposition}

\begin{proof}
 The  $\VV$-${<}\theta$-completeness of $\dot{\calU}^G$ yields $j\restriction\theta=\id_\theta$. The fineness and $\VV$-normality of $\dot{\calU}^G$ imply that $j[\gamma]=[\id_{\Poti{\gamma}{\theta}^\VV}]_{\dot{\calU}^G}\in\Ult{\VV}{\dot{\calU}^G}$, and moreover 
 \begin{equation*}
  \theta \le \gamma  ~ = ~ [a\mapsto\otp{a}]_{\dot{\calU}^G} ~ < ~ [a\mapsto \theta]_{\dot{\calU}^G}=j(\theta). \qedhere
 \end{equation*}
\end{proof}

The following results yield strong characterizations of measurable and of supercompact cardinals through Neeman's pure side condition forcing.

\begin{lemma}\label{theorem:GenGammaSupercompactAtInaccessbles}
  The following statements are equivalent for every inaccessible cardinal $\theta$ and every ordinal $\gamma\geq\theta$: 
 \begin{enumerate}[leftmargin=0.7cm]
  \item $\theta$ is a $\gamma$-supercompact cardinal. 

  \item There is a partial order with the $\sigma$-approximation property witnessing that $\theta$ is generically $\gamma$-supercompact. 
 \end{enumerate}
\end{lemma}

\begin{proof}
  If (i) holds, then the trivial partial order clearly witnesses that $\theta$ is generically $\gamma$-supercompact. 
 In order to verify the reverse direction, let $\PPP$ be a partial order with the $\sigma$-approximation property that witnesses $\theta$ to be generically $\gamma$-supercompact, let $H$ be $\PPP$-generic over $\VV$, and let $\map{j}{\VV}{M}$ be the elementary embedding definable in $\VV[H]$ that is provided by an application of Proposition \ref{proposition:GenericElementaryEmbeddingFromGenericSupercompact}. In this situation, an application of Lemma \ref{lemma:FilterInGroundModel} with $\VV=\VV[G]$ shows that the set $\calU=\Set{A\in\POT{\Poti{\gamma}{\theta}}^\VV}{j[\gamma]\in j(A)}$ is an element of $\VV$ and it is easy to see that $\calU$ is a fine, ${<}\theta$-complete, normal ultrafilter on $\POT{\Poti{\gamma}{\theta}}$ in $\VV$. 
  Hence, $\calU$ witnesses that $\theta$ is $\gamma$-supercompact in $\VV$. 
\end{proof}

%
%
%

 The following result shows how $\gamma$-supercompactness can be characterized  through Neeman's pure side condition forcing. Note that in particular, this theorem yields a strong characterization of measurability and yet another strong characterization of supercompactness.

\begin{theorem}\label{theorem:NeemanGammaSupercompact}
 The following statements are equivalent for every inaccessible cardinal $\theta$ and every ordinal $\gamma\geq\theta$: 
 \begin{enumerate}[leftmargin=0.7cm]
  \item $\theta$ is a $\gamma$-supercompact cardinal. 

  \item $\mathbbm{1}_{\PPP_\theta}\Vdash$\anf{There is a partial order with the $\sigma$-approximation property that witnesses $\omega_2$ to be generically $\check{\gamma}$-supercompact}.
 \end{enumerate}
\end{theorem}

\begin{proof}
 First, assume that (i) holds, and let $\map{j}{\VV}{M}$ be an elementary embedding witnessing the $\gamma$-supercompactness of $\theta$. Set $K=\HH{j(\theta)}^M$, $\calS=\calS_{j(\theta)}^M$ and $\calT=\calT_{j(\theta)}^M$. Then $\PPP_{K,\calS,\calT}=\PPP_{j(\theta)}^M=j(\PPP_\theta)$.

 \begin{claim*}
  The set $K$ is suitable and the pair $(\calS,\calT)$ is appropriate for $K$. 
 \end{claim*}

 \begin{proof}[Proof of the Claim]
  Since $\omega_1^M=\omega_1<\theta<j(\theta)$, elementarity directly yields the above statements. 
 \end{proof}

 Note that the closure properties of $M$ imply that $\HH{\theta}\in M$, $\PPP_\theta=\PPP_\theta^M$ and $\HH{\theta}\in\calT$. Moreover, Lemma \ref{trans_in_generic} and elementarity imply that $\suborder{\PPP_{K,\calS,\calT}}{\langle\HH{\theta}\rangle}$ is dense in $\PPP_{K,\calS,\calT}$.  Define $\dot{\QQQ}=\dot{\QQQ}^{\HH{\theta}}_{K,\calS,\calT}$.  Then Lemma \ref{chain_condition} and the closure properties of $M$ imply that $\dot{\QQQ}=(\dot{\QQQ}_{j(\theta)}^{\HH{\theta}})^M$. Let $G$ be $\PPP_\theta$-generic over $\VV$.

 \begin{claim*}
  The partial order $\dot{\QQQ}^G$ has the $\sigma$-approximation property in $\VV[G]$. 
 \end{claim*}

 \begin{proof}[Proof of the Claim]
  By Corollary \ref{corollary:sigma_approximation}, it suffices to show that $\calS$ is a stationary subset of $\POT{K}$ in $\VV$. Work in $\VV$ and fix a function $\map{f}{[K]^{{<}\omega}}{K}$. Then the closure properties of $M$ imply that $M$ contains a sequence $\seq{X_n}{n<\omega}$ of countable elementary substructures of $K$ with the property that $f[[X_n]^{{<}\omega}]\subseteq X_{n+1}$ for all $n<\omega$. But then $\bigcup\Set{X_n}{n<\omega}\in C_f\cap\calS\neq\emptyset$. 
 \end{proof}

 If $H$ is $\dot{\QQQ}^G$-generic over $\VV[G]$ and $F$ is the filter on $\PPP_{K,\calS,\calT}$ induced by the embedding $D^{\HH{\theta}}_{K,\calS,\calT}$ and the filter $G*H$, then $j\restriction\PPP_\theta=\id_{\PPP_\theta}$ implies that  $j[G]\subseteq F$ and hence there is an embedding $\map{j_{G,H}}{\VV[G]}{M[F]}$ that extends $j$ and is definable in $\VV[G,H]$. Let $\dot{\calU}$ denote the canonical $\dot{\QQQ}^G$-name in $\VV[G]$ with the property that whenever $H$ is $\dot{\QQQ}^G$-generic over $\VV[G]$, then $$\dot{\calU}^H ~ = ~ \Set{A\in\POT{\Poti{\gamma}{\theta}}^{\VV[G]}}{j[\gamma]\in j_{G,H}(A)}$$ and therefore standard arguments show that $\dot{\calU}^H$ is a fine, $\VV[G]$-normal, $\VV[G]$-${<}\theta$-complete ultrafilter on $\POT{\Poti{\gamma}{\theta}}^{\VV[G]}$ with the property that $\Ult{\VV[G]}{\dot U^H}$ is well-founded. This allows us to conclude that (ii) holds.

 Now, assume that (ii) holds and let $G$ be $\PPP_\theta$-generic over $\VV$. In $\VV[G]$, there is a partial order $\QQQ$ with the $\sigma$-approximation property that witnesses that $\theta$ is generically ${\gamma}$-supercompact. Let $H$ be $\QQQ$-generic over $\VV[G]$. Then Proposition \ref{proposition:GenericElementaryEmbeddingFromGenericSupercompact} yields an elementary embedding $\map{j}{\VV[G]}{M}$ definable in $\VV[G,H]$ with critical point $\theta$, $j(\theta)>\gamma$, and $j[\gamma]\in M$. In this situation, Corollary \ref{corollary:sigma_approximation} and Lemma \ref{chain_condition} show that the assumptions of Lemma \ref{lemma:FilterInGroundModel} are satisfied, and therefore $j[\gamma]\in j(\Poti{\gamma}{\theta}^\VV)$ and $\calU=\Set{A\in\POT{\Poti{\gamma}{\theta}}^\VV}{j[\gamma]\in j(A)}\in\VV$. Since it is easy to see that $\calU$ is a fine, ${<}\theta$-complete, normal ultrafilter on $\POT{\Poti{\gamma}{\theta}}$ in $\VV$, 
 it follows that $\theta$ is $\gamma$-supercompact in $\VV$, as desired. 
\end{proof}


%
%
 
%


\section{Almost huge and super almost huge cardinals}

In this section, we want to show that we can use Neeman's pure side condition forcing for the characterization of some of the strongest large cardinals, that is we want to provide strong characterizations of \emph{almost huge} and of \emph{super almost huge} cardinals through Neeman's pure side condition forcing, the proofs of which strongly rely on the proofs from Section \ref{section:supercompact}. 
Remember that a cardinal $\theta$ is \emph{almost huge} if there is an elementary embedding $\map{j}{\VV}{M}$ with $\crit{j}=\theta$ and ${}^{{<}j(\theta)} M\subseteq M$. If such an embedding $j$ exists, then we say that \emph{$\theta$ is almost huge with target $j(\theta)$}. Our characterization of almost hugeness will rely on a generic large cardinal concept for almost hugeness. The following lemma provides us with an adaption of Lemma \ref{lemma:FilterInGroundModel} to the setting of almost huge cardinals.

\begin{lemma}\label{lemma:MainLemmaGenAlmostHuge}
 Let $\VV[G]$ be a generic extension of the ground model $\VV$, let $\VV[G,H]$ be a generic extension of $\VV[G]$, let $\theta$ be an uncountable regular cardinal in $\VV[G]$ and let $\lambda>\theta$ be an uncountable regular cardinal in $\VV[G,H]$. Assume that the following statements hold: 
 \begin{enumerate}[leftmargin=0.7cm]
  \item $\theta$ and $\lambda$ are inaccessible cardinals in $\VV$.  

  \item The pair $(\VV,\VV[G])$ satisfies the $\sigma$-approximation and the $\theta$-cover property. 

  \item The pair $(\VV[G],\VV[G,H])$ satisfies the $\sigma$-approximation property. 

  \item There is an elementary embedding $\map{j}{\VV[G]}{M}$ definable in $\VV[G,H]$ with the property that $\crit{j}=\theta$, $j(\theta)=\lambda$ and $j[\gamma]\in M$ for all $\gamma<\lambda$.  
 \end{enumerate}
 Then $\theta$ is almost huge with target $\lambda$ in $\VV$. 
\end{lemma}

\begin{proof}
  Given $\theta\leq\gamma<\lambda$, define $$\calU_\gamma ~ = ~ \Set{A\in\POT{\Poti{\gamma}{\theta}}^\VV}{j[\gamma]\in j(A)} ~ \in ~ \VV[G,H].$$ Then $\calU_\gamma=\Set{\Set{a\cap\gamma}{a\in A}}{A\in\calU_\delta}$ for all $\theta\leq\gamma\leq\delta<\lambda$. Moreover, we can apply Lemma \ref{lemma:FilterInGroundModel} to conclude that  for every $\theta\leq\gamma<\lambda$, we have $j[\gamma]\in j(\Poti{\gamma}{\theta}^\VV)$, and  $\calU_\gamma$ is an element of $\VV$. Define $$\UUU ~ = ~ \Set{\calU_\gamma}{\theta\leq\gamma<\lambda} ~ \in ~ \VV[G,H].$$

 \begin{claim*}
  $\UUU\in\VV$. 
 \end{claim*}

 \begin{proof}[Proof of the Claim]
   First, assume, towards a contradiction, that $\UUU\notin\VV[G]$. Then our assumptions imply that there is $u\in\VV[G]$ that is countable in $\VV[G]$ with the property that $\UUU\cap u\notin\VV[G]$. Since $\lambda$ is regular and uncountable in $\VV[G,H]$, we can find $\theta\leq\delta<\lambda$ with $$\UUU\cap u ~ = ~ \Set{\calU_\gamma}{\gamma<\delta}\cap u ~ = ~ \Set{\Set{\Set{a\cap\gamma}{a\in A}}{A\in\calU_\delta}}{\gamma<\delta}\cap u.$$ But then  $\calU_\delta\in\VV\subseteq\VV[G]$ implies that $\UUU\cap u\in\VV[G]$, a contradiction.

   Since we already know that $\UUU\subseteq\VV$, we can use the same argument to show that the set $\UUU$ is an element of $\VV$. 
 \end{proof}

 \begin{claim*}
  If $\theta\leq\gamma<\lambda$ and $f\in\left({}^{\Poti{\gamma}{\theta}}\theta\right)^\VV$ with $\Set{a\in\Poti{\gamma}{\theta}}{\otp{a}\leq f(a)}\in\calU_\gamma$, then there is $\gamma\leq\delta<\lambda$ with $\Set{a\in\Poti{\delta}{\theta}}{f(a\cap\gamma)=\otp{a}}\in\calU_\delta$. 
 \end{claim*}
 
 \begin{proof}[Proof of the Claim]
  Since  $j[\gamma]\in j(\Poti{\gamma}{\theta}^\VV)=\dom{j(f)}$, there is a $\delta<\lambda=j(\theta)$ with $\delta=j(f)(j[\gamma])$. Then, we have  $\gamma=\otp{j[\gamma]}\leq\delta<\lambda$, and $$j(f)(j(\gamma)\cap j[\delta]) ~ = ~ j(f)(j[\gamma]) ~ = ~ \delta ~ = ~ \otp{j[\delta]}.$$  This shows that $\Set{a\in\Poti{\delta}{\theta}}{f(a\cap\gamma)=\otp{a}}\in\calU_\delta$. 
 \end{proof}

For every $\theta\leq\gamma<\lambda$,  $j[\gamma]\in j(\Poti{\gamma}{\theta}^\VV)$ implies that $\calU_\gamma$ is a fine, normal, $\theta$-complete filter on $\POT{\Poti{\gamma}{\theta}}$ in $\VV$. Let $M_\gamma=\Ult{\VV}{\calU_\gamma}$ denote the corresponding ultrapower and let $\map{j_\gamma}{\VV}{M_\gamma}$ denote the induced ultrapower embedding. Given $\theta\leq\gamma\leq\delta<\lambda$, we have $\calU_\gamma=\Set{\Set{a\cap\gamma}{a\in A}}{A\in\calU_\delta}$, and the map $$\Map{k_{\gamma,\delta}}{M_\gamma}{M_\delta}{[f]_{\calU_\gamma}}{[a\mapsto f(a\cap\gamma)]_{\calU_\delta}}$$ is an elementary embedding with $j_\delta=k_{\gamma,\delta}\circ j_\gamma$.

 Now, work in $\VV$, and fix $\theta\leq\gamma<\lambda$ and $\theta\leq\xi<j_\gamma(\theta)$. Then $\xi=[f]_{\calU_\gamma}$ for some function $\map{f}{\Poti{\gamma}{\theta}}{\theta}$, and therefore $\Set{a\in\Poti{\gamma}{\theta}}{\otp{a}\leq f(a)}\in\calU_\gamma$. In this situation, the last claim yields an ordinal $\gamma\leq\delta<\lambda$ with the property that $\Set{a\in\Poti{\delta}{\theta}}{f(a\cap\gamma)=\otp{a}}\in\calU_\delta$, and this implies that $$k_{\gamma,\delta}(\xi) ~ = ~ k_{\gamma,\delta}([f]_{\calU_\gamma}) ~ = ~ [a\mapsto f(a\cap\gamma)]_{\calU_\delta} ~ = ~ [a\mapsto\otp{a}]_{\calU_\delta} ~ = ~ \delta.$$

Since $\lambda$ is inaccessible in $\VV$, the above computation allow us to apply {\cite[Theorem 24.11]{MR1994835}} to conclude that $\theta$ is almost huge with target $\lambda$ in $\VV$. 
\end{proof}

Analogous to the previous section, we will now discuss the typical situation in which the assumptions of the previous lemma are satisfied.

\begin{definition}\label{definition:GenAlmostHuge}
 Given an uncountable regular cardinal $\theta$ and an inaccessible cardinal $\lambda>\theta$, we say that a partial order \emph{$\PPP$ witnesses that $\theta$ is generically almost huge with target $\lambda$} if the following statements hold: 
 \begin{enumerate}[leftmargin=0.7cm]
  \item Forcing with $\PPP$ preserves the regularity of $\lambda$. 

  \item There is a sequence $\seq{\dot{\calU}_\gamma}{\theta\leq\gamma<\lambda}$ of $\PPP$-names such that the following statements hold in $\VV[G]$ whenever $G$ is $\PPP$-generic over $\VV$:  
  \begin{enumerate}[leftmargin=0.7cm]
   \item If $\theta\leq\gamma<\lambda$, then $\dot{\calU}^G$ is a fine, $\VV$-normal, $\VV$-${<}\theta$-complete filter  on $\POT{\Poti{\gamma}{\theta}}^\VV$ with the property that the corresponding ultrapower $\Ult{\VV}{\dot{\calU}^G}$ is well-founded. 
  
   \item If $\theta\leq\gamma\leq\delta<\lambda$, then $\dot{\calU}^G_\gamma = \Set{\Set{a\cap\gamma}{a\in A}}{A\in\dot{\calU}_\delta^G}$. 

  \item If $\theta\leq\gamma<\lambda$ and $f\in({}^{\Poti{\gamma}{\theta}}\theta)^\VV$, then there is $\gamma\leq\delta<\lambda$ with $$\Set{a\in\Poti{\delta}{\theta}^\VV}{f(a\cap\gamma)\leq\otp{a}} ~ \in ~ \dot{\calU}^G_\delta.$$
  \end{enumerate}
 \end{enumerate}
\end{definition}

The name of the property defined above is justified by the following proposition and by {\cite[Lemma 3]{MR819932}} stating that, in the setting of that proposition, $j[\gamma]\in M$ implies  $\POT{\gamma}^\VV\in M$ for all $\gamma<\theta$.

\begin{proposition}\label{proposition:GenAlmostHugeEmbedding}
 Given an uncountable regular cardinal $\theta$ and an inaccessible cardinal $\lambda>\theta$, if a partial order $\PPP$ witnesses that $\theta$ is generically almost huge with target $\lambda$ and $G$ is $\PPP$-generic over $\VV$, then there is an elementary embedding $\map{j}{\VV}{M}$ definable in $\VV[G]$ with $\crit{j}=\theta$, $j(\theta)=\lambda$ and $j[\gamma]\in M$ for all $\gamma<\lambda$. 
\end{proposition}

\begin{proof}
 Let $\seq{\dot{\calU}_\gamma}{\gamma\leq\gamma<\lambda}$ be the corresponding sequence of $\PPP$-names and let  $G$ be $\PPP$-generic over $\VV$. Given $\theta\leq\gamma<\lambda$, let $M_\gamma=\Ult{\VV}{\dot{\calU}^G_\gamma}$ denote the corresponding generic ultrapower and let $\map{j_\gamma}{\VV}{M_\gamma}$ denote the corresponding elementary embedding. Then Proposition \ref{proposition:GenericElementaryEmbeddingFromGenericSupercompact} shows that $j_\gamma$ has critical point $\theta$ and $j_\gamma[\gamma]\in M_\gamma$ for all $\theta\leq\gamma<\lambda$. Moreover, if $\theta\leq\gamma\leq\delta<\lambda$, then the function $$\Map{k_{\gamma,\delta}}{M_\gamma}{M_\delta}{[f]_{\calU_\gamma}}{[a\mapsto f(a\cap\gamma)]_{\calU_\delta}}$$ is an elementary embedding with $j_\delta=k_{\gamma,\delta}\circ j_\gamma$. In addition, it is easy to see that $k_{\gamma,\epsilon}=k_{\delta,\epsilon}\circ k_{\gamma,\delta}$ holds for all $\theta\leq\gamma\leq\delta\leq\epsilon<\lambda$. Since $\lambda$ has uncountable cofinality in $\VV[G]$, the corresponding limit $$\langle M, ~ \seq{\map{k_\gamma}{M_\gamma}{M}}{\theta\leq\gamma<\lambda}\rangle$$ of the resulting directed system $$\langle\seq{M_\gamma}{\theta\leq\gamma<\lambda}, ~ \seq{\map{k_{\gamma,\delta}}{M_\gamma}{M_\delta}}{\theta\leq\gamma\leq\delta<\lambda}\rangle$$ is well-founded, and we can identify $M$ with its transitive collapse. If $\map{j}{\VV}{M}$ is the unique map with $j=k_\gamma\circ j_\gamma$ for all $\theta\leq\gamma<\lambda$, then the above remarks  directly imply that $j$ is an elementary embedding with critical point $\theta$.

 Now, fix $\theta\leq\gamma<\lambda$. If $\alpha<\gamma$, then $j_\gamma(\alpha)\in j_\gamma[\gamma]$, and therefore $j(\alpha)\in k_\gamma(j_\gamma[\gamma])$. In the other direction, pick $\beta\in k_\gamma(j_\gamma[\gamma])$. Then we can find  $\gamma\leq\delta<\lambda$ and $\beta_0\in k_{\gamma,\delta}(j_\gamma[\gamma])=[a\mapsto a\cap\gamma]_{\dot{\calU}^G_\delta}$ with $\beta=k_\delta(\beta_0)$. In this situation, $\VV$-normality implies that there is an $\alpha<\gamma$ with $\beta_0=j_\delta(\alpha)$ and hence $\beta=j(\alpha)$. In combination, these arguments show that $j[\gamma]=k_\gamma(j_\gamma[\gamma])\in M$ for all $\gamma<\lambda$. But this also implies that $\gamma=\otp{j[\gamma]} =k_\gamma(\otp{j_\gamma[\gamma]})=k_\gamma(\gamma)$ holds for all $\theta\leq\gamma<\lambda$.

 Finally, fix $\beta<j(\theta)$. Then there is a $\theta\leq\gamma<\lambda$ and a function  $f\in({}^{\Poti{\gamma}{\theta}}\theta)^\VV$ such that  $\beta=k_\gamma([f]_{\calU_\gamma})$. By Definition \ref{definition:GenAlmostHuge}, we can find an ordinal $\gamma\leq\delta<\lambda$ with $\Set{a\in\Poti{\delta}{\theta}^\VV}{f(a\cap\gamma)\leq\otp{a}}  \in\dot{\calU}^G_\delta$. This implies that $k_{\gamma,\delta}([f]_{\dot{\calU}^G_\gamma})\leq \delta$ and hence $\beta\leq k_\delta(\delta)=\delta$. This shows that $j(\theta)\leq\lambda$. Since we obviously also have $j(\theta)\geq\lambda$, we can conclude that $j(\theta)=\lambda$. 
\end{proof}

The next result shows that the characterization of almost hugeness presented below is strong.

\begin{lemma} 
 The following statements are equivalent for all inaccessible cardinals $\theta<\lambda$:
  \begin{enumerate}[leftmargin=0.7cm]
   \item  $\theta$ is almost huge with target $\lambda$. 

  \item There is a partial order with the $\sigma$-approximation property witnessing that $\theta$ is generically almost huge with target $\lambda$.  
  \end{enumerate}
\end{lemma}

\begin{proof}
  If $\theta$ is almost huge with target $\lambda$, then the trivial partial order witnesses that $\theta$ is generically almost huge with target $\lambda$ by {\cite[Theorem 24.11]{MR1994835}}. In order to verify the reverse direction, let $\PPP$ be a partial order with the $\sigma$-approximation property that witnesses that $\theta$ is generically almost huge with target $\lambda$, let $H$ be $\PPP$-generic over $\VV$ and let $\map{j}{\VV}{M}$ be the elementary embedding definable in $\VV[H]$ that is provided by an application of Proposition \ref{proposition:GenAlmostHugeEmbedding}. Then, an application of Lemma \ref{lemma:MainLemmaGenAlmostHuge} with $\VV=\VV[G]$ shows that $\theta$ is almost huge with target $\lambda$ in $\VV$.  
\end{proof}

The following theorem contains our characterization of almost hugeness through Neeman's pure side condition forcing.

\begin{theorem}\label{theorem:NeemanAlmostHuge}
 The following statements are equivalent for every inaccessible cardinal $\theta$: 
 \begin{enumerate}[leftmargin=0.7cm]
  \item $\theta$ is an almost huge cardinal. 

  \item $\mathbbm{1}_{\PPP_\theta}\Vdash$\anf{There is an inaccessible cardinal $\lambda$ and a partial order $\PPP$ with the $\sigma$-approximation property that witnesses that $\omega_2$ is generically almost huge with target $\lambda$}.
 \end{enumerate}
\end{theorem}

\begin{proof}
  First, assume that (i) holds, and let the almost hugeness of $\theta$ be witnessed by the embedding $\map{j}{\VV}{M}$. Then $\lambda=j(\kappa)$ is an inaccessible cardinal,  $\HH{\lambda}\subseteq M$, $\PPP_\theta=\PPP_\theta^M$, $j(\PPP_\theta)=\PPP_\lambda$ and $\HH{\theta}\in\calT_\lambda$.  Set $\dot{\QQQ}=\dot{\QQQ}^{\HH{\theta}}_\lambda$ and let $G$ be $\PPP_\theta$-generic over $\VV$. Then $\lambda$ is inaccessible in $\VV[G]$ and Corollary \ref{corollary:sigma_approximation} implies that $\dot{\QQQ}^G$ has the $\sigma$-approximation property in $\VV[G]$. Now, if $H$ is $\dot{\QQQ}^G$-generic over $\VV[G]$ and $F$ is the filter on $\PPP_\lambda$ induced by the embedding $D^{\HH{\theta}}_\lambda$ and the filter $G*H$, then $j[G]=G\subseteq F$ and there is an embedding $\map{j_{G,H}}{\VV[G]}{M[F]}$ that extends $j$ and is definable in $\VV[G,H]$. Given $\theta\leq\gamma<\lambda$, let $\dot{\calU}_\gamma$ be the canonical $\dot\QQQ^G$-name in $\VV[G]$ such that $$\dot{\calU}_\gamma^H ~ = ~ \Set{A\in\POT{\Poti{\gamma}{\theta}}^{\VV[G]}}{j[\gamma]\in j_{G,H}(A)}$$ holds whenever $H$ is $\dot{\QQQ}^G$-generic over $\VV[G]$. Then forcing with $\dot{\QQQ}^G$ over $\VV[G]$ preserves the regularity of $\lambda$ and, as in the proof of Lemma \ref{lemma:MainLemmaGenAlmostHuge}, we can also show the sequence $\seq{\dot{\calU}_\gamma}{\theta\leq\gamma<\lambda}$ of $\dot{\QQQ}^G$-names satisfies the statements listed in Item (ii) of Definition \ref{definition:GenAlmostHuge} in $\VV[G]$. In particular, $\dot{\QQQ}^G$ witnesses $\theta$ to be generically almost huge with target $\lambda$ in $\VV[G]$.

 In the other direction, assume that (ii) holds and let $G$ be $\PPP_\theta$-generic over $\VV$. In $\VV[G]$, there is an inaccessible cardinal $\lambda>\theta$ and a partial order $\QQQ$ with the $\sigma$-approximation property that witnesses that $\theta$ is generically almost huge with target $\lambda$. Let $H$ be $\QQQ$-generic over $\VV[G]$. An application of Proposition \ref{proposition:GenAlmostHugeEmbedding} shows that there is an elementary embedding $\map{j}{\VV[G]}{M}$ definable in $\VV[G,H]$ with $\crit{j}=\theta$, $j(\theta)=\lambda$ and $j[\gamma]\in M$ for all $\gamma<\lambda$. Since Corollary \ref{corollary:sigma_approximation} and Lemma \ref{chain_condition} show that the assumptions of Lemma \ref{lemma:MainLemmaGenAlmostHuge} are satisfied, it follows by Lemma \ref{lemma:MainLemmaGenAlmostHuge} that $\theta$ is almost huge with target $\lambda$ in $\VV$.  
\end{proof}

The arguments contained in the above proofs also allow us to prove the analogous results for \emph{super almost huge cardinals} (see, for example, \cite{MR1779746} and \cite{MR3486170}), i.e.\ cardinals $\theta$ with the property that for every $\gamma>\theta$, there is an inaccessible cardinal $\lambda>\gamma$ such that $\theta$ is almost huge with target $\lambda$.

\begin{corollary}
  The following statements are equivalent for every inaccessible cardinal $\theta$: 
  \begin{enumerate}[leftmargin=0.7cm]
   \item $\theta$ is a super almost huge cardinal. 

   \item For every $\gamma>\theta$, there is an inaccessible cardinal $\lambda>\gamma$ and a partial order $\PPP$ with the $\sigma$-approximation property witnessing that $\theta$ is generically almost huge with target $\lambda$.  

   \item $\mathbbm{1}_{\PPP_\theta}\Vdash$\anf{For every ordinal   $\gamma$, there is an inaccessible cardinal $\lambda>\gamma$ and a partial order $\PPP$ with the $\sigma$-approximation property witnessing that $\omega_2$ is generically almost huge with target $\lambda$}. \qed 
  \end{enumerate}
\end{corollary}


\section{Concluding remarks and open questions}

The characterizations provided in the first half of this paper only rely on a short list of structural properties of Neeman's pure side condition forcing: the uniform definability of $\PPP_\theta$ as a subset of $\HH{\theta}$ and the consequences of Corollary \ref{corollary:Proper}, Corollary \ref{corollary:sigma_approximation} and Lemma \ref{chain_condition}.  Using the presentation of results of Mitchell in {\cite[Section 5]{MR2959668}}, one can directly see that the partial order constructed in Mitchell's classical proof of the consistency of the tree property at $\omega_2$ in \cite{MR0313057} is an example of a forcing that satisfies all of the relevant properties. Therefore, it is also possible to use partial orders of this form to characterize inaccessibility, Mahloness, $\Pi^m_n$-indescribability, subtlety, $\lambda$-ineffability and supercompactness.

The fact that quotients of forcing notions of the form $\PPP_\theta$ satisfy the $\sigma$-approxi\-ma\-tion property is central for almost all large cardinal characterizations presented in this paper. Note that this property implies that these quotients add new real numbers, and that this causes the Continuum Hypothesis to fail in the final forcing extension. In addition, if we want to use some sequence of collapse forcing notions in order to characterize inaccessibility as in Theorem \ref{theorem:CharInaccessible}, then these collapses have to force failures of the $\GCH$ below the relevant cardinals. This shows that, in order to obtain large cardinal characterizations based on forcing notions whose quotients do not add new reals, one has to work with different combinatorial principles. Since Proposition \ref{proposition:ColDoesNotWork} shows that the canonical collapse forcing with this quotient behaviour, the \emph{L{\'e}vy Collapse $\Col{\kappa}{{<}\theta}$}, is not suitable for the type of large cardinal characterization as in Definition \ref{definition:Characterization}, it is then natural to consider the two-step iteration $\Add{\omega}{1}*\Col{\kappa}{{<}\theta}$ that first adds a Cohen real and then collapses some cardinal $\theta$ to become the successor of a regular uncountable cardinal $\kappa$. Using results of \cite{MR2063629}, showing that forcings of this form satisfy the $\sigma$-approximation and cover property, it is possible to modify the characterizations obtained in the early sections of this paper in order to characterize inaccessibility, Mahloness and weak compactness with the help of the sequence $\seq{\CCC_\theta}{\theta\in\Card}$, if we let $\CCC_\theta=\Add{\omega}{1}*\Col{\omega_1}{{<}\theta}$. In these modifications, we replace statements about the non-existence of certain trees by statements claiming that these trees contain \emph{Cantor subtrees}, i.e.\ that there is an embedding $\map{\iota}{{}^{{\leq}\omega}2}{\TTT}$ of the full binary tree ${}^{{\leq}\omega}2$ of height $\omega+1$ into the given tree $\TTT$, that satisfies $\length{\iota(s)}{\TTT}=\sup_{n<\omega}\length{\iota(s\restriction n)}{\TTT}$ and $\length{\iota(s\restriction n)}{\TTT}=\length{\iota(t\restriction n)}{\TTT}$ for all $s,t\in{}^{\omega}2$ and $n<\omega$. Using results from \cite{MR631563} and ideas contained in the proof of  {\cite[Theorem 7.2]{MR3430247}}, it is then possible to obtain the following  characterizations: 
 \begin{itemize}[leftmargin=0.7cm]
  \item An infinite cardinal $\theta$ is inaccessible if and only if $\CCC_\theta$ forces $\theta$ to become $\omega_2$ and every tree of height $\omega_1$ with $\aleph_2$-many cofinal branches to contain a Cantor subtree. 

  \item An inaccessible cardinal $\theta$ is a Mahlo cardinal if and only if $\CCC_\theta$ forces all special $\omega_2$-Aronszajn trees to contain a Cantor subtree. 

  \item An inaccessible cardinal $\theta$ is weakly compact if and only if $\CCC_\theta$ forces all $\omega_2$-Aronszajn trees to contain a Cantor subtree.
 \end{itemize}
 In addition, it is also possible to use {\cite[Theorem 10]{MR2063629}} and arguments from the proof of Lemma \ref{lemma:FilterInGroundModel} to prove analogues of the results of the previous two sections for the sequence $\seq{\CCC_\theta}{\theta\in\Card}$: 
 \begin{itemize}[leftmargin=0.7cm]
  \item An inaccessible cardinal $\theta$ is $\lambda$-supercompact for some cardinal $\lambda\geq\theta$ if and only if in every $\CCC_\theta$-generic extension, there is a $\sigma$-closed partial order witnessing that $\omega_2$ is generically $\lambda$-supercompact. 

  \item An inaccessible cardinal $\theta$ is almost huge with target $\lambda>\theta$ if and only if in every $\CCC_\theta$-generic extension, there is a $\sigma$-closed partial order witnessing that $\omega_2$ is generically almost huge with target $\lambda$.  
 \end{itemize}
It follows directly that the large cardinal characterizations obtained in this way are all strong. The details of these results will be presented in the forthcoming \cite{internal}. Note that the above arguments provide no analogues for the results of Section \ref{section:Indescribable} (except for the case of weakly compact cardinals) and of Section \ref{section:ineffable}. We do not know which combinatorial principles could  replace the ones used in these sections in order to allow characterizations of the corresponding large cardinal properties through the sequence $\seq{\CCC_\theta}{\theta\in\Card}$. These observations motivate the following question:

\begin{question}
 Does the sequence $\seq{\CCC_\theta}{\theta\in\Card}$ characterize $\Pi^m_n$-indescribability, subtlety or $\lambda$-ineffability? 
\end{question}

Proposition \ref{proposition:ColDoesNotWork} shows that the L{\'e}vy Collapse is not suitable for large cardinal characterizations in the sense of Definition \ref{definition:Characterization}, by showing that it cannot characterize inaccessibility in this way. 
However, we do not know whether it could be used to characterize stronger large cardinal properties if we restrict the desired provable equivalences to inaccessible cardinals. In particular, we cannot answer the following sample question:

\begin{question}
 Is there a parameter-free formula $\varphi(v)$ in the language of set theory with the property that $$\ZFC\vdash\textit{$\forall\theta$ inaccessible} ~ [\textit{$\theta$ is weakly compact} ~ \longleftrightarrow ~ \mathbbm{1}_{\Col{\omega_1}{{<}\theta}}\Vdash\varphi(\check\theta)\hspace{0.9pt}] ~ ?$$
\end{question}

In the remainder of this section, we present some arguments suggesting that if it is possible to characterize stronger large cardinal properties of inaccessible cardinals using forcings of the form $\Col{\omega_1}{{<}\theta}$, then the combinatorial principles to be used in these equivalences are not as canonical as the ones that appear in the above characterization through Neeman's pure side condition forcing. The proof of the following result is based upon a classical construction of Kunen from \cite{MR495118}.

\begin{theorem}\label{theorem:ConsNoCantorSubtreeNotWC}
 If $\theta$ is a weakly compact cardinal, then the following statements hold in a regularity preserving forcing extension $\VV[G]$ of the ground model $\VV$: 
 \begin{enumerate}[leftmargin=0.7cm]
  \item $\theta$ is an inaccessible cardinal that is not weakly compact. 

  \item $\mathbbm{1}_{\Col{\omega_1}{{<}\theta}}\Vdash\anf{\textit{Every $\theta$-Aronszajn tree contains a Cantor  subtree}}$. 
 \end{enumerate}
\end{theorem}

\begin{proof}
 By classical results of Silver, we may assume that $$\mathbbm{1}_{\Add{\theta}{1}}\Vdash\anf{\textit{$\check{\theta}$ is weakly compact}}.$$ Given $D\subseteq\theta$, let $\pi_D$ denote the unique automorphism of the tree ${}^{{<}\theta}2$ with the property that $$\pi_D(t)(\alpha) ~ = ~ t(\alpha) ~ \Longleftrightarrow ~ \alpha\notin D$$ holds for all $t\in{}^{{<}\theta}2$ and $\alpha\in\dom{t}$. Moreover, given $s,t\in{}^{{<}\theta}2$, we set $$\Delta(s,t) ~ = ~ \Set{\alpha\in\dom{s}\cap\dom{t}}{s(\alpha)\neq t(\alpha)}.$$ Note that $\pi_{\Delta(s,t)}(s)=t$ holds for all $s,t\in{}^{{<}\theta}2$ with $\dom{s}=\dom{t}$.

 Define $\PPP$ to be the partial order whose conditions are either $\emptyset$, or normal, $\sigma$-closed subtrees $S$ of ${}^{{<}\theta}2$ of cardinality less than $\theta$ and height $\alpha_S+1<\theta$,\footnote{In this situation, normality means that if $s\in S$ with $\dom{s}\in\alpha_S$, then $s^\frown\langle i\rangle\in S$ for all $i<2$, and there is a $t\in S(\alpha_S)$ with $s\subseteq t$.} with the additional property that for all $s,t\in S$ with $\dom{s}=\dom{t}$, the map  $\pi_{\Delta(s,t)}\restriction S$ is an automorphism of $S$. Let $\PPP$ be ordered by reverse end-extension. If $G$ is $\PPP$-generic over $\VV$, then $\bigcup\bigcup G$ is a subtree of ${}^{{<}\theta}2$. Let $\dot{\SSS}$ be the canonical $\PPP$-name for the forcing notion corresponding to the tree $\bigcup\bigcup G$, and let $$D ~ = ~ \Set{\langle S,\check{s}\rangle\in \PPP*\dot{\SSS}}{S\in\PPP, ~ s\in S(\alpha_S)}.$$ Then it is easy to see that $D$ is dense in $\PPP*\dot{\SSS}$. 

  \begin{claim*}
   Let $\lambda<\theta$, and let $\seq{S_\gamma}{\gamma<\lambda}$ be a descending sequence in $\PPP$. Define $\alpha=\sup_{\gamma<\lambda}\alpha_{S_\gamma}$,  $S=\bigcup\Set{S_\gamma}{\gamma<\lambda}$ and $[S] = \Set{t\in{}^\alpha 2}{\forall\gamma<\lambda ~ t\restriction\alpha_{S_\gamma}\in S_\gamma}$. 
 \begin{enumerate}[leftmargin=0.7cm]
  \item[(a)] If $\cof{\lambda}=\omega$, then $[S]\neq\emptyset$ and $S\cup[S]$ is the unique condition $T$ in $\PPP$ with $\alpha_T=\alpha$ and $T\leq_\PPP S_\gamma$ for all $\gamma<\lambda$. 

  \item[(b)] If $\cof{\lambda}>\omega$ and $[S]\neq\emptyset$, then $S\cup [S]$ is a condition in $\PPP$ below $S_\gamma$ for all $\gamma<\lambda$. 
  
  \item[(c)] If $\cof{\lambda}>\omega$, $G$ is the subgroup of the group of all automorphisms of ${}^{{<}\theta}2$ that is generated by the set $\Set{\pi_{\Delta(s,t)}}{s,t\in S, ~ \dom{s}=\dom{t}}$, $u\in[S]$ and $B=\Set{\pi(u)}{\pi\in G}$, then $S\cup B$ is a condition in $\PPP$ below $S_\gamma$ for all $\gamma<\lambda$.  
 \end{enumerate}
 In particular, the dense suborder $D$ of $\PPP*\dot{\SSS}$ is ${<}\theta$-closed, $\PPP*\dot{\SSS}$ is forcing equivalent to $\Add{\theta}{1}$, and forcing with $\PPP$ preserves the inaccessibility of $\theta$.   
  \end{claim*}

 By the above claim, there is a winning strategy $\Sigma$ for player Even in the game $G_\theta(\PPP)$ of length $\theta$ associated to the partial order $\PPP$ (see {\cite[Definition 5.14]{MR2768691}}), with the property that whenever $\seq{S_\gamma}{\gamma<\theta}$ is a run of $G_\theta(\PPP)$ in which player Even played according to $\Sigma$, then the following statements hold: 
 \begin{enumerate}[leftmargin=0.7cm]
  \item[(1)] There is a sequence $\seq{t_\gamma}{\gamma<\theta}$ of elements of ${}^{{<}\theta}2$ with the property that $\seq{\langle S_{2\cdot\gamma},\check{t}_\gamma\rangle}{\gamma<\theta}$ is a strictly descending sequence of conditions in $D$. 

  \item[(2)] The set $\Set{\alpha_{S_{2\cdot\gamma}}}{\gamma<\theta}$ is a club in $\theta$. 

  \item[(3)] If $\lambda\in\Lim\cap\theta$ and $S=\bigcup\Set{S_\gamma}{\gamma<\lambda}$, then $S_\lambda=S\cup[S]$. 
 \end{enumerate}
In particular, $\Sigma$ witnesses that $\PPP$ is $\theta$-strategically closed.

\begin{claim*}
  $\mathbbm{1}_\PPP\Vdash\anf{\textit{$\dot{\SSS}$ is a $\sigma$-closed $\check{\theta}$-Souslin tree}}$.  
\end{claim*}

 \begin{proof}[Proof of the Claim]
  It is immediate that $\dot{\SSS}$ is forced to be a tree of height $\theta$ whose levels all have cardinality less than $\theta$, and that the tree $\dot{\SSS}$ is forced to be  $\sigma$-closed. It remains to show that its antichains have size less than $\theta$.

Therefore, let $S_*$ be a condition in $\PPP$, let $\dot{A}\in\VV$ be a $\PPP$-name for a maximal antichain in $\dot{\SSS}$, and let  $\dot{C}\in\VV$ be the induced $\PPP$-name for the club of all ordinals less than $\theta$ with the property that the intersection of $\dot{A}$ with the corresponding initial segment of $\dot{\SSS}$ is a maximal antichain in this initial segment. Then there is a run $\seq{S_\gamma}{\gamma<\theta}$ of $G_\theta(\PPP)$ in which player Even played according to $\Sigma$, $S_1\leq_\PPP S_*$, and there exist sequences $\seq{\beta_\gamma}{\gamma<\theta}$ and $\seq{A_\gamma}{\gamma<\theta}$ with the properties that $\alpha_{S_{2\cdot\gamma+1}}>\beta_\gamma$ and  $$S_{2\cdot\gamma+1}\Vdash_\PPP\anf{\check{\beta}_\gamma =\min(\dot{C}\setminus\check{\alpha}_{S_{2\cdot\gamma}})  ~ \wedge ~ \check{A}_\gamma=\dot{A}\cap {}^{{<}\check{\beta}_\gamma}2}$$ for all $\gamma<\theta$. Since $C=\Set{\alpha_{S_{2\cdot\gamma}}}{\gamma<\theta}$ is a club in $\theta$, we can find an inaccessible cardinal $\eta<\theta$ with $\eta=\alpha_{S_\eta}$ and $\betrag{S_\gamma}<\eta$ for all $\gamma<\eta$. Set $A=\bigcup\Set{A_\gamma}{\gamma<\eta}$ and $S=\bigcup\Set{S_\gamma}{\gamma<\eta}$. Then we have  $$S_\eta\Vdash\anf{\check{\eta}\in\dot{C}  ~ \wedge ~ \check{A}=\dot{A}\cap{}^{{<}\check{\eta}}2 
 ~ \wedge ~ \check{S}=\dot{\SSS}\cap{}^{{<}\check{\eta}}2 }.$$ Hence $S$ is a normal tree of cardinality and height $\eta$, and $A$ is a maximal antichain in $S$. Fix an enumeration $\seq{\pi_\gamma}{\gamma<\eta}$ of the subgroup of the group of all automorphisms of  ${}^{{<}\theta}2$ generated by all automorphisms of the form $\pi_{\Delta(s,t)}$ with $s,t\in S$ and $\dom{s}=\dom{t}$. Since $S\cap{}^\gamma 2=[S\cap{}^{{<}\gamma}2]$ holds for all $\gamma\in C\cap\eta$, we can now inductively construct a continuous increasing sequence $\seq{s_\gamma}{\gamma<\eta}$ of elements of $S$ with the property that for every $\gamma<\eta$, we have $\dom{s_\gamma}\in C$, and there is a $t_\gamma\in A$ with $\pi_\gamma^{{-}1}(t_\gamma)\subseteq s_{\gamma+1}$. Set $s=\bigcup\Set{s_\gamma}{\gamma<\eta}\in[S]$,  $B=\Set{\pi_\gamma(s)}{\gamma<\eta}$ and $T=S\cup[B]$. By the above claim, $T$ is a condition in $\PPP$ below $S_*$. By the construction of $s$, for every $u\in B$, there is a $t\in A$ with $t\subseteq u$. Hence $T\Vdash_\PPP\anf{\dot{A}=\check{A}}$. 
 \end{proof}

Let $G$ be $\PPP$-generic over $\VV$, set $\SSS=\dot{\SSS}^G$, and let $H$ be $\SSS$-generic over $\VV[G]$. Then the above computations ensure that $\theta$ is weakly compact in $\VV[G,H]$.  
Set $\CCC=\Col{\omega_1}{{<}\theta}^{\VV[G,H]}$, and let $K$ be $\CCC$-generic over $\VV[G,H]$. Since the partial order $\SSS$ is ${<}\theta$-distributive in $\VV[G]$, we have $\CCC=\Col{\omega_1}{{<}\theta}^{\VV[G]}$, and $\VV[G,H,K]$ is a $(\CCC\times\SSS)$-generic extension of $\VV[G]$. Moreover, $\CCC$ is a $\sigma$-closed, $\theta$-Knaster partial order in $\VV[G]$, and therefore $\SSS$ remains a $\sigma$-closed $\theta$-Souslin tree in $\VV[G,K]$. But this shows that the partial order $\CCC\times\SSS$ is $\sigma$-distributive in $\VV[G]$.

Let $\TTT$ be a $\theta$-Aronszajn tree in $\VV[G,K]$. First, assume that $\TTT$ has a cofinal branch in $\VV[G,H,K]$. Then, in $\VV[G,K]$, there is a $\sigma$-closed forcing that adds a cofinal branch through $\TTT$, and therefore standard arguments show that $\TTT$ contains a Cantor subtree in $\VV[G,H,K]$. In the other case, assume that $\TTT$ is a $\theta$-Aronszajn tree in $\VV[G,H,K]$. Since $\theta$ is weakly compact in $\VV[G,H]$, results from \cite{MR631563} show that $\TTT$ contains a Cantor subtree in $\VV[G,H,K]$. Let $\map{\iota}{{}^{{\leq}\omega}2}{\TTT}$  be an embedding in $\VV[G,H,K]$ witnessing this. Since the above remarks show that $({}^\omega\VV[G,K])^{\VV[G,H,K]}\subseteq\VV[G,K]$, the map $\iota\restriction({}^{{<}\omega}2)$ is an element of $\VV[G,K]$. Pick $\alpha<\theta$ with $\iota[{}^\omega 2]\subseteq\TTT(\alpha)$. Given $x\in({}^\omega 2)^{\VV[G,K]}$, we then know that there is an element $t$ of $\TTT(\alpha)$ with $\iota(x\restriction n)\leq_\TTT t$ for all $n<\omega$. This allows us to conclude that, in $\VV[G,K]$, there is an embedding from ${}^{{\leq}\omega}2$ into $\TTT$ that extends $\iota\restriction({}^{{<}\omega}2)$ and witnesses that $\TTT$ contains a Cantor subtree.  
\end{proof}

Note that, in combination with {\cite[Theorem 3.9]{MR2013395}}, the above proof shows that the existence of a weakly compact cardinal is equiconsistent with the existence of a non-weakly compact  inaccessible cardinal $\theta$ with the property that every $\theta$-Aronszajn tree contains a Cantor subtree. In contrast, the proof of the following result shows that the corresponding statement for special Aronszajn trees has much larger consistency strength. In particular, it shows that the inconsistency of certain large cardinal properties strengthening measurability would imply that the Mahloness of inaccessible cardinals can be characterized by partial orders of the form $\Col{\omega_1}{{<}\theta}$ in a canonical way.

\begin{theorem}\label{theorem:LowerBoundConsNonMahlo}
 Let $\theta$ be an inaccessible cardinal with property that one of the following statements holds:
 \begin{enumerate}[leftmargin=0.7cm]
  \item Every special $\theta$-Aronszajn tree contains a Cantor subtree. 

  \item $\mathbbm{1}_{\Col{\omega_1}{{<}\theta}}\Vdash\anf{\textit{Every special $\omega_2$-Arosnzajn tree contains a Cantor subtree}}$. 
 \end{enumerate}

 If $\theta$ is not a Mahlo cardinal, then there is an inner model that contains a stationary limit of measurable cardinals of uncountable Mitchell order. 
\end{theorem}

\begin{proof}
 Fix a closed and unbounded subset $D$ of $\theta$ that consists of singular strong limit cardinals and assume that the above conclusion fails. Then, the proof of {\cite[Theorem 1]{MR2467218}} shows that \emph{Jensen's $\square$-principle} holds up to $\theta$, i.e.\ there is a sequence $\seq{B_\alpha}{\textit{$\alpha\in\Lim\cap\theta$ singular}}$ such that for all singular limit ordinals $\alpha<\theta$, the set $B_\alpha$ is a closed and unbounded subset of $\alpha$ of order-type less than $\alpha$, and, if $\beta\in\Lim(B_\alpha)$, then $\cof{\beta}<\beta$ and $C_\beta=C_\alpha\cap\beta$. Then, we may pick a sequence $\vec{C}=\seq{C_\alpha}{\alpha\in\Lim\cap\theta}$ satisfying the following statements for all $\alpha\in\Lim\cap\theta$: 
  \begin{enumerate}[leftmargin=0.7cm]
   \item[(i)] If $\alpha\in\Lim(D)$ and $B_\alpha\cap D$ is unbounded in $\alpha$, then $C_\alpha=B_\alpha\cap D$.  
   
   \item[(ii)] If $\alpha\in\Lim(D)$ and $\max(B_\alpha\cap D)<\alpha$, then $C_\alpha$ is an unbounded subset of $\alpha$ of order-type $\omega$ with $\min(C_\alpha)>\max(B_\alpha\cap D)$. 
   
   \item[(iii)] If $\max(D\cap\alpha)<\alpha$, then $C_\alpha=(\max(D\cap\alpha),\alpha)$. 
  \end{enumerate}
It is easy to check that $\vec{C}$ is a $\square(\theta)$-sequence (see {\cite[Definition 7.1.1]{MR2355670}}).

 \begin{claim*}
  $\vec{C}$ is a special $\square(\theta)$-sequence (see {\cite[Definition 7.2.11]{MR2355670}}). 
 \end{claim*}
 
 \begin{proof}[Proof of the Claim]
  Given $\alpha\leq\beta<\theta$, let $\map{\rho_0^{\vec{C}}(\alpha,\beta)}{\beta}{{}^{{<}\omega}\theta}$ denote the \emph{full code of the walk from $\beta$ to $\alpha$ through $\vec{C}$}, as defined in {\cite[Section 7.1]{MR2355670}}. Let $\TTT=\TTT(\rho_0^{\vec{C}})$ be the tree of all functions of the form $\rho_0^{\vec{C}}(\hspace{1.7pt} \cdot \hspace{1.7pt},\beta)\restriction\alpha$ with $\alpha\leq\beta<\theta$. Then, the results of {\cite[Section 7.1]{MR2355670}} show that $\TTT$ is a $\theta$-Aronszajn tree. Fix a bijection $\map{b}{\theta}{{}^{{<}\omega}\theta}$ with $b[\kappa]={}^{{<}\omega}\kappa$ for every cardinal $\kappa\leq\theta$. Now, fix $\alpha\leq\beta<\theta$ with $\alpha\in D$, and let $\langle\gamma_0,\ldots,\gamma_n\rangle$ denote the walk from $\beta$ to $\alpha$ through $\vec{C}$. If $C_{\gamma_{n-1}}\cap\alpha$ is unbounded in $\alpha$, then the above definitions ensure that $\gamma_{n-1}\in\Lim$, $\cof{\gamma_{n-1}}<\gamma_{n-1}$, $\alpha\in\Lim(B_{\gamma_{n-1}})$, and therefore $\otp{C_{\gamma_{n-1}}\cap\alpha}\leq\otp{B_{\gamma_{n-1}}\cap\alpha}=\otp{B_\alpha}<\alpha$. This shows that we always have $\otp{C_{\gamma_{n-1}}\cap\alpha}<\alpha$, and hence there is an $\varepsilon<\alpha$ with $b(\varepsilon)=\rho_0^{\vec{C}}(\alpha,\beta)$. Define $r(\rho_0^{\vec{C}}(\hspace{1.7pt} \cdot \hspace{1.7pt},\beta)\restriction\alpha)=\rho_0^{\vec{C}}(\hspace{1.7pt} \cdot \hspace{1.7pt},\beta)\restriction\varepsilon$. Then, the proof of {\cite[Theorem 6.1.4]{MR2355670}} shows that the resulting regressive function $\map{r}{\TTT\restriction D}{\TTT}$ witnesses that the set $D$ is non-stationary with respect to $\TTT$. Since $D$ is a club in $\theta$, this implies that the tree $\TTT$ is special and, by the results of \cite{MR929410}, this conclusion is equivalent to the statement of the claim.   
 \end{proof}

 The above claim now allows us to use  {\cite[Theorem 3.14]{MR2013395}} to conclude that there is a special $\theta$-Aronszajn tree $\TTT$ without Cantor subtrees and therefore (i) fails. Since the partial order $\Col{\omega_1}{{<}\theta}$ is $\sigma$-closed, we may argue as in the last part of the proof of Theorem \ref{theorem:ConsNoCantorSubtreeNotWC} to show  that (ii) implies (i) and therefore the above assumption also implies a failure of (ii).  
\end{proof}

 The next proposition shows that examples of inaccessible non-Mahlo cardinals satisfying Statement (i) in Theorem \ref{theorem:LowerBoundConsNonMahlo} can be obtained using  supercompactness.

\begin{proposition}
 Let $\kappa<\theta$ be uncountable regular cardinals. If $\kappa$ is $\theta$-supercompact, then the following statements hold:
 \begin{enumerate}[leftmargin=0.7cm]  
  \item Every $\theta$-Aronszajn tree contains a Cantor subtree. 
  
    \item $\mathbbm{1}_{\Col{\omega_1}{{<}\kappa}}\Vdash\anf{\textit{Every $\check{\theta}$-Aronszajn tree contains a Cantor subtree}}$. 
 \end{enumerate}
\end{proposition}

\begin{proof}
 Fix an elementary embedding $\map{j}{\VV}{M}$ with $\crit{j}=\kappa$, $j(\kappa)>\theta$ and ${}^\theta M\subseteq M$. Set $\nu=\sup(j[\theta])<j(\theta)$. 
 Let $G$ be $\Col{\omega_1}{{<}\kappa}$-generic over $\VV$, let $H$ be $\Col{\omega_1}{[\kappa,j(\kappa))}$-generic over $\VV[G]$ and let $\map{j_*}{\VV[G]}{M[G,H]}$ denote the canonical lifting of $j$.

 Fix a  $\theta$-Aronszajn tree $\TTT$ in $\VV[G]$. By standard arguments, we may, without loss of generality, assume that every node in $\TTT$ has at most two direct successors, and that all elements of the limit levels of $\TTT$ are uniqueley determined by their sets of predecessors. Pick a node $t\in j_*(\TTT)(\nu)$, and define $b=\Set{s\in\TTT}{j_*(s)\leq_{j(\TTT)}t}\in\VV[G,H]$. Then $b$ is a branch through $\TTT$, and the above assumptions on $\TTT$ imply that $b$ does not have a maximal element. Set $\lambda=\otp{b,\leq_\TTT}\leq\theta$.

 \begin{claim*}
  $b\notin\VV[G]$. 
 \end{claim*}
 
 \begin{proof}[Proof of the Claim]
  Assume, towards a contradiction, that $b\in\VV[G]$. Since $\TTT$ is a $\theta$-Aronszajn tree, we know that $\lambda\in\Lim\cap\theta$. This implies that $t$ extends every element of the branch $j_*(b)$ through the tree $j_*(\TTT)$, and therefore $j_*(b)$ is equal to the set of all predecessors of some node in the level $j_*(\TTT)(j(\lambda))$. By elementarity, there is a node $u$ in $\TTT(\lambda)$ with the property that $b$ consists of all predecessors of $u$ in $\TTT$. But then, the above assumptions on $\TTT$ imply that $j_*(u)\leq_{j(\TTT)}t$, and hence that $u\in b$, a contradiction.
 \end{proof}
 
 Since $b\not\in\VV[G]$ and $\Col{\omega_1}{[\kappa,j(\kappa))}$ is $\sigma$-closed in $\VV[G]$, we thus know that $\cof{\lambda}^{\VV[G]}>\omega$. 
 This shows that, in $\VV[G]$, there is a $\sigma$-closed notion of forcing that adds a new  branch of uncountable cofinality through $\TTT$. In this situation, standard arguments show that $\TTT$ contains a Cantor subtree in $\VV[G]$. These computations show that (ii) holds and, by applying the arguments used in the last part of the proof of Theorem \ref{theorem:ConsNoCantorSubtreeNotWC}, we  know that this also yields (i). 
\end{proof}

The above arguments leave open the possibility that Statement (ii) in Theorem \ref{theorem:LowerBoundConsNonMahlo} provably fails for inaccessible non-Mahlo cardinals, and therefore motivate the following question, asking whether the Mahloness of inaccessible cardinals can be characterized by the existence of Cantor subtrees of special Aronszajn trees in collapse extensions.

\begin{question}
 Is the existence of an inaccessible non-Mahlo cardinal $\theta$ with $$\mathbbm{1}_{\Col{\omega_1}{{<}\theta}}\Vdash\anf{\textit{Every special $\omega_2$-Aronszajn tree contains a Cantor subtree}}$$ consistent with the axioms of $\ZFC$?
 \end{question}

\bibliographystyle{plain}
\bibliography{references}

\end{document}